\documentclass[10pt, 3p]{article}

\usepackage{amsmath,amssymb,latexsym}
\usepackage{amsthm}
\usepackage{array,graphicx,mathrsfs}
\usepackage{hyperref,stmaryrd,xcolor} 

\usepackage[french,english]{babel}
\usepackage[latin1]{inputenc}
\usepackage[T1]{fontenc}

\newtheorem{theorem}{Theorem}[section]
\newtheorem{proposition}[theorem]{Proposition}
\newtheorem{lemma}[theorem]{Lemma}
\newtheorem{remark}[theorem]{Remark}

\numberwithin{equation}{section}
\newcommand\eh{\mathrm{e}_h}
\newcommand\rh{\mathrm{r}_h}
\newcommand\trh{\mathrm{\tilde r}_h}
\newcommand\rhB{\mathrm{r}_{h, \partial\Omega}}
\newcommand{\ds}{\displaystyle}
\newcommand{\N}{\mathbb{N}}
\newcommand{\Z}{\mathbb{Z}}
\newcommand{\C}{\mathbb{C}}
\newcommand{\R}{\mathbb{R}}
\newcommand{\Ocal}{\mathcal{O}}
\newcommand{\tn}{\textnormal}

\newcommand{\norm}[1]{\left\Vert#1\right\Vert}

\begin{document}
 %
 
 
 \title{Stability of an inverse problem for the discrete wave equation and convergence results
\footnote{Partially supported by the Agence Nationale de la Recherche (ANR, France),  Project CISIFS number NT09-437023, Fondecyt-CONICYT 1110290 grant
and the University Paul Sabatier (Toulouse 3), AO PICAN. This work was started and partially done while A. O. visited the University Paul Sabatier.}}

\author{Lucie Baudouin$^{1,2,}$\footnote{e-mail: {\tt baudouin@laas.fr}}\\
{\it\footnotesize $^{1}$ CNRS, LAAS, 7 avenue du colonel Roche, F-31400 Toulouse, France,}\\
{\it\footnotesize $^{2}$ Univ. de Toulouse, LAAS, F-31400 Toulouse, France.}\\
Sylvain Ervedoza$^{3,4,}$\footnote{e-mail: {\tt ervedoza@math.univ-toulouse.fr}}\\
{\it\footnotesize $^{3}$ CNRS, Institut de Math\'ematiques de Toulouse UMR 5219, F-31400 Toulouse, France,}\\
{\it\footnotesize $^{4}$ Univ. de Toulouse, IMT, F-31400 Toulouse, France.
}\\
Axel Osses$^{5,}$\footnote{e-mail: {\tt axosses@dim.uchile.cl}}\\
{\it\footnotesize $^{5}$ Departamento de Ingenier\'ia Matem\'atica and Centro de Modelamiento Matem\'atico}\\
{\it\footnotesize (UMI 2807 CNRS), FCFM Universidad de Chile, Casilla 170/3-Correo 3, Santiago, Chile.
}}

\maketitle
%
%
%
%
%
%

\begin{abstract}
Using uniform global Carleman estimates for  semi-discrete elliptic and hyperbolic equations, 
we study Lipschitz and logarithmic stability for the inverse problem of recovering a potential in a semi-discrete wave equation,
discretized by finite differences in a 2-d uniform mesh, from boundary or internal measurements. The discrete stability results, 
when compared with their continuous counterparts, include new terms depending on the discretization parameter $h$.
From these stability results, we design a numerical method to compute convergent approximations of the continuous potential.\\

\noindent{\bf Résumé}\\
A partir d'inégalités de Carleman pour des équations aux dérivées partielles dicrétisées elliptiques et hyperboliques, nous étudions la stabilité Lipschitz et logarithmique du problème inverse de détermination du potentiel dans une équation des ondes semi-discrétisée, par un schéma aux différences finies sur un maillage 2-d uniforme, à partir de mesures internes ou frontières.
Quand ils sont comparés avec leur contrepartie continue, les résultats de stabilité dans le cadre discret contiennent de nouveaux termes dépendants du pas $h$ du maillage utilisé. C'est à partir de ces résultats que nous décrivons une méthode numérique de  calcul d'approximations convergentes du potentiel continu.
\end{abstract}

%

 
 \section{Introduction}\label{Sec-Intro}
%
The goal of this article is to study the convergence of an inverse problem for the wave equation, which consists in recovering a potential through the knowledge of the flux of the solution on a part of the boundary. This article follows the previous work \cite{BaudouinErvedoza11} on that precise topic in the 1-d case.
%
\subsection{The continuous inverse problem}\label{SubSec-ContinuousIP}

\noindent{\bf Setting.} We will first present the main features of the continuous inverse problem we consider in this article. Let $\Omega$ be a smooth bounded domain of $\R^d$, and for $T>0$, consider the wave equation: 
\begin{equation}
	\label{WaveEq-Q}
		\left\{\begin{array}{ll}
			\partial_{tt} y - \Delta y+ q y= f,  &\hbox{ in }  (0,T)\times \Omega,
				\\
			y  = f_{\partial}, & \hbox{ on } (0,T) \times \partial \Omega,
				\\
			y(0,\cdot)= y^0, \quad \partial_t y(0,\cdot)= y^1, & \hbox{ in }  \Omega.
		\end{array}
		\right.
\end{equation}
Here, $y= y(t,x)$ is the amplitude of the waves, $(y^0, y^1)$ is the initial datum, $q = q(x)$ is a potential, $f$ is a distributed source term and $f_{\partial}$ is a boundary source term.\\
In the following, we explicitly write down the dependence of the function $y$ solution of \eqref{WaveEq-Q} in terms of $q$ by denoting it $y[q]$ and similarly for the other quantities depending on $q$.\\
We assume that the initial datum $(y^0, y^1)$ and the source terms $f$ and $f_\partial$ are known. We also assume the additional knowledge of the flux 
\begin{equation}
	\label{Flux}
	\mathscr{M}[q] = \partial_\nu y[q] \quad \hbox{ on } \quad (0,T) \times \Gamma_0,
\end{equation}
where $\Gamma_0$ is a non-empty open subset of the boundary $\partial \Omega$ and $\nu$ is the unit outward normal vector on  $\partial \Omega$. Note that for this map to be well-defined, we need to give a precise functional setting: for instance, we may assume $(y^0, y^1) \in H^1(\Omega) \times L^2(\Omega)$, $f \in L^1((0,T) ; L^2(\Omega))$, $f_\partial \in H^1((0,T)\times \partial \Omega)$ and $\left.y^0\right|_{\partial \Omega} = f_\partial(t = 0)$ so that $\mathscr{M}$ is well-defined for all $q \in L^\infty(\Omega)$ and takes value in $L^2((0,T) \times \partial \Omega)$, see e.g. \cite{LasieckaLionsTriggiani}.\\
This article is about the recovering the potential $q$ from $\mathscr{M}[q]$. As usual when considering inverse problems, this topic can be decomposed into the following questions:
\begin{itemize}
	\item Uniqueness: Does the measurement $\mathscr{M}[q]$ uniquely determine the potential $q$?
	\item Stability: Given two measurements $\mathscr{M}[q^a]$ and $\mathscr{M}[q^b]$ which are close, are the corresponding potentials $q^a$ and $q^b$ close?
	\item Reconstruction: Given a measurement $\mathscr{M}[q]$, can we design an algorithm to recover the potential $q$?
\end{itemize}
Concerning the precise inverse problem we are interested in, the uniqueness result is due to \cite{BuKli81} and we shall focus on the stability properties of the inverse problem \eqref{WaveEq-Q}. The question of stability has attracted a lot of attention and is usually based on Carleman estimates. There are mainly two types of results:  Lipschitz stability results, see \cite{KazemiKlibanov,PuelYam96,PuelYam97,ImYamCom01,Baudouin01,ImYamIP03,BaudouinDeBuhanErvedoza,StefUhlmann09}, provided the observation is done on a sufficiently large part of the boundary and the time is large enough, or logarithmic stability results \cite{BellassouedIP04,BellaYam06} when the observation set does not satisfy any geometric requirement. We also mention the works \cite{BellaChoulli-JMPA09,dBO-IP10} for logarithmic stability of inverse problems for other related equations.

Below we present more precisely these two type of results, since our main goal will be to discuss discrete counterparts in these two cases.\\

\smallskip
\noindent{\bf Lipschitz stability results under the Gamma-conditions.}  Getting Lipschitz sta\-bi\-lity results for the continuous inverse problem usually requires the following assumptions, originally due to \cite{Ho}. We say that the triplet $(\Omega,\Gamma,T)$ satisfy the Gamma-conditions (see \cite{Lions}) if
\begin{itemize}
	\item $(\Omega,\Gamma)$ satisfies the geometric condition:
	\begin{equation}
		\label{Gamma-Condition}
		\exists x_0 \in \R^N \setminus \overline{\Omega}, \quad \{ x \in \partial \Omega, \hbox{ s.t. } (x-x_0)\cdot \nu(x) \geq 0 \} \subset \Gamma, 
	\end{equation}
	\item $T$ satisfies the lower bound:
	\begin{equation}
		\label{Time-Condition}
		T > \sup_{x \in \Omega} | x - x_0|.
	\end{equation}
\end{itemize}
In \cite{Baudouin01}, following the works \cite{ImYamIP01,Im02}, the next stability result was proved:
\begin{theorem}[\cite{Baudouin01}]
	\label{Thm-Lucie}
	Let $m >0$ and consider a potential $q^a \in L^\infty(\Omega)$ with 	$\norm{q^a}_{L^\infty(\Omega)} \leq m$,
	and assume for some $K>0$ the regularity condition
	\begin{equation}
		\label{SmoothnessCond}
		y[q^a] \in H^1(0,T; L^\infty(\Omega)) \quad \hbox{with}\quad \norm{y[q^a]}_{H^1(0,T; L^\infty(\Omega))} \leq K,
	\end{equation}
	where $y[q^a]$ denotes the solution of \eqref{WaveEq-Q} with potential $q^a$.
	Let us further assume that $(\Omega,\Gamma_0,T)$ satisfies the Gamma-conditions \eqref{Gamma-Condition}--\eqref{Time-Condition} 
	and the following positivity condition
	\begin{equation}
		\label{PositivityCond}
		\exists \alpha_0 > 0, \quad \inf_{x\in\Omega} |y^0(x) | \geq \alpha_0.
	\end{equation}
	Then there exists a constant $C >0$ depending on $m,\, K$ and $\alpha_0$ such that for all $q^b \in L^\infty(\Omega)$ satisfying $\norm{q^b}_{L^\infty(\Omega)} \leq m$, we have $\mathscr{M}[q^a]- \mathscr{M}[q^b] \in H^1(0,T; L^2(\Gamma_0))$ and 
	\begin{equation}
		\label{Stab-Lucie}
		\frac{1}{C} \norm{q^a - q^b}_{L^2(\Omega)} \leq \norm{\mathscr{M}[q^a]- \mathscr{M}[q^b]}_{H^1(0,T; L^2(\Gamma_0))} \leq C \norm{q^a - q^b}_{L^2(\Omega)}.
	\end{equation}
	Besides, if $\omega$ is a neighborhood of $\Gamma_0$, i.e. for some $\delta >0$, 
	$ \{ x \in \Omega, \, d(x, \Gamma_0) < \delta\} \subset \omega$,
	we also have $\partial_t y[q^a] - \partial_t y[q^b] \in H^1((0,T) \times \omega)$ and
	\begin{eqnarray}
	\label{Stab-Lucie-Interne}
		\frac{1}{C} \norm{q^a - q^b}_{L^2(\Omega)} \leq \norm{\partial_t y[q^a] - \partial_t y[q^b]}_{H^1((0,T) \times \omega)} \leq C \norm{q^a - q^b}_{L^2(\Omega)}.
	\end{eqnarray}
\end{theorem}
\begin{remark}
	Note that in Theorem \ref{Thm-Lucie}, we do not give assumptions on the smoothness of the data $y^0, \, y^1, \, f, \, f_\partial$ directly. They rather appear through the bound $K$ in \eqref{SmoothnessCond} in an intricate way. 
	Also note that estimate \eqref{Stab-Lucie-Interne} is not written in \cite{Baudouin01}, but the proof of \eqref{Stab-Lucie-Interne} follows line to line the one of \eqref{Stab-Lucie}.\\
\end{remark}

\smallskip
\noindent{\bf Logarithmic stability results under weak geometric condition.} Let us now explain what can be done when the geometric part  \eqref{Gamma-Condition} of the Gamma conditions is not satisfied. In this case, to our knowledge, the best result available is due to \cite{BellassouedIP04}. Below, we state a slightly improved version of it:
\begin{theorem}[\cite{BellassouedIP04}, revisited]
	\label{ThmBellassoued}
	Assume that there exist an open subset $\Gamma_1 \subset \partial \Omega$ of the boundary $\partial \Omega$ and an open subset $\Ocal$ of $\Omega$ such that:
	\begin{itemize}
		\item $\Gamma_0 \subset \Gamma_1$ and $(\Omega$, $\Gamma_1)$ satisfies the condition \eqref{Gamma-Condition}; 
		\item $\Ocal$ contains a neighborhood of $\Gamma_1$ in $\Omega$, i.e. for some $\delta >0$, 
		\begin{equation}
			\label{Ocal-Condition} 
			 \{ x \in \Omega, d(x, \Gamma_1) < {\delta} \} \subset \Ocal.
		\end{equation}
	\end{itemize}
	Let $q^a $ be a potential lying in the class $\Lambda(Q,m)$ defined for $Q \in L^{ \infty}(\Ocal)$ and $m>0$ by
	\begin{equation}
		\label{Class-Pot-Bellassoued}
		\Lambda (Q,m) = \{ q\in L^\infty(\Omega),   \hbox{ s.t. } q|_{\Ocal} = Q \ \hbox{ and } \ \norm{q}_{L^{\infty}(\Omega)} \leq m \}.
	\end{equation}
	Let $y^0 \in H^1(\Omega)$ satisfying the positivity condition \eqref{PositivityCond} and assume that $y[q^a]$ satisfies the regularity condition
	\begin{equation}
		\label{Ass-Smoothness-Bel}
		y[q^a] \in H^1(0,T; L^\infty(\Omega)) \cap W^{2,1}(0,T; L^2(\Omega)). 
	\end{equation}
	Let $\alpha>0$ and $M>0$. 
	Then there exists $C>0$ such that for $T>0$ large enough, for all $q^b \in \Lambda (Q, m)$ satisfying
	\begin{equation}
		\label{Error-W-1-2-+}
		q^a - q^b \in H^1_0(\Omega)  \hbox{ and } \norm{q^a- q^b}_{H^1_0(\Omega)} \leq M,
	\end{equation}
	we have $\mathscr{M}[q^a] - \mathscr{M}[q^b] \in H^1(0,T; L^2(\Gamma_0))$ and
	\begin{equation}
		\label{Stability-Bellassoued}
		\norm{q^a - q^b }_{L^2(\Omega)}  \leq C \left[ \log\left( 2+ \frac{C}{\norm{\mathscr{M}[q^a] - \mathscr{M}[q^b]}_{H^1(0,T; L^2(\Gamma_0))} }\right)\right]^{-\frac{1}{1+\alpha}}.
	\end{equation}
	Besides, the constant $C$ depends on $m$ in \eqref{Class-Pot-Bellassoued}, $M$ in \eqref{Error-W-1-2-+}, $\alpha_0$ in \eqref{PositivityCond}, a priori bounds on $\norm{y^0}_{H^1(\Omega)} + \norm{y[q^a]}_{H^1(0,T; L^\infty(\Omega)) \cap W^{2,1}(0,T; L^2(\Omega))}$ and the geometric setting $(\Gamma_0, \Gamma_1, \Ocal, \Omega)$.
\end{theorem}

To be more precise, \cite{BellassouedIP04} states the previous result with $\alpha = 1$ and under slightly stronger geometric and regularity conditions. Since Theorem \ref{ThmBellassoued} states a slightly better result than the one in \cite{BellassouedIP04}, we will prove it in Section \ref{AEC}. Similarly as in \cite{BellassouedIP04}, we will work on the difference $y[q^a]- y[q^b]$ and use the Fourier-Bros-Iagoniltzer transform which links solutions of the wave equation with solutions of an elliptic PDE, but instead of considering the usual Gaussian transform as in \cite{BellassouedIP04} (see also \cite{Robbiano91,Robbiano95}), we will consider the one used in \cite{LebRob97} (see also \cite{BellaYam06,Phung09}). We will thus be led to prove a quantified unique continuation result for an elliptic PDE, which we derive using a classical Carleman estimate (\cite{Hormander-III}). Nevertheless, we will do it in a somewhat different way as the one in \cite{Robbiano95,Phung09} by constructing one global weight which allows to prove Theorem \ref{ThmBellassoued} without the use of iterated three spheres inequalities. The proof of Theorem \ref{ThmBellassoued} will then be completed by the use of  the stability estimates \eqref{Stab-Lucie-Interne}.
\\

\smallskip
\noindent{\bf Objectives.} Our goal is to derive counterparts of Theorem \ref{Thm-Lucie} and Theorem \ref{ThmBellassoued} for the finite-difference space approximations of the wave equation discretized on a uniform mesh.
In order to give precise statements, we need to introduce several notations listed in the next section. For simplicity of notations, we make the choice of focusing on the unit square in the $2$-d case 
\begin{equation}
	\label{Omega}
	\Omega = (0,1)^2,
\end{equation}
though our methodology applies similarly in the case of the $d$-dimensional domains of rectangular form $\Omega = \Pi_{j = 1}^d [a_j, b_j]$ (still discretized on a uniform mesh).
Note that, even if we stated Theorems~\ref{Thm-Lucie} and \ref{ThmBellassoued} for smooth bounded domains, both Theorems also hold in the case of a domain $\Omega = (0,1)^2$.

%
%
\subsection{Some notations in the discrete framework}\label{SubSec-Notations}
%
Here, we introduce the notations corresponding to the case of a finite-difference discretization of the wave equation on a uniform mesh. 
Let $N \in \mathbb{N}$ be the number of interior points in each direction, and $h = 1/(N+1)$ the mesh size.
All the notations introduced in the discrete setting will be indexed by the parameter $h>0$ to avoid confusion with the continuous case.\\

\noindent{\bf Discrete domains.}
We introduce the following  (see also an illustration in Figure~\ref{Fig-Notations}):
\begin{equation}\label{esph}
	\begin{array}{l}	
	\Omega_h =\{h,2h,\dots, Nh \}^2, \quad \overline{\Omega_h} = \{0, h,2h,\dots, Nh,1 \}^2,
	\\
	 \partial \Omega_h = \left(\left( \{0\} \cup \{1\}\right) \times \{h, \dots, Nh \}\right) \cup \left( \{h, \dots, Nh\} \times \left( \{0\} \cup \{1\}\right) \right),
	 \\
	 \Gamma_{h,1}^- =  \{0\} \times \{h, \dots, Nh \}, \qquad  \Gamma_{h,2}^- =    \{h, \dots, Nh \} \times \{0\},
	  \\
	 \Gamma_{h,1}^+ =  \{1\} \times \{h, \dots, Nh \}, \qquad  \Gamma_{h,2}^+ =    \{h, \dots, Nh \} \times \{1\},
	  \\	
	 \Gamma_{h}^- =\Gamma_{h,1}^-\cup\Gamma_{h,2}^-,\qquad
	 \Gamma_{h}^+ =\Gamma_{h,1}^+\cup\Gamma_{h,2}^+,\qquad
	 \partial \Omega_h =  \Gamma_{h}^- \cup  \Gamma_{h}^+,
	 \\
	 \Omega_{h,1}^- =\Omega_h\cup\Gamma_{h,1}^-,\qquad
	 \Omega_{h,2}^- =\Omega_h\cup\Gamma_{h,2}^-,\qquad 
	 \Omega_{h}^- =\Omega_{h,1}^-\cap\Omega_{h,2}^-.
	\end{array}
\end{equation}
Note that this naturally introduces two representations of the discrete set $\overline{\Omega_h}$. We will use alternatively $x_h \in \overline{\Omega_h}$ or $(i,j) \in \llbracket 0,N+1\rrbracket^2$ (where $\llbracket a,b\rrbracket = [a,b] \cap \mathbb N$) to denote the point $x_h = (ih, jh)$, the advantage of the first writing being its consistency with the continuous model.\\
\begin{figure}
	\begin{center}
		\includegraphics[width=7cm]{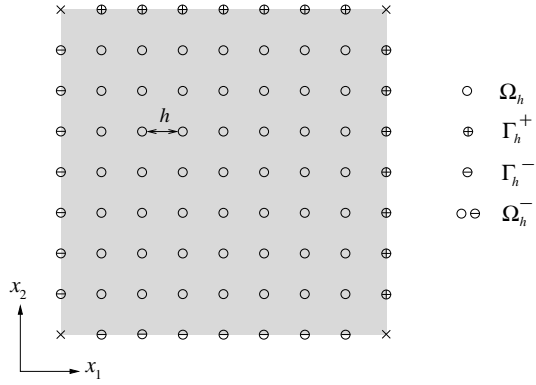}
	\end{center}
	\caption{Main discrete notations in $\Omega=(0,1)\times(0,1)$.}
	\label{Fig-Notations}
\end{figure}

\noindent{\bf Discrete integrals.}
By analogy with the continuous case, if we denote by 
$f_h = (f(x_h))_{x_h \in\Omega_h}$, respectively $f_h = (f(x_h))_{x_h \in\Omega_{h,1}^-}$, $f_h = (f(x_h))_{x_h \in\Omega_{h,2}^-}$, a discrete function, we will use the following shortcuts:
\begin{equation}\label{intf}
	 \ds\int_{\Omega_h} f_h =  \ds\int_{\Omega_h} f_{i,j} = h^2 \ds\sum_{i,j=1}^{N} f_{i,j} \, ;
	\,
	 \ds\int_{\Omega_{h,1}^-} f_h  = h^2 \ds\sum_{i=0}^{N}\sum_{j=1}^{N} f_{i,j} \, ;
	 \,
	\ds\int_{\Omega_{h,2}^-} f_h  = h^2 \ds\sum_{i=1}^{N}\sum_{j=0}^{N} f_{i,j}.
\end{equation}
One should notice that if these symbols are applied to continuous functions or products of discrete and continuous functions, they have to be understood as the corresponding Riemann sums. 

When considering integrals on the boundary $\partial \Omega_h$, we use the natural scale for the boundary and we define, for $f_h$ a discrete function on $\partial \Omega_h$,
\begin{equation}
	\label{int-f-bord}
	 \ds\int_{\partial \Omega_h} f_h = h \ds\sum_{x_h \in \partial \Omega_h} f(x_h). 
\end{equation}

\smallskip
\noindent{\bf Subsets.}
In several places, we will consider open subsets $\Ocal, \omega \subset \Omega$ and we then note $\Ocal_h = \Ocal \cap \Omega_h$, $\overline{\Ocal_h} = \{ x \in \overline{\Omega}, \, d(x, \Ocal) \leq  h \} \cap \overline{\Omega_h}$, $\Ocal_{h,k}^- = \{x \in \overline{\Omega},\, \exists \epsilon \in [0,h], \, x + \epsilon e_k \in \Ocal\}\cap \Omega_{h,k}^-$, and similarly for the sets $\omega_h$, $\overline{\omega_h}$ and $\omega_{h,k}^-$ (notice that these sets are always non-empty for $h$ small enough). Integrals on these discrete approximations of open subsets of $\Omega$ are given for $f_h$ discrete functions on $\overline{\Ocal_h}$ as follows:
\begin{equation}\label{intf-subset}
	 \ds\int_{\Ocal_h} f_h =  \ds\int_{\Omega_h} f_h {\bf 1}_{\Ocal_h}, 
	\quad
	 \ds\int_{\Ocal_{h,k}^-} f_h  = \ds\int_{\Omega_{h,k}^-} f_h {\bf 1}_{\Ocal_{h,k}^-},
\end{equation}
and similarly for the integrals on $\omega_h$, $\omega_{h,k}^-$. 
\\
When considering open subsets $\Gamma$ of the boundary $\partial \Omega$, we will similarly set $\Gamma_{h} = \Gamma\cap \partial \Omega_h$, and  the integrals on these discrete approximations of subsets of the boundary will be given by 
$$
	  \ds\int_{\Gamma_h} f_h  = \ds\int_{\partial \Omega_h} f_h {\bf 1}_{\Gamma_h}.
$$

\smallskip
\noindent{\bf Discrete $L^p$-spaces.}
We also define in a natural way a discrete version of the $L^p(\Omega )$-norms as follows: for $p\in[1, \infty)$, we introduce $L^p_h(\Omega_h )$ (respectively $L^p_h(\Omega_{h,1}^-)$) the space of discrete functions $f_h = (f_{i,j})_{i,j \in \llbracket 1,N\rrbracket}$, 
(respectively $i \in \llbracket 0,N\rrbracket, \,j \in \llbracket 1,N\rrbracket$) endowed with the norms
\begin{equation}
	\label{NormLp}
	\norm{f_h}_{L^p_h(\Omega_{h})}^p = 	\ds\int_{\Omega_{h}} |f_h|^p  
	\quad \hbox{\Big(resp. } \norm{f_h}_{L^p_h(\Omega_{h,1}^-)}^p = 	\ds\int_{\Omega_{h,1}^-} |f_h|^p \,\hbox{\Big)},
\end{equation}
and, for $p = \infty$, 
$
	\norm{f_h}_{L^\infty_h(\Omega_{h})} = 	\sup_{ i,j \in \llbracket 1,N\rrbracket} |f_{i,j}| 
	$, $\hbox{(resp. } \norm{f_h}_{L^\infty_h(\Omega_{h,1}^-)} = 	\sup_{i \in \llbracket 0,N\rrbracket;j \in \llbracket 1,N\rrbracket} |f_{i,j}| \hbox{)}.
$
\\
We define the spaces $L^p_h(\Omega_{h,2}^-)$, $L^p_h(\Ocal_h)$ and $L^p_h(\omega_h)$ for open subsets $\Ocal, \omega \subset \Omega$ in a similar way.
We also define discrete norms on parts of the boundary: if $\Gamma$ is an open subset of $\partial \Omega$, 
the space $L^p_h(\Gamma_h)$, ($p \in [1, \infty)$) is the set of discrete functions $f_h$ defined on $\Gamma_h$ endowed with the norm 
$$
	\norm{f_h}_{L^p_h(\Gamma_h)}^p = \ds\int_{\Gamma_{h}} |f_h|^p.
$$

\smallskip
\noindent{\bf Discrete operators.}
We approximate the Laplace operator by the $5$-points finite-difference approximation: $\forall (i,j) \in \llbracket 1, N\rrbracket^2$,
\begin{equation}
	\label{Delta-h}
	(\Delta_h v_h)_{i,j} = \frac{1}{h^2} \left(v_{i+1,j}+ v_{i,j+1} + v_{i-1,j} + v_{i,j}- 4 v_{i,j} \right).
\end{equation}
Besides the discrete Laplacian $\Delta_h$, let us also introduce the following discrete operators:
\begin{gather*}
	(m_{h,1} v_h)_{i,j} =  \dfrac{v_{i+1,j} + 2 v_{i,j}+v_{i-1,j}}{4} ~; \quad
	(m_{h,2} v_h)_{i,j} =  \dfrac{v_{i,j+1} + 2 v_{i,j}+v_{i,j-1}}{4} ~; 
	 \\
	 (m^+_{h,1} v_h)_{i,j}  = (m_{h,1}^- v_h )_{i+1,j} = 
	 \dfrac{v_{i+1,j} + v_{i,j}}{2}   ~;
	\quad
	 (m^+_{h,2} v_h)_{i,j}  = (m_{h,1}^- v_h )_{i+1,j} = 
	 \dfrac{v_{i,j} + v_{i,j+1}}{2}   ~;
	\\
	(\partial_{h,1} v_h)_{i,j} = \dfrac{v_{i+1,j} -v_{i-1,j}}{2h}~;\quad
	(\partial_{h,2} v_h)_{i,j} = \dfrac{v_{i,j+1} -v_{i,j-1}}{2h}~; 
	\quad 
	\nabla_h = \left( \begin{array}{c} \partial_{h,1} \\ \partial_{h,2} \end{array} \right)~; 
	\\
	(\partial^+_{h,1} v_h)_{i,j} = (\partial^-_{h,1} v_h)_{i+1,j} = \dfrac{v_{i+1,j} -v_{i,j}}{h}  ~;\quad
	(\partial^+_{h,2} v_h)_{i,j} = (\partial^-_{h,2} v_h)_{i,j+1} = \dfrac{v_{i,j+1} -v_{i,j}}{h}  ~;
	\\
	(\Delta_{h,1} v_h)_{i,j}  = \dfrac{v_{i+1,j} - 2 v_{i,j}+v_{i-1,j}}{h^2}~;\quad
	(\Delta_{h,2} v_h)_{i,j}  = \dfrac{v_{i,j+1} - 2 v_{i,j}+v_{i,j-1}}{h^2}.
\end{gather*}
We finally introduce the following semi-discrete wave operator:
$$
	\Box_h = \partial_{tt} - \Delta_h = \partial_{tt} - \Delta_{h,1} - \Delta_{h,2}. 
$$

\smallskip
\noindent{\bf Spaces of more regularity.} We will use the space $H^1_h(\Omega_h)$ of discrete functions $f_h$ defined on $\overline{\Omega_h}$ endowed with the norm
$$
	\norm{f_h}_{H^1_h(\Omega_h)}^2 =  \norm{f_h}_{L^2_h(\overline{\Omega_h})}^2 + \sum_{k = 1,2} \norm{\partial_{h,k}^+ f_h}^2_{L^2_h(\Omega_{h,k}^-)}.
$$
We also denote $H^1_{0,h}(\Omega_h)$ the set of functions $f_h$  defined on $\overline{\Omega_h}$  and vanishing on $\partial \Omega_h$ endowed with the above norm.

Note down that $H^1_h(\Omega_h)$ and $H^1_{0,h}(\Omega_h)$ denote spaces of functions defined on $\overline{\Omega_h}$. We decided to slightly abuse the notations by denoting them that way, since the topology of these spaces is strong enough to define the trace operators.

Similarly, when $\omega$ is a non-empty open subset of $\Omega$, we denote by $H^1_h(\omega_h)$ the set of discrete functions $f_h$ defined in $\overline{\omega_h}$ endowed with the norm 
$$
	\norm{f_h}_{H^1_h(\omega_h)}^2 =  \norm{f_h}_{L^2_h(\overline{\omega_h})}^2 + \sum_{k = 1,2} \norm{\partial_{h,k}^+ f_h}^2_{L^2_h(\omega_{h,k}^-)}.
$$

We finally introduce $H^2_h(\Omega_h)$ the set of discrete functions $f_h$ defined on $\overline{\Omega_h}$ endowed with the norm 
$$
	\norm{f_h}_{H^2_h(\Omega_h)}^2 = \norm{f_h}_{H^1_h(\Omega_h)}^2 +  \norm{\Delta_{h,1} f_h}_{L^2_h({\Omega_h})}^2+  \norm{\Delta_{h,2} f_h}_{L^2_h({\Omega_h})}^2+ \norm{\partial_{h,1}^+ \partial_{h,2}^+ f_h}_{L^2(\Omega_h^-)}^2.
$$
Besides, with an abuse of notations, we will often denote $L^2(0,T;H^1_h(\Omega_h)) \cap H^1(0,T; L^2_h(\Omega_h))$ by $H^1_h((0,T)\times \Omega_h)$  and the space $H^2(0,T;L^2_h(\Omega_h)) \cap H^1(0,T; H^1_h(\Omega_h))\cap L^2(0,T;H^2_h(\Omega_h))$ by $H^2_h((0,T)\times \Omega_h)$.\\

\smallskip
\noindent{\bf Extension and restriction operators.}
Finally, we shall explain how to compare discrete functions with continuous ones. In order to do so, we introduce extension and restriction operators.

The first one extends discrete functions by continuous piecewise affine functions and is denoted by $\eh$. To be more precise, if $f_h$ is a discrete function $(f_{i,j})_{i,j \in \llbracket 0,N+1\rrbracket }$, 
the extension $\eh (f_h)$ is defined on $[0,1]^2$ for $(x_1,x_2) \in [ih,(i+1)h]\times [jh, (j+1)h]$ by
\begin{multline}\label{eh}
	\eh(f_h)(x_1,x_2) = 
			\ds \left(1- \frac{x_1-ih}{h} \right)\left(1- \frac{x_2-jh}{h} \right)  f_{i,j}
			+
			 \left( \frac{x_1-ih}{h} \right)\left(1- \frac{x_2-jh}{h} \right) f_{i+1,j}
			\\
			\ds +
			 \left(1- \frac{x_1-ih}{h} \right)\left( \frac{x_2-jh}{h} \right) f_{i,j+1}
			+ \left(\frac{x_1-ih}{h} \right)\left(\frac{x_2-jh}{h} \right)f_{i+1,j+1}  .
\end{multline}
This extension presents the advantage of being naturally in $H^1(\Omega)$.
The second extension operator is the piecewise constant extension $\eh^0(f_h)$, defined for discrete functions $f_h = (f_{i,j})_{i,j \in \llbracket 1,N\rrbracket}$ by
\begin{equation}\label{eh0}
	\begin{array}{ll}
	\eh^0 (f_h) = f_{i,j}  \quad &\hbox{ on } [(i-1/2)h, (i+1/2)h[\times  [(j-1/2)h, (j+1/2)h[, \quad i,j \in \llbracket 1,N\rrbracket,
	\smallskip\\
	\eh^0 (f_h) = 0  \quad &\hbox{ elsewhere} .
	\end{array}
\end{equation}
This one is natural when dealing with functions lying in $L^2({\Omega})$ as $\norm{ \eh^0 (f_h)}_{L^2(\Omega)} = \norm{ f_h}_{L^2_h(\Omega_h)}$.
Also note that easy (but tedious) computations show that $\eh(f_h)$ converge to $f$ in $L^2(\Omega)$ if and only if $\eh^0(f_h)$ converge to $f$ in $L^2(\Omega)$.

We finally introduce restriction operators $\rh$, $\trh$ and $\rhB$ where $\rh$ is defined for continuous function $f \in C(\overline{\Omega})$ by 
$$
	\rh(f) = f_h \ \tn{ given by } f_{i,j} = f(ih,jh), \quad \forall i,j \in \llbracket 1, N \rrbracket,
$$
 $\trh$ for functions $f \in L^2({\Omega})$ by 
$$
	\trh(f) = f_h \ \tn{ given by } 
	\left\{
		\begin{array}{l}
		\ds f_{i,j} = \frac{1}{h^2} \underset{\substack{ |x_1 - ih| \leq h/2 \\ |x_2 - jh| \leq h/2 }}\iint f(x_1,x_2) \, dx_1 dx_2, \quad \forall i,j \in \llbracket 1, N \rrbracket,
		\smallskip\\
		\ds 
		f_{i,j} = \frac{1}{2 h^2} \underset{\substack{ |x_1 - ih| \leq h/2 \\ |x_2 - jh| \leq h/2 \\ (x_1, x_2) \in \Omega}}\iint f(x_1,x_2) \, dx_1 dx_2, \quad \forall x_h = (ih,jh) \in \partial \Omega_h
		\end{array}
	\right.
$$
and $\rhB$ for functions $f_\partial \in L^2(\partial \Omega)$ by 
$$
	\rhB(f_\partial)(x_h) = \frac{1}{h}
	\underset{\substack{|x - x_h| \leq h/2,\\x \in \partial \Omega}}\int f_\partial(x) d\sigma\ \tn{ for } x_h \in \partial \Omega_h.
$$

%
\subsection{The semi-discrete inverse problem and main results}\label{SubSec-MainResults}
%
We discretize the usual 2-d wave equation on $\Omega = (0,1)^2$ using the finite difference method on a uniform mesh of mesh size $h>0$. Using the above notations, this leads to the following equation:
 \begin{equation}
	\label{DisWaveEq}
	\left\{
		\begin{array}{ll}
			\partial_{tt} y_h - \Delta_h y_h + q_h y_h= f_h & \hbox{ in } (0,T) \times \Omega_h, 
			\\
			y_h = f_{\partial,h} & \hbox{ on } (0, T) \times \partial \Omega_h, 
			\\
			y_h(0)= y_{h}^0, \quad\partial_t y_h(0)  = y_{h}^1 & \hbox{ in } \Omega_h.
		\end{array}
	\right.
\end{equation}  
Here, $y_h(t,x_h)$ is an approximation of the solution $y$ of \eqref{WaveEq-Q} in $(t,x_h)$, $\Delta_h$ approximates the Laplace operator and we assume that $(y_{h}^0,y_{h}^1)$ are the initial sampled data $(y^0, y^1)$ at $x_h$, and the source terms $f_{\partial,h}\in L^2(0,T;L^2_h(\partial\Omega_h))$ and $ f_h \in L^1(0,T; L^2_h(\Omega_h))$ are discrete approximations of the boundary and source terms $f_\partial$ and $f$.

Our main goal is to establish the {\it convergence} of the discrete inverse problems for \eqref{DisWaveEq} toward the continuous one for \eqref{WaveEq-Q} in the sense developed in \cite{BaudouinErvedoza11}. Let us rapidly present what kind of results should be expected. 

The natural idea to compute an approximation of the potential $q$ in \eqref{WaveEq-Q} from the boundary measurement $\mathscr{M}[q]$ is to try to find a discrete potential $q_h$ such that the measurement 
\begin{equation}
	\label{Flux-h}
	\mathscr{M}_h [q_h] = 
	 \partial_{\nu} \eh(y_h [q_h]) \quad \hbox{ on } \quad (0,T) \times \Gamma_0 
\end{equation}
where $y_h[q_h]$ is the solution of \eqref{DisWaveEq}, and $\eh$ is the piecewise affine extension defined in \eqref{eh}, approximates $\mathscr{M}[q]$ defined in \eqref{Flux}. We are thus asking the following:
\begin{quote}
 if one finds a sequence $q_h$ of discrete potentials such that $\mathscr{M}_h[q_h]$ converges towards $\mathscr{M}[q]$ as $h\to 0$ (in a suitable topology), can we guarantee that the sequence $q_h$ converges (in a suitable topology) towards $q$ ?
\end{quote}
As it is classical in numerical analysis - this is the so-called Lax theorem for the convergence of numerical schemes - such result can be achieved using the consistency and the uniform stability of the problem. In our context, even if the consistency requires some work, the stability issue is much more intricate since even in the continuous case it is based on Carleman estimates. Here, \textit{stability} refers to the possibility of getting bounds of the form
\begin{equation}
	\label{Stability?}
	\norm{\eh^0(q_h^a - q_h^b)}_{*} \leq C \norm{\mathscr{M}_h[q_h^a] - \mathscr{M}_h[q_h^b] }_{\#},
\end{equation}
where  $\eh^0$ is the piecewise constant extension defined in \eqref{eh0}, and the norms $\norm{\cdot}_{*}$ and $\norm{\cdot}_{\#}$ have to be precised, for some positive constant $C$ \emph{independent} of $h$.

As we already pointed out in \cite{BaudouinErvedoza11} in the 1-d case, a stability estimate of the form \eqref{Stability?} is far from obvious and actually, instead of getting an estimate like \eqref{Stability?}, we proposed a slightly modified observation operator $\widetilde{\mathscr{M}_h}$ for which we prove uniform stability estimates and the convergence of the inverse problem.

Hence the main difficulty in obtaining convergence results is to derive suitable stability estimates for the discrete inverse problem under consideration. We will thus state convergence results for the discrete inverse problems in the forthcoming Theorem~\ref{Thm-Cvg}, while the main part of the article focuses on the proof of stability estimates for the discrete inverse problem set on \eqref{DisWaveEq} stated hereafter in Theorems~\ref{TCWE1} and \ref{TCWElog}.

\subsubsection{Discrete stability results}

\smallskip
\noindent{\bf Discrete Lipschitz stability.} 
Since we assumed $\Omega = (0,1)^2$, the condition \eqref{Gamma-Condition} will be satisfied by a set $\Gamma_0 \subset \partial \Omega$ if and only if $\Gamma_0$ contains two consecutive edges, and in this case the time $T$ in \eqref{Time-Condition} can be taken to be any $T > \sqrt{2}$. Thus, with no loss of generality, when the Gamma-conditions \eqref{Gamma-Condition}--\eqref{Time-Condition} are satisfied, we can focus on the study of the case
\begin{equation}
	\label{Gamma-+-T-sqrt2}
	\Omega = (0,1)^2, \quad \Gamma_0 \supset \Gamma_+ = (\{1\} \times (0,1)) \cup ((0,1) \times \{1\} ), \quad T > \sqrt{2}.
\end{equation}
When the measurement is done on a part of the boundary $\Gamma_0$ satisfying the  above conditions, 
we will prove the following counterpart of Theorem \ref{Thm-Lucie}:
\begin{theorem}[Lipschitz stability under Gamma-conditions]\label{TCWE1}
	Assume that $(\Omega, \Gamma_0,T)$ satisfy the configuration \eqref{Gamma-+-T-sqrt2}.
	Let $m>0$, $K>0$, $\alpha_0>0$, and $q_h^a \in L^\infty_{h}(\Omega_h)$ with $\norm{q_h^a}_{L^\infty_h(\Omega_h)} \leq m$.
	Assume also that $y^{0}_h$ and the solution $y_h[q_h^a]$ of \eqref{DisWaveEq} with potential $q_h^a$ satisfy
	\begin{equation}
		\label{RegDiscrete}
				\inf_{\Omega_h} |y^{0}_h|\geq \alpha_0
		\quad \hbox{ and }\quad
		\norm{ y_h[q_h^a]}_{H^1 (0,T;L_h^{\infty}(\Omega_h)) } \leq K.
	\end{equation}
	Then there exists a constant $C=C(T,m, K, \alpha_0)>0$ independent of $h$ such that for all $ q_h^b\in L^\infty_{h}(\Omega_h)$ with $\norm{q_h^b}_{L^\infty_h(\Omega_h)} \leq m$, the following 	uniform stability estimate holds:
	\begin{multline}
		\norm{q_h^a - q_h^b}_{L^2_h(\Omega_h)} 
		\leq
		C \norm{\mathscr{M}_h[q_h^a] - \mathscr{M}_h[q_h^b]}_{H^1(0,T;L^2_h(\Gamma_{0,h}))} 
		\label{UniformStability-3}
		\\
		+ Ch \sum_{k=1,2} \norm{\partial_{h,k}^+\partial_{tt} y_{h}[q_h^a] - \partial_{h,k}^+\partial_{tt} y_h[q_h^b]}_{L^2(0,T;L^2_h(\Omega_{h,k}^-))}
	\end{multline}
	where $y_h[q_h^b]$ is the solution of \eqref{DisWaveEq} with potential $q_h^b$.
	\\
	Similarly, if $\omega$ is a neighborhood of $\Gamma_+$, i.e. there exists $\delta>$ such that
	\begin{equation}
		\label{Configuration-Lipschitz-distributed}
		 ((1, 1- \delta) \times (0,1)) \cup ((0,1) \times (1-\delta, 1)  ) \subset \omega,
	\end{equation}
	then there exists a constant $C=C(T,m, K, \alpha_0,\delta)>0$ independent of $h$ such that for all $ q_h^b\in L^\infty_{h}(\Omega_h)$ with $\norm{q_h^b}_{L^\infty_h(\Omega_h)} \leq m$, the following uniform stability estimate holds:
	\begin{multline}
		\norm{q_h^a - q_h^b}_{L^2_h(\Omega_h)} 
		\leq 
		C  \norm{\partial_t y_{h}[q_h^a] - \partial_t y_h[q_h^b]}_{H^1(0,T;L^2_h(\omega_h))}
		\\
		 + C \sum_{k = 1,2} \norm{\partial_{h,k}^+ \partial_t y_{h}[q_h^a] - \partial_{h,k}^+ \partial_t y_h[q_h^b]}_{L^2(0,T;L^2_h(\omega_{k,h}^-))} 
		\label{UniformStability-3Loc}
		\\
		+ Ch \sum_{k=1,2} \norm{\partial_{h,k}^+\partial_{tt} y_{h}[q_h^a] - \partial_{h,k}^+\partial_{tt} y_h[q_h^b]}_{L^2(0,T;L^2_h(\Omega_{h,k}^-))}.
	\end{multline}
\end{theorem}
When comparing Theorem \ref{TCWE1} with Theorem \ref{Thm-Lucie}, one immediately sees that estimate \eqref{UniformStability-3} is a reinforced version of \eqref{Stab-Lucie} due to the additional term
\begin{equation}
	\label{PenalizationTerm}
	Ch \sum_{k=1,2} \norm{\partial_{h,k}^+\partial_{tt} y_{h}[q_h^a] - \partial_{h,k}^+\partial_{tt} y_h[q_h^b]}_{L^2(0,T;L^2_h(\Omega_{h,k}^-))}.
\end{equation}
This was already observed in \cite{BaudouinErvedoza11} for the corresponding 1-d inverse problems, and is remanent from the fact that observability estimates for the discrete wave equations do not hold uniformly if they are not suitably penalized, see \cite{InfZua,Zua05Survey,ErvZuaCime}. Note in particular that as $h\to 0$ and under suitable convergence assumptions, this term vanishes and allows to recover the left hand side inequality of \eqref{Stab-Lucie} by passing to the limit in \eqref{UniformStability-3}. 
Theorem \ref{TCWE1} is proved in Section \ref{CSIP}. Following the proof of its continuous counterpart Theorem \ref{Thm-Lucie}, the main issue is to derive a discrete Carleman estimate for the wave operator (Theorem \ref{Thm-CarlemanDisc-Boundary}), as it was already done in \cite{BaudouinErvedoza11} in the 1-d setting. Though the proof of this discrete Carleman estimate is very close to the one in 1-d, the dimension 2 introduces new cross-terms involving discrete operators in space that require careful computations. Note however that our proof also applies in higher dimension when the domain is a cuboid discretized on uniform meshes as this  would involve similar terms. Actually, this has already been done in the context of elliptic equations, see \cite{BoyerHubertLeRousseau2}.\\

\smallskip
\noindent{\bf Discrete logarithmic stability.} 
Since we limit ourselves to the case $\Omega = (0,1)^2$, we may assume that $\Gamma_0$ is a (non-empty) subset of one edge and that the counterpart of $\Gamma_1$ appearing in Theorem \ref{ThmBellassoued} satisfying the Gamma conditions \eqref{Gamma-Condition} is formed by two consecutive edges. Due to the invariance by rotation, with no loss of generality, we may thus assume:
\begin{equation}
	\label{Config-Bellassoued}
	\Omega = (0,1)^2, \quad \Gamma_0 \subset \{1\} \times (0,1), \quad \Gamma_1 = \Gamma_+ = (\{1\} \times (0,1)) \cup ((0,1) \times \{1\} ).
\end{equation}

\begin{theorem}[Logarithmic stability under weak geometric conditions]
	\label{TCWElog}
	Assume that the triplet  $(\Omega, \Gamma_0, \Gamma_1)$ satisfy the geometric configuration \eqref{Config-Bellassoued} and the existence of an open set $\Ocal \subset \Omega$ such that 
	\begin{itemize}
		\item $\Ocal$ contains a neighborhood $\omega$ of $\Gamma_1$ in $\Omega$, i.e. such that \eqref{Configuration-Lipschitz-distributed} holds.
		\item the potential $q_h$ is known on $\partial \Omega_h$ and in $\Ocal_h$, where it takes the value $Q_h \in L^{\infty}_h(\Ocal_h)$. 
	\end{itemize}
	Let $q^a_h $ be a potential lying in the class $\Lambda_h(Q_h,m)$ defined for $Q_h \in L^{ \infty}_h(\Ocal_h)$ and $m>0$ by
	\begin{equation}
		\label{Class-Pot-Bellassoued-h}
		\Lambda_h (Q_h,m) = \{ q_h\in L^\infty_h(\Omega_h), \hbox{ s.t. } q_h|_{\Ocal_h} = Q_h \ \hbox{ and } \ \norm{q_h}_{L^{\infty}_h(\Omega_h)} \leq m \}.
	\end{equation}
	Let $\alpha_0>0,M>0$ and $\alpha>0$. Assume also that $y^0_h \in H^1_{h}(\Omega_h)$ and the solution $y_h[q_h^a]$ of \eqref{DisWaveEq} with potential $q_h^a$ satisfy the conditions
	\begin{equation}
		\label{Ass-Smoothness-Bel-h}
		 \inf_{\Omega_h} |y^{0}_h|\geq \alpha_0
	\quad \hbox{ and }\quad
		y_h[q^a] \in H^1(0,T; L^\infty_h(\Omega_h)) \cap W^{2,1}(0,T; L^2_h(\Omega_h)). 
		\end{equation}
	Then there exist $C>0$ and $h_0>0$ such that for $T>0$ large enough, for all $h \in (0,h_0)$, for all $q_h^b \in \Lambda_h (Q_h, m)$ satisfying
	\begin{equation}
		\label{Error-W-1-2-+-h}
		q_h^a - q_h^b \in H^1_{0,h}(\Omega_h) \quad \hbox{ and }\quad \norm{q^a_h- q_h^b}_{H^1_{0,h}(\Omega_h)} \leq M,
	\end{equation}
	we have
	\begin{multline}
		\label{Stability-Bellassoued-h}
		\norm{q^a_h - q^b_h }_{L^2_h(\Omega_h)}  \leq Ch ^{1/(1+\alpha)} + C \left[ \log\left( 2+ \frac{C}{\norm{\mathscr{M}_h[q^a_h] - \mathscr{M}_h[q^b_h]}_{H^1(0,T; L^2(\Gamma_{0}))} }\right)\right]^{-\frac{1}{1+\alpha}}
		\\
		+ Ch \sum_{k=1,2} \norm{\partial_{h,k}^+\partial_{tt} y_{h}[q_h^a] - \partial_{h,k}^+\partial_{tt} y_h[q_h^b]}_{L^2(0,T;L^2_h(\Omega_{h,k}^-))}.
	\end{multline}
	Besides, the constant $C$ depends on the constants $m$, $M$ in \eqref{Error-W-1-2-+-h},  $\alpha_0$ in \eqref{Ass-Smoothness-Bel-h}, an a priori bound on $\norm{y^0_h}_{H^1_h(\Omega_h)} + \norm{y_h[q_h^a]}_{H^1(0,T; L^\infty_h(\Omega_h)) \cap W^{2,1}(0,T; L^2_h(\Omega_h))}$, and on the geometric configuration.
\end{theorem}

When compared with the corresponding continuous result of Theorem \ref{ThmBellassoued}, the stability estimate \eqref{Stability-Bellassoued-h} contains two extra terms: the penalization term \eqref{PenalizationTerm} and the new term $C h^{1/(1+\alpha)}$. 

The proof of \eqref{Stability-Bellassoued-h}, given in Section~\ref{AEC}, follows the same path as in the continuous case and combines the stability results obtained in the case  where the Gamma conditions are satisfied with stability results obtained for solutions of the wave equation through a Fourier-Bros-Iagoniltzer transform and a Carleman estimate for elliptic operators due to  \cite{BoyerHubertLeRousseau,BoyerHubertLeRousseau2}.
Hence, the penalization term \eqref{PenalizationTerm} is remanent from Theorem \ref{TCWE1}. But the term $C h^{1/(1+\alpha)}$ comes from the fact that the parameters within the discrete Carleman estimates cannot be made arbitrarily large and should be at most at the order of $1/h$. This fact has already been observed in several articles in the elliptic case, see \cite{BoyerHubertLeRousseau,BoyerHubertLeRousseau2,ErvGournay}. We also refer to \cite{KlibanovSantosa} for a previous work related to the convergence of the quasi-reversibility method.

\subsubsection{Discrete convergence results}
The stability results of the previous Theorems \ref{TCWE1} and \ref{TCWElog} suggest to introduce the observation operators $\widetilde{\mathscr{M}_h}= \widetilde{\mathscr{M}_h}\{ y^0_h,y^1_h,f_h, f_{\partial, h} \} $ defined for $h>0$ by
\begin{equation}
	\label{DiscreteFh}
	\begin{array}{lcll}
		\widetilde{\mathscr{M}_h} :  	&L^\infty_{h} (\Omega_h) &\to &  L^2(0,T;L^2(\Gamma_{0}))\times L^2((0,T)\times\Omega) \\
		&q_h &\mapsto &\big(\partial_{\nu} \eh(y_{h}[q_h]) , h\nabla_x \eh( \partial_{tt} y_{h}[q_h]) \big),
	\end{array}
\end{equation}
where $y_h[q_h]$ is the solution of \eqref{DisWaveEq} with potential $q_h$ and data $ y^0_h,\, y^1_h,\, f_h, \, f_{\partial, h}$ and $\eh$ is the piecewise affine extension defined in \eqref{eh}.
Corresponding to the case $h = 0$, we introduce its continuous analogous $\widetilde{\mathscr{M}_0} = \widetilde{\mathscr{M}_0}\{y^0,y^1,f,f_\partial\}$:
\begin{equation}\label{contF}
	\begin{array}{lcll}
		\widetilde{\mathscr{M}_0}: 	&L^\infty (\Omega) &\to&  L^2(0,T;L^2(\Gamma_0))\times L^2((0,T)\times\Omega) \\
		& q &\mapsto&\big(\partial_\nu y[q], 0 \big), 
	\end{array}
\end{equation}
where $y[q]$ is the solution of \eqref{WaveEq-Q}. Recall that according to \cite{LasieckaLionsTriggiani}, this map $\widetilde{\mathscr{M}_0}$ is well defined on $L^\infty(\Omega)$ for data 
\begin{equation}
	\label{SmallerPossibleClass}
	\begin{array}{rr}
	(y^0, y^1,f, f_\partial) \in H^1(\Omega) \times L^2(\Omega)\times L^1((0,T) ; L^2(\Omega))\times H^1((0,T)\times \partial \Omega) \\
	\hfill{} \hbox{ with } \left.y^0\right|_{\partial \Omega} = f_\partial(t = 0),
	\end{array}
\end{equation}
that we shall always assume in the following.

Remark that with these notations, the quantities 
$$
	\norm{\mathscr{M}_h[q_h^a] - \mathscr{M}_h[q_h^b]}_{H^1((0,T); L^2_h(\Gamma_{0,h}))} + h \sum_{k=1,2}\norm{\partial_{h,k}^+ \partial_{tt} y_h[q_h^a] -\partial_{h,k}^+ \partial_{tt} y_h[q_h^b]  }_{L^2((0,T) \times \Omega_{h,k}^-)}
$$
and 
$$
	\norm{\widetilde{\mathscr{M}_h}[q_h^a] -\widetilde{\mathscr{M}_h}[q_h^b]}_{ H^1(0,T;L^2(\Gamma_0))\times L^2((0,T)\times\Omega)}
$$
are equivalent, uniformly with respect to the parameter $h>0$.
Hence the stability results in Theorems \ref{TCWE1} and \ref{TCWElog} easily recast into stability results for $\widetilde{\mathscr{M}_h}$.
\\
Our convergence result is then the following: 

\begin{theorem}[Convergence of the inverse problem]
	\label{Thm-Cvg}
		Let $q \in H^1\cap L^\infty(\Omega) $ and assume that we know $q_\partial = q|_{\partial \Omega}$. 
		Let the data $(y^0,y^1, f, f_\partial)$ follow conditions \eqref{SmallerPossibleClass} and  
		the positivity condition $\inf_{\overline{\Omega} } |y^0| \geq \alpha_0>0$.
	Furthermore, assume that the trajectory $y[q]$ solution of \eqref{WaveEq-Q} satisfies
	\begin{equation}
		\label{ConditionOnYq-1}
		y[q] \in H^2(0,T; H^1(\Omega)) \cap H^1(0,T; H^2(\Omega) ).
	\end{equation}
	We can construct discrete sequences $(y^0_h, y^1_h, f_h, f_{\partial,h})$, such that if we assume either 
	\begin{quote}
	$\bullet$	 $(\Omega,  \Gamma_0,T)$ satisfy the configuration \eqref{Gamma-+-T-sqrt2}, and in this case we define \\
	$
	X_h = L^\infty_h(\Omega_h),
	$
	\end{quote}
 or 
	\begin{quote}
	$\bullet$	$(\Omega,  \Gamma_0,  \Gamma_+)$ satisfy the configuration \eqref{Config-Bellassoued}, $T>0$ is large enough, $q$ is known on $\mathcal{O}$, neighborhood of $\Gamma_+$, and takes the value $q|_{\mathcal O} = Q$, 
	and we define 
	\begin{multline*}
	X_h = \{q_h \in L^\infty_h(\Omega_h) \hbox{ s.t. } q_h|_{ \mathcal O_h} = \trh(Q),  \\
	\hbox{ and $q_h$, extended on } \partial \Omega_h \hbox{ by }  q_h|_{\partial \Omega_h} = \rhB(q_\partial), \hbox{ belongs to } H^1_{h}(\Omega_h) \}, 
	\end{multline*}
	that we endow with the $L^\infty(\overline{\Omega}_h) \cap H^1_h(\Omega_h)$-norm,
	\end{quote}
	then\\
	- there exists a sequence $(q_h)_{h>0}\in X_h$ of potentials such that 
		\begin{equation}
			\label{Condition-For-Conv}
			\limsup_{ h \to 0}\norm{q_h}_{X_h} <\infty, 
		\quad\hbox{ and }\quad \lim_{h \to 0} \norm{\widetilde{\mathscr{M}_h}
		[q_h] - \widetilde{\mathscr{M}_0}
		[q]}_{ H^1(0,T;L^2(\Gamma_0))\times L^2((0,T)\times\Omega)} = 0,
		\end{equation}
	- for all sequence $(q_h)_{h>0}\in X_h$ of potentials satisfying \eqref{Condition-For-Conv}, we have
	$$
		 \lim_{h \to 0} \norm{\eh^0 (q_h)-q }_{L^2(\Omega)} = 0.
	$$	
\end{theorem}

Let us briefly comment the assumptions of Theorem \ref{Thm-Cvg}, which might seem much stronger compared to the ones for the stability results in Theorems \ref{TCWE1} and \ref{TCWElog}. This is due to the consistency of the inverse problem, detailed in Lemma \ref{Lem-Consistency-Main}, which requires to find discrete potentials such that the corresponding solutions of the discrete wave equation \eqref{DisWaveEq} belongs to $H^1(0,T; L^\infty(\Omega))$. But this class is not very natural for the wave equation, and we will thus rather look for the class $H^1(0,T; H^2(\Omega))$, which embeds into $H^1(0,T; L^\infty(\Omega))$ according to Sobolev's embeddings (since $\Omega\subset \mathbb R^2$). This is actually the only place in the article which truly depends on the dimension. 

It may also seem surprising to assume the knowledge of $q$ on the boundary even in the configuration \eqref{Gamma-+-T-sqrt2}, for which Theorem \ref{TCWE1} applies with only an $L^\infty_h(\Omega_h)$-norm on the potential. This is actually due to the fact that the knowledge of $q_{|\partial \Omega}$ is hidden in the regularity assumptions on $y[q]$. Indeed, if $y[q]$ is smooth and satisfies \eqref{WaveEq-Q}, we may write $\partial_{tt} y(0, x) = \Delta y^0(x)- q(x) y^0 (x) + f(0,x)$ for all $x \in \Omega$ and in particular $x \in \partial \Omega$, whereas $\partial_{tt} y(0,x) = \partial_{tt} f_\partial(0,x)$ for $x \in \partial \Omega$. In particular, since $y^0$ does not vanish on the boundary, these two identities imply that $q_{|\partial \Omega}$ can be immediately deduced from the knowledge of $y^0, \, f$ and $f_\partial$ for sufficiently smooth solutions, see Remark \ref{Rem-Reg-Y[Q]}.

Details on the derivation of Theorem \ref{Thm-Cvg} are given in Section \ref{SecCV}, with a particular emphasis on the related consistency issues. In particular, Lemma \ref{Lem-Consistency-Main} explains how to derive the discrete data $y^0_h$, $y^1_h$, $f_h$ and $f_{\partial,h}$ from the data $y^0$, $y^1$, $f$, $f_\partial$ and $\left.q\right|_{\partial\Omega}$.

\subsection{Outline}
Section~\ref{AHC} will be devoted to the establishment  of a uniform semi-discrete hyperbolic Carleman estimates in two-dimensions,
including the boundary observation case in Theorem \ref{Thm-CarlemanDisc-Boundary} and the distributed observation
case in Theorem~\ref{Thm-CarlemanDisc-Distributed}. We will then derive from these tools the discrete stability result of Theorem~\ref{TCWE1}.
In Section~\ref{AEC}, we will present a revisited version of Theorem~\ref{ThmBellassoued} based on a global elliptic Carleman estimate and follow the same strategy to establish the discrete stability result of Theorem \ref{TCWElog}, that relies  on a global uniform semi-discrete elliptic Carleman estimate due to \cite{BoyerHubertLeRousseau2}. Finally, Section~\ref{SecCV} will gather the proof of Theorem~\ref{Thm-Cvg}, some informations about the Lax type argument, and  a detailed discussion about consistency issues.

%
\section{Application of hyperbolic Carleman estimates}\label{AHC}
 In this section, we discuss uniform Carleman estimates for the 2-d space semi-discrete wave operator discretized using the finite difference method and applications to stability issues for discrete wave equations. These discrete results are closely related to the study of the 1-d space semi-discrete wave equation one can read in \cite{BaudouinErvedoza11}. Actually, our methodology (here and in \cite{BaudouinErvedoza11}) goes back to the articles \cite{BoyerHubertLeRousseau,BoyerHubertLeRousseau2} where uniform Carleman estimates were derived for elliptic operators.
%
%
\subsection{Discrete Carleman estimates for the wave equation in a square}\label{DCE}
%
The proofs of the results stated here will be presented in Sections \ref{2.3} and \ref{Sec-Carleman-Distributed-Proofs}.\\
Recall that we assume the geometric configuration
\begin{equation}
	\label{Configuration-Lipschitz-boundary}
	\Omega = (0,1)^2,\quad \Gamma_0 \supset \Gamma_+ = (\{1\} \times (0,1)) \cup ((0,1) \times \{1\} ).
\end{equation}

\noindent{\bf Carleman weight functions.} Let $a >0$, $x_a = (-a,-a) \notin \overline\Omega = [0,1]^2$,  and $\beta \in (0,1)$.  In $[-T,T]\times [0,1]^2$, we define the weight functions $\psi =\psi(t,x)$ and $\varphi =\varphi(t,x)$ 
as
\begin{equation}
	\psi (t,x)= | x-x_a|^2 -\beta t^2+c_0,
	\qquad ~\varphi(t,x) = e^{\mu \psi(t,x)},
	\label{varphi}
\end{equation}
where $c_0>0$ is such that $\psi\geq 1$ on $[-T,T]\times[0,1]^2$ and $\mu \geq 1$ is a parameter. \\

\noindent{\bf Uniform discrete Carleman estimates: the boundary case.} One of the main results of this article is the following:
\begin{theorem}
	\label{Thm-CarlemanDisc-Boundary}
	Assume the configuration \eqref{Configuration-Lipschitz-boundary} for $\Omega$ and $\Gamma_+$.
	Let $a>0,\, \beta \in (0,1)$ in \eqref{varphi} and $T>0$.
	There exist $\tau_0\geq 1$, $\mu\geq 1$, $\varepsilon >0$, $h_0>0$ and a constant $C=C(\tau_0,\mu, T, \varepsilon,\beta) >0$ independent of $h>0$ such that for all $ h\in (0,h_0)$ 
and $ \tau\in \left(\tau_{0},\varepsilon/ h\right)$, for all $w_h$ satisfying
	\begin{equation}
		\label{Assumption-W}
		\left\lbrace
			\begin{array}{ll}
				\Box_h w_h \in L^2(-T,T; L^2_h({\Omega_h})),
				\\
				w_{0,j}(t) = w_{N+1,j}(t) = w_{i,0}(t) = w_{i,N+1}(t) = 0 &\quad \forall  t \in (-T,T), \, i,j\in \llbracket 0, N+1\rrbracket, 
				\\
				w_{i,j}(\pm T) = \partial_t w_{i,j}(\pm T ) = 0 &\quad \forall  i,j\in \llbracket 0, N+1\rrbracket,
			\end{array}
		\right.
	\end{equation}
	we have 
	\begin{gather}
		\tau \int_{-T}^{T} \int_{\Omega_{h}} e^{2\tau\varphi_h}| \partial_t w_h|^2 \, dt
		+
		 \tau\sum_{k=1,2} \int_{-T}^{T} \int_{\Omega_{h,k}^-} e^{2\tau\varphi_h} |\partial_{h,k}^+ w_h |^2\, dt
		+ \tau^3\int_{-T}^{T}\int_{\Omega_{h}} e^{2\tau\varphi_h}| w_h |^2 \, dt
		\nonumber \\
		\leq C\int_{-T}^{T}\int_{\Omega_{h}}e^{2\tau\varphi_h}| \Box_h w_h|^2 \, dt  + 
		C\tau\sum_{k=1,2} \int_{-T}^{T} \int_{\Gamma_{h,k}^+} e^{2\tau\varphi_h} \left| \partial_{h,k}^- w_h\right|^2 \, dt
		\label{CarlemD}\\
		+ C\tau h^2\sum_{k=1,2} \int_{-T}^{T}  \int_{\Omega_{h,k}^-} e^{2\tau\varphi_h} |\partial_{h,k}^+\partial_t w_h|^2\, dt, \nonumber
	\end{gather}
	where $\varphi_h$ is defined as the approximation of $\varphi$ given by  $\varphi_{h}(t) = \rh\varphi(t)$ for $t \in [0,T]$.
	\\
	Besides, if $w_h(0,x_h) = 0 $ for all $x_h \in \Omega_h$, we also have
	\begin{multline}
		\label{Est-t=0}
		\tau^{1/2} \int_{\Omega_h} e^{2 \tau \varphi_h (0)} |\partial_t w_h(0,x_h)|^2 
		\leq
		 C\int_{-T}^{T}\int_{\Omega_{h}}e^{2\tau\varphi_h}| \Box_h w_h|^2 \, dt  
		\\
		+ 
		C\tau\sum_{k=1,2} \int_{-T}^{T} \int_{\Gamma_{h,k}^+} e^{2\tau\varphi_h} \left| \partial_{h,k}^- w_h\right|^2 \, dt
		+ C\tau h^2\sum_{k=1,2} \int_{-T}^{T}  \int_{\Omega_{h,k}^-} e^{2\tau\varphi_h} |\partial_{h,k}^+\partial_t w_h|^2\, dt.
	\end{multline}
\end{theorem}
The proof of Theorem \ref{Thm-CarlemanDisc-Boundary} will be given later in Section \ref{2.3}. It is very similar to the one of \cite[Theorem 2.2]{BaudouinErvedoza11} but more intricate.
The continuous counterpart of Theorem \ref{Thm-CarlemanDisc-Boundary} is given in \cite[Theorem 2.1 and Theorem 2.10]{BaudouinDeBuhanErvedoza}, and very close versions of it can be found in \cite{ImYamIP01,Im02}. However, two main differences with respect to the corresponding continuous Carleman estimates appear: \\
\indent $\bullet$ The parameter $\tau$ is limited from above by the condition $\tau h \leq \varepsilon$: this restriction on the range of the Carleman parameter always appear in discrete Carleman estimates, see \cite{BoyerHubertLeRousseau,BoyerHubertLeRousseau2,BaudouinErvedoza11,ErvGournay}. This is related to the fact that the conjugation of discrete operators with the exponential weight behaves as in the continuous case only for $\tau h$ small enough, since for instance
	$$
		e^{ \tau \varphi} \partial_h (e^{-\tau \varphi}) \simeq - \tau \partial_x \varphi \quad \hbox{only for $\tau h$ small enough.}
	$$
\indent $\bullet$ There is an extra term in the right hand-side of \eqref{CarlemD}, namely 
	\begin{equation}
		\label{Additional-Term-Carl}
		\tau h^2\sum_{k=1,2} \int_{-T}^{T}  \int_{\Omega_{h,k}^-} e^{2\tau\varphi_h} |\partial_{h,k}^+\partial_t w_h|^2\, dt, 
	\end{equation}
	that cannot be absorbed by the left hand-side terms of \eqref{CarlemD}. This is not a surprise as this term already appeared in the Carleman estimates obtained for the waves in the 1-d case, see \cite[Theorem 2.2]{BaudouinErvedoza11}, and also in the multiplier identity \cite{InfZua}. As it has been widely studied in the context of the control of discrete wave equations (see e.g. the survey articles \cite{Zua05Survey,ErvZuaCime}), this term is needed since the discretization process creates spurious frequencies that do not travel at the velocity prescribed by the continuous dynamics (see also \cite{Tref}). 
	Also note that this additional term only concerns the high-frequency part of the solutions, since the operators $h \partial_{h,1}^+$, $h\partial_{h,2}^+$ are of order $1$ for frequencies of order $1/h$, whereas it can be absorb by the right hand-side of \eqref{CarlemD} for scale $\mathcal{O}(1/h^{1- \varepsilon})$ for all $\varepsilon >0$ by choosing $h$ sufficiently small.\\

\noindent{\bf Uniform discrete Carleman estimates: the distributed case.} The usual assumption in the distributed case for getting Carleman estimates in the continuous setting (see \cite{Im02}) is that the observation set $\omega$ is a neighborhood of a part of the boundary satisfying the Gamma condition~\eqref{Gamma-Condition}. Since in our geometric setting $\Omega = (0,1)^2$, with no loss of generality we may assume that there exists $\delta >0$ such that \eqref{Configuration-Lipschitz-distributed} holds.
Under these conditions, we show:
\begin{theorem}
	\label{Thm-CarlemanDisc-Distributed}
	Assume the configuration \eqref{Configuration-Lipschitz-distributed} for $\omega$. 
	We then set 
	$$
		\omega_h = \Omega_h \cap \omega, \quad \omega_{h,k}^- = \Omega_{h,k}^- \cap \overline{\omega}, \quad k \in \{1, 2\}.
	$$
	Let $a>0,\, \beta \in (0,1)$ in \eqref{varphi} and $T>0$.
	There exist $\tau_0\geq 1$, $\mu\geq 1$, $\varepsilon >0$, $h_0>0$ and a constant $C=C(\tau_0,\mu, T, \varepsilon,\beta) >0$ independent of $h>0$ such that for all $ h\in (0,h_0)$ 
and $ \tau\in \left(\tau_{0},\varepsilon/ h\right)$, for all $w_h$ satisfying \eqref{Assumption-W},
	\begin{gather}
		\tau \int_{-T}^{T} \int_{\Omega_{h}} e^{2\tau\varphi_h}| \partial_t w_h|^2 \, dt
		+
		 \tau\sum_{k=1,2} \int_{-T}^{T} \int_{\Omega_{h,k}^-} e^{2\tau\varphi_h} |\partial_{h,k}^+ w_h |^2\, dt
		+ \tau^3\int_{-T}^{T}\int_{\Omega_{h}} e^{2\tau\varphi_h}| w_h |^2 \, dt
		\nonumber \\
		\leq C\int_{-T}^{T}\int_{\Omega_{h}}e^{2\tau\varphi_h}| \Box_h w_h|^2 \, dt   
		+ C\tau h^2\sum_{k=1,2} \int_{-T}^{T}  \int_{\Omega_{h,k}^-} e^{2\tau\varphi_h} |\partial_{h,k}^+\partial_t w_h|^2\, dt 
		\label{CarlemD-distributed}\\
		C\tau\int_{-T}^{T}\int_{\omega_{h}} e^{2\tau\varphi_h}| \partial_t w_h |^2 \, dt
		+
		C \tau\sum_{k=1,2} \int_{-T}^{T} \int_{\omega_{h,k}^-} e^{2\tau\varphi_h} |\partial_{h,k}^+ w_h |^2\, dt
		+ 
		C\tau^3\int_{-T}^{T}\int_{\omega_{h}} e^{2\tau\varphi_h}| w_h |^2 \, dt,
		 \nonumber
	\end{gather}
	where $\varphi_{h}(t) = \rh\varphi(t)$ for $t \in [0,T]$. Besides, if $w_h(0,x_h) = 0 $ for all $x_h \in \Omega_h$, the term 
	$$\tau^{1/2} \ds\int_{\Omega_h} e^{2 \tau \varphi_h (0)} |\partial_t w_h(0,x_h)|^2 $$ is also bounded by the right hand side of \eqref{CarlemD-distributed}.
\end{theorem}
Of course, Theorem \ref{Thm-CarlemanDisc-Distributed} shares the same features as Theorem \ref{Thm-CarlemanDisc-Boundary}. Actually, Theorem~\ref{Thm-CarlemanDisc-Distributed} is a corollary of Theorem \ref{Thm-CarlemanDisc-Boundary}, and we postpone its proof to Section \ref{Sec-Carleman-Distributed-Proofs}.
%
\subsection{Proof of the discrete Carleman estimate - boundary case}\label{2.3}
%
\begin{proof}[Proof of Theorem \ref{Thm-CarlemanDisc-Boundary}]
The proof of estimate \eqref{CarlemD}  is long and follows the same lines as \cite[Theorem 2.2]{BaudouinErvedoza11}. 
In particular, the main idea is to work on the conjugate operator
\begin{equation}
	\label{L-h-ConjugateOperator}
	\mathscr{L}_h v_h:= e^{\tau\varphi_h}  \Box_h  (e^{-\tau\varphi_h} v_h).
\end{equation}
The precise computation of $\mathscr{L}_h$  already involves tedious computations summed up below:
\begin{proposition}\label{Prop-B}
	The conjugate operator $\mathscr{L}_h$ can be written in the following way:
	\begin{align}\label{ConjugateOperator}
		& \mathscr{L}_h v_h =
		  \partial_{tt}v_h - 2\tau\mu \, \varphi\, \partial_t \psi\, \partial_t v_h
		+  \tau^2\mu^2\varphi^2 \left( \partial_t \psi \right)^2 v_h 
		-\tau\mu^2\varphi\left( \partial_t \psi \right)^2v_h  
		-\tau\mu\varphi( \partial_{tt} \psi) v_h
		\\
		 & -  \sum_{k=1,2} (1+A_{0,k})\Delta_{h,k} v_h +~2 \tau\mu  \sum_{k=1,2} A_{1,k}\partial_{h,k} v_h
		-  \sum_{k=1,2} (\tau^2\mu^2  A_{2,k} - \tau\mu^2 A_{3 ,k}- \tau\mu A_{4,k})v_h,\nonumber
	\end{align}
	where the coefficients $A_{\ell,k}$ are given, for $(t,x_h) \in (-T,T) \times \Omega_h$ and $\tn{e}^1 = (1, 0)$, $\tn{e}^2 = (0,1)$, by
	\begin{align}
		A_{1,k} (t,x_h)	
		&= \dfrac 12\int_{-1}^1 \left[\varphi \partial_{x_k}\psi\right](t, x_h+\sigma h \tn{e}^k) \dfrac{e^{-\tau\varphi(t, x_h+\sigma h \tn{e}^k)}}{e^{-\tau\varphi(t,x_h)}}\,d\sigma, 
		  \label{A1}\\
		A_{2,k}(t,x_h)
		&= \int_{-1}^1 (1-|\sigma|) \left[ \varphi^2 (\partial_{x_k}\psi)^2\right](t,x_h+\sigma h \tn{e}^k) \dfrac{e^{-\tau\varphi(t,x_h+\sigma h \tn{e}^k)}}{e^{-\tau\varphi(t,x_h)}}\,d\sigma,
		\label{A2}\\
		A_{3,k}(t,x_h) &= \int_{-1}^1 (1-|\sigma|) \left[\varphi (\partial_{x_k}\psi)^2 \right](t,x_h+\sigma h \tn{e}^k) \dfrac{e^{-\tau\varphi(t,x_h+\sigma h \tn{e}^k)}}{e^{-\tau\varphi(t,x_h)}}\,d\sigma,
		\label{A3}\\
		A_{4,k}(t,x_h) &= \int_{-1}^1 (1-|\sigma|) \left[ \varphi\partial_{x_k x_k}\psi\right](t,x_h+\sigma h \tn{e}^k) \dfrac{e^{-\tau\varphi(t,x_h+\sigma h \tn{e}^k)}}{e^{-\tau\varphi(t,x_h)}}\,d\sigma,\label{A4}
		\\
		A_{0,k} &= \dfrac {h^2}2 (\tau^2\mu^2  A_{2,k} - \tau\mu^2 A_{3,k}- \tau\mu A_{4,k}).
	\end{align}
	In particular, these functions $A_{\ell,k}$ defined on $[0,T] \times \Omega_h$ can be extended on $[0,T] \times \overline{\Omega}$ in a natural way by the formulas \eqref{A1}--\eqref{A4} and satisfy the following property: setting
	$$
		f_{0,k} = 0, \quad f_{1,k} = \varphi \partial_{x_k} \psi, \quad f_{2,k} = \varphi^2  (\partial_{x_k}\psi)^2, \quad f_{3,k} = \varphi (\partial_{x_k}\psi)^2, \quad f_{4,k} =  \varphi\partial_{x_k x_k}\psi,
 	$$
	for some constants $C_\mu$ depending on $\mu$ but independent of $\tau$ and $h$, we have
	\begin{equation}
		\label{ProxC2}
		\norm{A_{\ell, k} - f_{\ell,k}}_{C^2([0,T] \times \overline{\Omega})} \leq C_\mu \tau h, \qquad \forall \ell \in \{0,\ldots,4\},\, \forall k \in \{1, 2\}.
	\end{equation}
\end{proposition}
The proof of Proposition \ref{Prop-B} can be easily deduced from the detailed one in \cite[Propositions 2.7, 2.8 and Lemma 2.9, 2.10]{BaudouinErvedoza11} and the details are left to the reader. Note in particular that \eqref{ProxC2} implies for all $(\ell, k) \in \llbracket  0,4 \rrbracket \times \{1,2\}$, 
\begin{multline*}
	\norm{A_{\ell, k} - \rh f_{\ell,k}}_{L^\infty((0,T);L^\infty_h( \Omega_h))} 
	+ \sum_{k' = 1,2} \norm{ \partial_{h,k'}^+ A_{\ell,k} - \rh \partial_x f_{\ell,k} }_{L^\infty((0,T);L^2_h(\Omega_{h,k'}^-))} 
	\\
	+ \norm{\Delta_{h} A_{\ell, k} - \rh \Delta f_{\ell,k}}_{L^\infty((0,T);L^\infty_h( \Omega_h))} \leq C_\mu \tau h.
\end{multline*}
Afterwards, one step of the usual way to prove a Carleman estimate is to split $\mathscr{L}_h$ into two operators $\mathscr{L}_{h,1}$ and $\mathscr{L}_{h,2}$, that, roughly speaking, corresponds to a decomposition into a self-adjoint part and a skew-adjoint one. To be more precise, using the notations 
$$
	A_2 = A_{2, 1} + A_{2,2}, \qquad A_3 = A_{3,1} + A_{3,2}, \qquad A_4 = A_{4,1} + A_{4,2}, 
$$
we set
\begin{eqnarray}
	\mathscr{L}_{h,1}v_h
		&=&
		\partial_{tt}v_h  -  \sum_{k=1,2} (1+A_{0,k})\Delta_{h,k} v_h 
		+  \tau^2\mu^2 \left(\varphi^2 \left( \partial_t \psi \right)^2 
		-  A_{2} \right)v_h\,,
	\label{P1}\\
	\mathscr{L}_{h,2}v_h
		&=&
		(\alpha_1-1) \tau \mu\left( \varphi \partial_{tt} \psi - A_{4} \right) v_h
		- \tau\mu^2 \left( \varphi |\partial_t \psi|^2 -  A_{3} \right)v_h \nonumber \\
		&&\hspace{4.5cm}- 2 \tau\mu  \left( \varphi \partial_t \psi  \partial_t v_h - \sum_{k=1,2} A_{1,k} \partial_{h,k} v_h\right)\, ,
	\label{P2}\\
	\mathscr{R}_h v_h &=&  \alpha_1 \tau \mu\left(\varphi \partial_{tt} \psi - A_{4} \right)v_h\, , \label{Rh} \quad \hbox{ with } \alpha_1  = \frac{\beta + 1}{\beta + 2},
\end{eqnarray}
so that we have $\mathscr{L}_{h,1}v+\mathscr{L}_{h,2}v=\mathscr{L}_h v+\mathscr{R}_h v$. Here, $\mathscr{R}_h$ will be considered as a lower order perturbation of no interest and the letter $\mathscr{R}$ states for ``reminder''. More precisely, all our computations will be based on the following straightforward estimate:
\begin{multline}\label{double}
	\int_{-T}^{T} \int_{\Omega_h}| \mathscr{L}_{h,1}v_h| ^2\, dt + 	\int_{-T}^{T} \int_{\Omega_h}| \mathscr{L}_{h,2}v_h | ^2\, dt
	+
	 2\int_{-T}^{T} \int_{\Omega_h}\mathscr{L}_{h,1}v_h\,  \mathscr{L}_{h,2}v_h \, dt
	\\
	 \leq 
	2 \int_{-T}^{T} \int_{\Omega_h} | \mathscr{L}_h v_h|^2 \, dt 
	+
	2 \int_{-T}^{T} \int_{\Omega_h} |\mathscr{R}_h v|^2  \, dt.
\end{multline}
In particular, we claim the following proposition, proved in Appendix \ref{Sec-Proof-Prop-Decompo}:
\begin{proposition}\label{PropDecompo}
	For any $T>0$, there exist $\mu\geq 1$, $\tau_0\geq 1 $, $\varepsilon_0 >0$ and a constant $C_0>0$ such that for all 
	$\tau \in (\tau_0, \varepsilon_0/h)$, for all $v_h$ satisfying $v_{0,j} = v_{N+1,j} = v_{i,0} = v_{i,N+1} = 0$ 
	and $v_{i,j}(\pm T) = \partial_t v_{i,j}(\pm T ) = 0, \forall i,j \in \llbracket 0, N+1\rrbracket $, 
	\begin{multline}
		 \tau \int_{-T}^{T} \int_{\Omega_{h}}| \partial_t v_h|^2 \, dt
		+  \tau\sum_{k=1,2} \int_{-T}^{T} \int_{\Omega_{h,k}^-} \ |\partial_{h,k}^+ v_h|^2\, dt
		+	\tau^3 \int_{-T}^{T} \int_{\Omega_{h}}|v_h|^2\, dt
		+ \int_{-T}^{T} \int_{\Omega_{h}}| \mathscr{L}_{h,1}v_h| ^2\, dt
 		\label{decompo}  \\
		\leq 
		C_0 \int_{-T}^{T} \int_{\Omega_{h}} |\mathscr{L}_h v_h|^2 \, dt
		+
		C_0 \tau\sum_{k=1,2} \int_{-T}^{T} \int_{\Gamma_{h,k}^+}  \left| \partial_{h,k}^- v_h\right|^2 \, dt
		+ C_0 \tau h^2 \sum_{k=1,2} \int_{-T}^{T}  \int_{\Omega_{h,k}^-} |\partial_{h,k}^+\partial_t v_h|^2\, dt
	\end{multline}
	where the operators $\mathscr{L}_h$ and $\mathscr{L}_{h,1}$ are defined by \eqref{L-h-ConjugateOperator} and \eqref{P1}.
\end{proposition}
The proof of Proposition \ref{PropDecompo} is the core of the derivation of the discrete Carleman estimate and consists in estimating from below the cross-product $\int_{-T}^{T} \int_{\Omega_h}\mathscr{L}_{h,1}v_h\, \mathscr{L}_{h,2}v_h \, dt$ in \eqref{double}. This is done in two steps: Computation of the cross-product and computations of the leading order terms coefficients in front of $v_h,\, \partial_t v_h, \partial_{h,k}^+ v_h$. The proof of Proposition \ref{PropDecompo} is given in Appendix \ref{Sec-Proof-Prop-Decompo}.

Actually, this closely follows the proof of \cite[Lemma 2.11]{BaudouinErvedoza11} corresponding to the 1-d case. The main novelties with respect to \cite[Lemma 2.11]{BaudouinErvedoza11} are the following ones:\\
\indent $\bullet$ Some computations in the cross-product of $\mathscr{L}_{h,1}v_h $ and $ \mathscr{L}_{h,2}v_h$ are new since the term $(\alpha_1-1) \tau \mu( \varphi \partial_{tt} \psi -\sum_{k} A_{4,k} ) v_h$ in $\mathscr{L}_{h,2}$ in \eqref{P2} vanishes in dimension $1$. Actually, the coefficient $\alpha_1$ is chosen in some range that depends on the dimension $d$ of the space variable and is required to belong to $(2\beta/(\beta+d), 2/(\beta + d))$. Hence, since $d=1$ in \cite{BaudouinErvedoza11}, we chose $\alpha_1 = 1$ to simplify the computations.\\
\indent $\bullet$ There are also new cross-products involving integration by parts of discrete derivatives in different directions. In particular, besides the 1-d integration by parts formula in \cite[Lemma~2.6]{BaudouinErvedoza11} that we recall in \ref{Sec-IPP}, we will need the following specific 2-d formula:
	 \begin{lemma}[discrete integration by part formula]
		\label{LemIPP2}
		Let $v_h,g_h$ be discrete functions depending on the variable $x_h \in [0,1]^2$ such that $v_h=0$ on the boundary of the square. Then we have the following identity:
		\begin{multline}\label{New-1}
			 \int_{\Omega_h} g_h \, \Delta_{h,1} v_h\,  \partial_{h,2} v_h ~=~
				  \dfrac 12\int_{\Omega_{h,1}^-} |\partial_{h,1}^+ v_h|^2 \partial_{h,2} (m_{h,1}^+ g_h) 
		  		- \int_{\Omega_{h,1}^-} \partial_{h,1}^+ v_h\,  m_{h,1}^+ (\partial_{h,2} v_h)\,  \partial_{h,1}^+ g_h\\
				 - \frac{h^2}{4} \int_{\Omega_{h}^-} |\partial_{h,1}^+ \partial_{h,2}^+ v_h|^2 \partial_{h,2}^+(m_{h,1}^+ g_h).
		\end{multline}	
	\end{lemma}
	Though the formula \eqref{New-1} cannot be found as it is in \cite{BaudouinErvedoza11}, it can be easily deduced from the integration by parts formula in Appendix \ref{Sec-IPP} and the proof is left to the reader.
\\

Furthermore, if we assume $v_h(0)=0$ in $\Omega_h$, we can compute  the following cross-product (it is a straightforward modification of the computations in \cite[p.586]{BaudouinErvedoza11}):
\begin{multline*}
	\int_{-T}^0 \int_{\Omega_h} \partial_t v_h\,  \mathscr{L}_{h,1} v_h \, dt  
	= 
	\frac{1}{2} \int_{\Omega_h} |\partial_t v_h(0)|^2 
	- \frac{1}{2}  \sum_{k = 1,2}  \int_{-T}^0\int_{\Omega_{h,k}^-} m_{h,k}^+ (\partial_t A_{0,k}) \, |\partial_{h,k}^+ v_h|^2 \, dt
	\\
	+\sum_{k  = 1, 2} \int_{-T}^0 \int_{\Omega_{h,k}^-} \partial_{h,k}^+ A_{0,k}\, \partial_{h,k}^+ v_h\, m_{h,k}^+( \partial_t v_h )\, dt 
	- \frac{\tau^2 \mu^2}{2} \int_{-T}^0 \int_{\Omega_h} |v_h|^2 \partial_t \Big(\varphi^2 \left( \partial_t \psi \right)^2 - A_{2} \Big) \, dt.
\end{multline*}
Therefore, based on Proposition \ref{Prop-B}, we easily get
\begin{multline*}
	\int_{\Omega_h} |\partial_t v_h(0)|^2  
	\leq 
	\frac{C}{\sqrt{\tau}} \int_{-T}^0 \int_{\Omega_h} |\mathscr{L}_{h,1} v_h|^2 \, dt  + C \sqrt{\tau} \int_{-T}^0 \int_{\Omega_h} |\partial_t v_h|^2 \, dt  
	\\
	+~C_\mu \tau h \sum_{k = 1,2}  \int_{-T}^0\int_{\Omega_{h,k}^-}  |\partial_{h,k}^+ v_h|^2 \, dt + C_\mu \tau h \sum_{k = 1,2}  \int_{-T}^0\int_{\Omega_{h,k}^-}  |\partial_{t} v_h|^2 \, dt
	+ 
	C_\mu \tau^2  \int_{-T}^0 \int_{\Omega_h} |v_h|^2 \, dt.
\end{multline*}
As $\tau h \leq 1$, applying Proposition \ref{PropDecompo} then immediately yields
\begin{multline}
	\label{Est-t=0-v}
	\tau^{1/2}\int_{\Omega_h} |\partial_t v_h(0)|^2  \leq C \int_{-T}^{T} \int_{\Omega_{h}} |\mathscr{L}_h v_h|^2 \, dt
		+
		C \tau\sum_{k=1,2} \int_{-T}^{T} \int_{\Gamma_{h,k}^+}  \left| \partial_{h,k}^- v_h\right|^2 \, dt
		\\
		+ C \tau h^2 \sum_{k=1,2} \int_{-T}^{T}  \int_{\Omega_{h,k}^-} |\partial_{h,k}^+\partial_t v_h|^2\, dt.
\end{multline}

Finally, for $w_h$ satisfying \eqref{Assumption-W}, we set $v_h : = e^{\tau\varphi_h} w_h$. Remarking that by construction $\mathscr{L}_h v_h  = e^{\tau\varphi_h} \Box_h w_h$, we can apply directly Proposition \ref{PropDecompo}. We notice that for $\tau h \leq 1$,
\begin{align*}
	& |w_h|^2 e^{2 \tau\varphi_h}  \leq C_\mu |v_h|^2,
	\\
	&
	|\partial_t w_h|^2 e^{2 \tau\varphi_h} \leq C_\mu ( |\partial_t v_h|^2  +  |v_h|^2), 
	\quad
	|\partial_{h,k}^{+} w_h|^2 e^{2 \tau\varphi_h} \leq C_\mu( |\partial_{h,k}^+ v_h|^2 + C_\mu \tau^2 |m_{h,k}^+ v_h|^2), 
	\\
	&
	|\partial_{h,k}^+\partial_t v_h|^2 \leq C_\mu |\partial_{h,k}^+ \partial_t w_h|^2 e^{2 \tau \varphi_h} + C_\mu \tau^2 (|\partial_{h,k}^+ w_h|^2+|m_{h,k}^+ \partial_t w|^2) e^{2 \tau \varphi_h} + C_\mu \tau^4 |m_{h,k}^+ w|^2 e^{2 \tau \varphi_h},
\end{align*}
and $|\partial_{h,k}^- v_h|^2 \leq C_\mu  |\partial_{h,k}^- w_h|^2 e^{2 \tau\varphi_h}$ on the boundary $\Gamma_{h,k}^+$ as $w_h$ vanishes on $\partial \Omega_h$. We thus deduce Carleman estimate \eqref{CarlemD} for $\tau$ large enough and $\tau h$ small enough directly from \eqref{decompo}. Besides, when $w_h(0) = 0$ on $\Omega_h$, then $v_h(0)=0$ and $\partial_t v_h(0) = \partial_t w_h(0) e^{\tau \varphi_h(0)}$ on $\Omega_h$, hence we conclude \eqref{Est-t=0} from \eqref{Est-t=0-v}. 
\end{proof}
\subsection{Proof of the discrete Carleman estimate - distributed case}\label{Sec-Carleman-Distributed-Proofs}
\begin{proof}[Proof of Theorem \ref{Thm-CarlemanDisc-Distributed}]
	It can be deduced from Theorem \ref{Thm-CarlemanDisc-Boundary}. Indeed, under assumption \eqref{Configuration-Lipschitz-distributed}, it suffices to define a cut-off function $\chi\in C^\infty(\overline\Omega;[0,1])$ taking value $1$ on $\Omega \setminus  \{ x \in \Omega, \, d(x, \Gamma_0) < \delta/2\} $ and vanishing on the boundary $\Gamma_+ =  (\{1\} \times (0,1)) \cup ((0,1) \times \{1\} )$ and to apply the Carleman estimate \eqref{CarlemD} to $\chi_h w_h$ with $\chi_h = \rh (\chi)$: the boundary terms in \eqref{CarlemD} vanish by construction but we have
	$$
		\Box_h (\chi_h w_h) = \chi_h \Box_h w_h  - 2 \nabla_h \chi_h \nabla_h w_h - \Delta_h \chi_h (2 m_h w_h - w_h). 
	$$
	Using that $\chi \equiv 1$ on $\Omega \setminus  \{ x \in \Omega, \, d(x, \Gamma_0) < \delta/2\}$, one easily checks that for $h$ small enough, $\partial_h \chi_h$ and $\Delta_h \chi_h$ are supported on $\omega$. We thus readily obtain
	\begin{multline}
		\label{Carlem-Bound-to-Dist}
		\tau \int_{-T}^{T} \int_{\Omega_{h}} e^{2\tau\varphi_h}\chi_h^2 | \partial_t w_h|^2 \, dt
		+
		 \tau\sum_{k=1,2} \int_{-T}^{T} \int_{\Omega_{h,k}^-} e^{2\tau\varphi_h} |\partial_{h,k}^+ (\chi_h w_h) |^2\, dt
		+ \tau^3\int_{-T}^{T}\int_{\Omega_{h}} e^{2\tau\varphi_h} \chi_h^2 | w_h |^2 \, dt
		\\
		\leq C\int_{-T}^{T}\int_{\Omega_{h}}e^{2\tau\varphi_h} \chi_h^2 | \Box_h w_h|^2 \, dt  
		+ 
		C\int_{-T}^{T}\int_{\omega_{h}} e^{2\tau\varphi_h}\left( | \nabla_h w_h |^2 + |m_h w_h|^2 + |w_h|^2 \right)\, dt
		\\
		+
		 C\tau h^2\sum_{k=1,2} \int_{-T}^{T}  \int_{\Omega_{h,k}^-} e^{2\tau\varphi_h} |\partial_{h,k}^+\partial_t (\chi_h w_h)|^2\, dt. 
	\end{multline}
	One then easily checks that, for $\tau h$ small enough, 
	\begin{multline*}
		\int_{-T}^{T}\int_{\omega_{h}} e^{2\tau\varphi_h}\left( | \nabla_h w_h |^2 + |m_h(w_h)|^2 + |w_h|^2 \right)\, dt + \tau h^2\sum_{k=1,2} \int_{-T}^{T}  \int_{\Omega_{h,k}^-} e^{2\tau\varphi_h} |\partial_{h,k}^+\partial_t (\chi_h w_h)|^2\, dt
		\\
		\leq
		C \sum_{k=1,2} \int_{-T}^{T} \int_{\omega_{h,k}^-} e^{2\tau\varphi_h} |\partial_{h,k}^+ w_h |^2\, dt
		+ C\int_{-T}^{T}\int_{\omega_{h}} e^{2\tau\varphi_h}| w_h |^2 \, dt 
		\\
		+C  \tau h^2 \int_{-T}^{T}  \int_{\Omega_{h}} e^{2\tau\varphi_h}  | \partial_t w_h|^2\, dt 
		+ C \tau h^2\sum_{k = 1,2} \int_{-T}^{T}  \int_{\Omega_{h,k}^-} 
			e^{2\tau\varphi_h} |\partial_{h,k}^+\partial_t w_h|^2\, dt.
	\end{multline*}
	We thus conclude \eqref{CarlemD-distributed} only by adding the terms 
	$$
		\tau \int_{-T}^{T} \int_{\omega_{h}} e^{2\tau\varphi_h}| \partial_t w_h|^2 \, dt
		+ \tau\sum_{k=1,2} \int_{-T}^{T} \int_{\omega_{h,k}^-} e^{2\tau\varphi_h} |\partial_{h,k}^+ w_h |^2\, dt
		+ \tau^3\int_{-T}^{T}\int_{\omega_{h}} e^{2\tau\varphi_h}| w_h |^2 \, dt
	$$
	on both sides of \eqref{Carlem-Bound-to-Dist} and by taking $\tau$ large enough.
\end{proof}

\subsection{Proof of the uniform Lipschitz stability result
}\label{CSIP}
As said in the introduction, Theorem \ref{TCWE1} is a consequence of the Carleman estimates in Theorems~\ref{Thm-CarlemanDisc-Boundary} and \ref{Thm-CarlemanDisc-Distributed}. Its statement is very similar to the one of \cite[Theorem 3.1]{BaudouinErvedoza11} in the 1-d case. With respect to the stability estimates obtained in the continuous case in \cite{Baudouin01} (see also \cite{ImYamIP01,BaudouinDeBuhanErvedoza}), there is the additional term \eqref{PenalizationTerm} which is remanent from \eqref{Additional-Term-Carl} corresponding to some non-standard penalization of the discrete inverse problems. 
\begin{proof}[Proof of Theorem \ref{TCWE1}] Let us begin with the identity
$$		
	\sum_{k=1,2} \int_{-T}^{T} \int_{\Gamma_{h,k}^+}  \left| \partial_{h,k}^- y_h[q^a_h] - \partial_{h,k}^- y_h[q^b_h]\right|^2 \, dt
		=  \norm{\partial_{\nu} \eh(y_h [q^a_h])-\partial_{\nu} \eh(y_h [q^b_h])}_{H^1(0,T;L^2(\Gamma_+))}^2, 
$$
that allows to end the proof of Theorem \ref{TCWE1} as soon as we obtain the stability estimate \eqref{UniformStability-3} with $\norm{\mathscr{M}_h[q_h^a] - \mathscr{M}_h[q_h^b]}_{H^1(0,T;L^2(\Gamma_0))} $ replaced by  
$$
	\left(\sum_{k=1,2} \int_{-T}^{T} \int_{\Gamma_{h,k}^+}  \left| \partial_{h,k}^- y_h[q^a_h] - \partial_{h,k}^- y_h[q^b_h]\right|^2 \, dt\right)^{1/2}.
$$
Since the proof follows the one of \cite[Theorem 3.1]{BaudouinErvedoza11}, we only sketch the main steps required.

 $\bullet$ \textit{Step 1. Energy estimates.} We first write classical energy estimates in the context of the semi-discrete wave equation in $\Omega_h$, like the one written in \cite[Lemma 3.3]{BaudouinErvedoza11}, and apply them to $z_h = \partial_t (y_h[q_h^b] - y_h[q_h^a]) $ that satisfies
$$
	\left\{
		\begin{array}{ll}
			\partial_{tt} z_h - \Delta_h z_h + q_h^b z_h= (q_h^b - q_h^a)\partial_t y_h[q_h^a], & \hbox{ in } (0,T) \times \Omega_h, 
			\\
			z_h =0, & \hbox{ on } (0, T) \times \partial \Omega_h, 
			\\
			(z_h(0), \partial_t z_h(0) ) = \big(0, (q_h^b - q_h^a) y_h^0\big), & \hbox{ in } \Omega_h.
		\end{array}
	\right.
$$
We thus get a constant $C = C(T,m)>0$ independent of $h$ and such that for all $t\in(0,T)$,
\begin{equation}
	\label{Energy-Uniform}
 	\norm{\partial_h^+ z_h(t)}_{L_h^2(\Omega_h^-)} +  \norm{\partial_t z_h(t)}_{L_h^2(\Omega_h)} +  \norm{z_h(t)}_{L^2_h(\Omega_h)} 
	\leq CK  \norm{q_h^a-q_h^b}_{L_h^2(\Omega_h)},
\end{equation}
where $ \norm{y_h[q_h^a]}_{H^1(0,T; L^\infty_h(\Omega_h))} \leq K$.

$\bullet$ \textit{Step 2. Choice of the Carleman weight.}
Since we assumed $T>\sqrt 2$, we can find $a >0 $ and $\beta \in (0,1)$ such that
$$
		\beta T^2 > \sup_{x \in \Omega}|x-x_a|^2 - \inf_{x \in \Omega} |x-x_a|^2 = 2 + 4a.
$$
Therefore, we can choose $\eta > 0$ such that the Carleman weight function $\psi $ defined in \eqref{varphi} satisfies
\begin{equation}
	\label{WeightBiggerAtT=0}
			 \sup_{ |t | \in (T-\eta, T),\, x \in \Omega} \psi(t,x) \leq \inf_{x \in \Omega} \psi(0,x).
\end{equation}
We then choose $a$ and $\beta$ as above in the Carleman weight \eqref{varphi}, and choose $\mu$, $\tau_0$, $\varepsilon>0$ such that Theorem \ref{Thm-CarlemanDisc-Boundary} holds.

$\bullet$ \textit{Step 3. Extension and truncation.} 
We extend the equation in $z_h$ on $(-T,T)$, setting $z_h(t)=-z_h(-t)$ for all $t \in (-T, 0)$. We also extend $\partial_t y_h[q_h^a]$ as an odd function on $(-T,T)$. We define the cut-off function $\chi\in C^{\infty}(\mathbb{R};[0,1])$ such that $\chi(\pm T) = \partial_t \chi (\pm T) = 0$ and  $\chi(t) =1$  for all $t\in [-T+\eta,T-\eta]$.
Then $~w_h=\chi z_h~$ fulfills the assumptions of Theorem~\ref{Thm-CarlemanDisc-Boundary} and satisfies the following equation:
$$
	\left\{\begin{array}{ll}
		\partial_{tt} w_h - \Delta_h w_h + q_h^b w_h= \partial_{tt} \chi z_h + 2 \partial_t \chi \partial_t z_h + (q_h^b- q_h^a)\partial_t y_h[q_h^a], & \hbox{ in } (-T,T) \times \Omega_h, 
			\\
			w_h =0, & \hbox{ on } (-T, T) \times \partial \Omega_h, 
			\\
			(w_h(0), \partial_t w_h(0) ) = (0, (q_h^a - q_h^b) y_h^0), & \hbox{ in } \Omega_h,
			\\
			w_h(\pm T)= 0, \quad\partial_t w_h(\pm T)= 0, &\hbox{ in } \Omega_h.
	\end{array}\right.
$$

$\bullet$ \textit{Step 4. Using the Carleman estimate.} We apply Carleman estimates \eqref{Est-t=0} and \eqref{CarlemD} to $w_h$ and, using the expression of $\partial_t w_h(0)$ and Assumption \eqref{RegDiscrete}, we get, for all $\tau \in(\tau_0, \varepsilon/h)$,
\begin{multline}
\label{Stab-Almost-Done}
	\sqrt{\tau} \int_{\Omega_h} e^{\tau \varphi_h(0)} |q_h^a-q_h^b|^2 + \tau^3 \int_{-T}^T  \int_{\Omega_{h}} e^{\tau \varphi_h} | w_h|^2 \, dt 
	\leq 
	C\int_{-T}^{T}\int_{\Omega_{h}}e^{2\tau\varphi_h}| \Box_h w_h|^2 \, dt  
	\\
	+ 
	C\tau\sum_{k=1,2} \int_{-T}^{T} \int_{\Gamma_{h,k}^+} e^{2\tau\varphi_h} \left| \partial_{h,k}^- w_h\right|^2 \, dt
	+
	 C\tau h^2\sum_{k=1,2} \int_{-T}^{T}  \int_{\Omega_{h,k}^-} e^{2\tau\varphi_h} |\partial_{h,k}^+\partial_t w_h|^2\, dt 
\end{multline}
The end of the proof finally consists in estimating the term containing $\Box_h w_h$:
\begin{multline}\label{multfin}
	\int_{-T}^T \int_{\Omega_{h}}e^{2\tau\varphi_h}| \Box_h w_h|^2 \, dt 
	\leq 
	C \int_{-T}^T \int_{\Omega_{h}}e^{2\tau\varphi_h}| q_h^b w_h|^2 \, dt 
	\\+
	C \int_{|t| \in (T-\eta, T)} \int_{\Omega_{h}}e^{2\tau\varphi_h}(|\partial_t z_h|^2 + |z_h|^2) \, dt  
	+
	C \int_{-T}^T \int_{\Omega_{h}}e^{2\tau\varphi_h}| (q_h^a - q_h^b)\partial_t y[q_h^a]|^2 \, dt. 
\end{multline}
The first term of the right hand side of \eqref{multfin} can be absorbed by the left hand-side of \eqref{Stab-Almost-Done} as $q_h^b$ is of bounded $L^\infty_h(\Omega_h)$-norm. 
In the second term, we bound the weight function by its supremum on $[T- \eta, T]$ and then use the energy bound \eqref{Energy-Uniform} 
on $z_h$. This can then be absorbed by the left hand-side of \eqref{Stab-Almost-Done} due to the comparison 
\eqref{WeightBiggerAtT=0} of the weight at time $0$ and on $(T- \eta, T)$. 
Finally, since the weight function is maximal at $t = 0$, the last term can be bounded by 
$C  \int_{\Omega_h} e^{2\tau \varphi_h(0)} |q_h^a-q_h^b|^2$ due to the assumption~\eqref{RegDiscrete} 
and thus it can also be absorbed by the left hand-side of \eqref{Stab-Almost-Done}. 
Therefore, taking $\tau$ large enough completes the proof of Theorem~\ref{TCWE1} 
in the case of a boundary observation \eqref{UniformStability-3}. 
The case of a distributed observation can be deduced similarly from Theorem~\ref{Thm-CarlemanDisc-Distributed} 
stating a Carleman estimate for a distributed observation.
\end{proof}

\section{Application of elliptic Carleman estimates}\label{AEC}
%
\subsection{Logarithmic stability estimate in the continuous case}\label{LogSE}

The goal of this section is to prove Theorem \ref{ThmBellassoued}. Actually, it is a direct consequence of the following result, similar to the ones in \cite{LebRob97,Phung09}:
\begin{theorem}\label{Thm-Est-By-FBI}
	Let $\Gamma_0$ be a non-empty open subset of $\partial \Omega$ and let $\omega$ be a smooth connected open subset of $\Omega$ such that $\partial \omega \cap \partial \Omega$ is an open neighborhood of $\Gamma_0$.
	Let $m>0$ and $q \in L^\infty(\Omega)$ satisfying $\norm{q}_{L^\infty}\leq m$.
	Let $\mathscr{D} >0$ and $R_0 >0$, and assume that $\zeta = \zeta(t,x)$ solves the wave equation 
		\begin{equation}
			\label{Eq-zeta-Thm}
			\left\{
				\begin{array}{ll}
					\partial_{tt} \zeta - \Delta \zeta + q \zeta = f,  &\qquad \hbox{ in } (-T,T) \times \Omega,  
					\\
					 \zeta = 0  &\qquad  \hbox{ on } (-T,T) \times \partial \Omega,
				\end{array}
			\right.
		\end{equation}
	for some $f \in L^{1}(-T,T; L^2(\Omega))$ satisfying
	\begin{equation}
		\label{Cond-Support-F}
		f = 0 \quad \text{ in }	(-T,T) \times \{x \in \Omega, \, d(x, \omega) < R_0\},
	\end{equation}
	and satisfies $\zeta \in H^2((-T,T) \times \Omega)$ with
	$
		\norm{\zeta}_{H^2((-T,T) \times \Omega)} \leq \mathscr{D}.
	$
	
	Let $\alpha >0$. There exists $T_0>0$ such that for any $T\geq T_0$, there exists a constant $C = C(T)>0$ such that 
	\begin{equation}
		\label{Est-on-Ocal-zeta}
		\norm{\zeta }_{H^1((-T/8,T/8)\times \omega)} 
		\leq
		C \mathscr{D}\left[\log\left( 2 + \frac{\mathscr{D}}{\norm{\partial_\nu \zeta}_{L^2((-T,T)\times \Gamma_0)}} \right)\right]^{-\frac{1}{1+\alpha}}.
	\end{equation}
\end{theorem}

Indeed, let us first show how Theorem \ref{Thm-Est-By-FBI} implies Theorem \ref{ThmBellassoued}. 
\begin{proof}[Proof of Theorem \ref{ThmBellassoued}]
	 The idea is to apply Theorem \ref{Thm-Est-By-FBI} to $\zeta = \partial_t (y[q^a] - y [q^b])$, which satisfies the wave equation
		\begin{equation}
			\label{Eq-zeta}
			\left\{
				\begin{array}{ll}
					\partial_{tt} \zeta - \Delta \zeta + q^b \zeta = (q^b - q^a) \partial_t y[q^a],  & \qquad(t,x) \in (0,T) \times \Omega, 
					\\
					 \zeta = 0  &\qquad (t,x) \in (0,T) \times \partial \Omega,
					 \\
					 \zeta(0,x) = 0 , \partial_t \zeta(0,x) =(q^b-q^a)(x)y^0(x), &\qquad x \in  \Omega.
				\end{array}
			\right.
		\end{equation}
	Extending $\zeta$ as an odd function on $(-T, T)$, using the classical energy estimates on $\partial_t \zeta$, the fact that $\partial_t \zeta$ is continuous at $t = 0$ by construction, and recalling assumption \eqref{Error-W-1-2-+} on $q^a-q^b$, we easily get:
		\begin{align}
			\norm{ \zeta }_{H^2((-T,T) \times \Omega)} 
			\nonumber
			& \leq  C_m \left( \norm{(q^a-q^b) y^0}_{H^1_0(\Omega)}  
					+ \norm{(q^a -q^b) y^1}_{L^2(\Omega)} 	
					+ \norm{(q^a -q^b) \partial_{t} y[q^a]}_{W^{1,1}(0,T;L^2(\Omega))} \right) 
			\nonumber
			\\
			& \leq  C_m m \left(\norm{y^0}_{H^1(\Omega)} + \norm{y^1}_{L^2(\Omega)} + \norm{\partial_t y[q^a]}_{W^{1,1}(0,T;L^2(\Omega))} \right)+ C_m M \norm{y^0}_{L^\infty(\Omega)}
			\nonumber
			\\ 
			& \leq C_m (m+M)  \norm{y[q^a]}_{W^{2,1}(0,T;L^2(\Omega))\cap H^1(0,T;L^\infty(\Omega))}
			+ C_m m \norm{y^0}_{H^1(\Omega)} = \mathscr{D}. 
		\end{align}
	Since the potentials $q^a$ and $q^b$ coincide on $\Ocal$ by \eqref{Class-Pot-Bellassoued}, and because of \eqref{Ocal-Condition}, 
	the source term $f=(q^a - q^b) \partial_t y[q^a]$ extended to an odd function on $(-T,0)$, satisfies \eqref{Cond-Support-F} for $R_0 = \delta/2$ 
	and $\omega =  \{ x \in \Omega, d(x, \Gamma_1) < \delta/2 \} $. Applying Theorem~\ref{Thm-Est-By-FBI}, we obtain:
	\begin{equation*}
		\label{Est-on-Ocal}
		\norm{\partial_t y[q^a] -\partial_t y[q^b] }_{H^1((-T/8,T/8)\times \omega)} 
		\leq
		\mathscr{D}\left[\log\left( 2 + \frac{\mathscr{D}}{\norm{\partial_\nu \partial_t y[q^a] - \partial_\nu \partial_t  y[q^b]}_{L^2((-T,T)\times \Gamma_0)}} \right)\right]^{-\frac{1}{1+\alpha}}.
	\end{equation*}
	Because $\omega = \{ x \in \Omega,\, d(x, \Gamma_1) < \delta/2\} $ satisfies the condition \eqref{Ocal-Condition} and is thus a neighborhood of a boundary satisfying the Gamma-condition \eqref{Gamma-Condition}, the use of estimate \eqref{Stab-Lucie-Interne} of Theorem \ref{Thm-Lucie} then completes the proof of Theorem \ref{ThmBellassoued}.
\end{proof} 
Let us now focus on the proof of Theorem \ref{Thm-Est-By-FBI}. As we said in the introduction, this result follows from a suitable use of a Fourier-Bros-Iagoniltzer (FBI) transform to reduce the hyperbolic problem to an elliptic problem and on an elliptic Carleman estimate.\\
As in \cite{LebRob97,Phung09}, we use a FBI transform with a ``Gaussian-polynomial'' kernel: this ingredient allows us to improve the exponent in \eqref{Est-on-Ocal-zeta} to any $\alpha >0$ instead of only $\alpha = 1$ as in \cite{BellassouedIP04}. 
\\
Also, our proof shortcuts the one in \cite{Phung09} by using a global Carleman estimate for the elliptic equation, allowing to get rid of the iterated three spheres inequalities in \cite{Phung09} (see also \cite{BellassouedIP04}). Though this does not yield any particular improvement on the result in the continuous setting, we will follow the same strategy in the semi-discrete case and that way, we will manage to avoid the iterated use of  three spheres inequalities in the discrete setting, which would induce tedious discussions.
\begin{proof}[Proof of Theorem \ref{Thm-Est-By-FBI}]
	The proof is rather long and can be split into several steps. Along this proof, the constants written in large caps may depend on the parameter $n \in \N$ and $T>0$ and are independent of the other parameters. But constants with small caps, that will be numbered $c_0$, $c_1$, ($\dots$) have the additional property that they do not depend on the time parameter $T$ either.

$\bullet$ {\it Step 1. The Fourier Bros Iagoniltzer kernel.}
		In this step, we introduce the FBI kernel following \cite[p.473]{LebRob97}.
		Let us set $n\in \N^*$ such that $1/({2n-1}) < \alpha$ and  
		$
			\gamma = 1-1/(2n)
		$
		(that guarantees 
		$1/(1+ \alpha) <  \gamma  < 1$).
		Introduce a function $F$ defined on $\C$ as follows:
		\begin{equation}
			\label{F-kernel-Bel}
			F(z) = \frac{1}{2\pi} \int_{-\infty}^\infty e^{ i z \xi} e^{- \xi^{2n}} \, d\xi.
		\end{equation}
		According to \cite{LebRob97}, this function $F$ is even, holomorphic on $\C$ and satisfies, for some positive constants $C_0$, $c_0$, $c_1$, $c_2$: 
		\begin{equation}
			\label{Estimate-F-Bel}
			\left\{
				\begin{array}{ll}
					|F(z)| + |F'(z)| \leq C_0 \exp\left(c_0 |\Im(z)|^{1/\gamma}\right) , \quad & \forall z \in \C,
					\\
					|F(z)| \leq C_0 \exp\left(- c_1 |z|^{1/\gamma} \right), \quad &\forall z \in \C \hbox{ with } |\Im(z)|\leq c_2 |\Re(z)|,
				\end{array}	
			\right.
		\end{equation}
	 	Then, for $\lambda \geq 1$, we introduce
		$$
			F_\lambda(z) = \lambda^{\gamma} F(\lambda^{\gamma} z),
		$$
		which, due to \eqref{Estimate-F-Bel}, satisfies the following estimates:
		\begin{equation}
			\label{Estimate-F-lambda-Bel}
			\left\{
				\begin{array}{ll}
					|F_\lambda(z)| + |F_\lambda'(z)| \leq C_0 \lambda^{2\gamma} \exp\left(c_0 \lambda |\Im(z)|^{1/\gamma}\right) , \quad & \forall z \in \C,
					\\
					|F_\lambda (z)| \leq C_0 \lambda^{\gamma} \exp\left(- c_1 \lambda |z|^{1/\gamma} \right), \quad &\forall z \in \C \hbox{ with } |\Im(z)|\leq c_2 |\Re(z)|. 
				\end{array}	
			\right. 
		\end{equation}
	 	Let us remark that $F$ defined by \eqref{F-kernel-Bel} is the inverse Fourier transform of $\xi \mapsto e^{- \xi^{2n}}$ so that $F_\lambda$ is an approximation of the identity as $\lambda \to \infty$. 
		Finally, notice that by construction, the Fourier transform of $F_\lambda(t)$ is 
		\begin{equation}\label{Fourier-F-lambda}
			\mathcal{F} (F_\lambda) (\xi ) = \mathcal{F}(F) \left( \frac{\xi}{\lambda^{\gamma}} \right) = \exp \left( - \left( \frac{\xi}{\lambda^{\gamma}} \right)^{2n} \right).
		\end{equation}

$\bullet$ {\it Step 2. The Fourier-Bros-Iagoniltzer transform.}
		Let $\zeta$ be the solution of \eqref{Eq-zeta-Thm}.  We introduce a cut-off function $\eta \in C^\infty([-T,T];[0,1])$ such that
		$$
			\eta (t) = 
			\left\{
				\begin{array}{ll}
					1 \quad & \hbox{ if } |t| \leq T/2, 
					\\
					0 \quad & \hbox{ if } |t| \geq 3T/4.
				\end{array}
			\right.
		$$
	We define the FBI transform of $\zeta$ for $s \in \R$, $a \in [-T/4, T/4]$ and $x \in \Omega$ by
		\begin{equation}
			\label{FBI-transform-Bel}
			v_{a,\lambda} (s,x) = \int_\R F_\lambda (a + {\bf i}s - t) \eta(t) \zeta(t,x) \, dt, 
		\end{equation}
	where ${\bf i}$ denotes the imaginary unit.	 Since 
		$
			\partial_s v_{a, \lambda} (s,x) = {\bf i} \int_\R F_\lambda (a + {\bf i} s - t)\, \partial_t (\eta(t) \zeta(t,x) )\, dt,
		$
		using integration by parts, one easily checks that $v_{a, \lambda}$ solves the elliptic equation
		$$
			\left\{
				\begin{array}{ll}
					(-\partial_{ss} - \Delta_x +q ) v_{a,\lambda} = f_{a, \lambda} \quad &\hbox{ in } \R \times \Omega, 
					\\
					v_{a, \lambda} = 0 \quad & \hbox{ on } \R \times \partial\Omega,
				\end{array}
			\right.
		$$
		where $f_{a, \lambda}$ is defined as $f_{a, \lambda} = f_{a, \lambda,1} + f_{a,\lambda,2}$, with (since $\zeta$ satisfies \eqref{Eq-zeta-Thm})
		\begin{eqnarray*}
	f_{a, \lambda,1} (s,x) &=& \displaystyle\int_\R F_\lambda (a + {\bf i} s - t) \left(2 \eta'(t) \partial_t \zeta (t,x)+ \eta''(t) \zeta(t,x)\right) \, dt,
					\\	
	\quad f_{a, \lambda,2} (s,x) &=& \displaystyle \int_\R F_\lambda (a + {\bf i} s - t) \eta(t) f(t,x) \, dt.
		\end{eqnarray*}
		
		On the one hand, using that $2  \eta' \partial_t \zeta + \eta'' \zeta$ is supported in $\{(t,x) \in (-T,T) \times \Omega \hbox{ s.t. } |t| \geq T/2 \}$ and 
		the second estimate in \eqref{Estimate-F-lambda-Bel} on the kernel $F_\lambda$, we have
		\begin{equation}
			\label{Est-SourceTerm-Bel-1}
				\norm{f_{a, \lambda,1}}^2_{L^\infty(-3, 3;L^2 (\Omega))} 
				\leq 
				C \lambda^{2\gamma} e^{ -2c_1 \lambda  (T/2)^{1/\gamma}} \norm{\zeta}_{H^1((-T,T)\times \Omega)}^2
				\leq
				C \lambda^{2\gamma} e^{ -2c_1 \lambda  (T/2)^{1/\gamma}} \mathscr{D}^2,
		\end{equation}
		for any $T > 12/c_2$, since $a \in [-T/4,T/4]$, $|t|\geq T/2$ and since we decided to work for $s\in [-3,3]$ and needed $|s| \leq c_2 |a-t|$ to apply \eqref{Estimate-F-lambda-Bel}.
		
		On the other hand, the first estimate in \eqref{Estimate-F-lambda-Bel} also yields, for $c_3 = 2\cdot 3^{1/\gamma} c_0$,
		\begin{equation}
			\label{Est-v-a-lambda-Bel}
			\norm{v_{a,\lambda}}_{H^1((-3, 3 )\times \Omega)}^2 
			\leq 
			C \lambda^{4\gamma} e^{c_3 \lambda} \norm{\zeta}_{H^1((-T,T) \times \Omega)}^2
			\leq
			C \lambda^{4\gamma} e^{c_3 \lambda} \mathscr{D}^2,
		\end{equation}
		and, similarly, 
		\begin{equation}
			\label{ObsTerm-Bel}
				\norm{ \partial_\nu v_{a,\lambda}}_{L^2((-3,3) \times \Gamma_0)}^2
				\leq 
				C  \lambda^{4\gamma} e^{c_3 \lambda } \norm{\partial_\nu  \zeta}_{L^2((-T,T )\times \Gamma_0)}^2.
		\end{equation}
	
$\bullet$ {\it Step 3. Estimating $v_{a,\lambda}$ by an observation on $(-3, 3) \times \Gamma_0$.} This step strongly relies on a Carleman estimate for the following elliptic problem:
		\begin{equation}
			\label{EllipticEq-Bel}
				\left\{
					\begin{array}{ll}
						(- \partial_{ss} - \Delta_x + q ) w = g  &\qquad \hbox{ in } (-3,3) \times \Omega, 
						\\
						w = 0 \quad &\qquad \hbox{ on } \partial ((-3, 3) \times \Omega).
					\end{array}
				\right.
		\end{equation}
		One of the most important points is to suitably choose the Carleman weight. First construct a smooth function $\psi_0 = \psi_0(x)$ on $\overline{\omega}$ such that 
		\begin{equation}
			\label{Psi-0}
			\left\{
				\begin{array}{l}
					\forall x \in \overline{\omega},\, \psi_0(x) \geq 0,
					\\
					\inf_{\overline\omega} \{|\nabla \psi_0| \} >0,
					\\
					\forall x \in \partial \omega \setminus \Gamma_0, \,
						\psi_0(x) = 0 \hbox{ and } \partial_\nu \psi_0(x) < 0,
					\\
					\norm{\psi_0}_{L^\infty(\omega)} \leq 1/2.
				\end{array}
			\right.
		\end{equation}
		Note that such a function $\psi_0$ exists according to the construction in \cite{FursikovImanuvilov} (see also \cite[Appendix III]{TWbook}). 
		We then extend this function $\psi_0$ as a smooth function $\psi$ on $\overline{\Omega}$ satisfying 
		$\norm{\psi}_{L^\infty(\Omega)} \leq 1$. By continuity, there exists a positive constant  $R \in (0, R_0)$ such that in the set 
		$$
				\omega_R = \{ x \in \Omega, \, d(x, \omega) < R\},
		$$
		where the source term $f$ vanishes by assumption \eqref{Cond-Support-F},
		we have 
		$
			\inf_{x\in \overline{\omega_R}} \{|\nabla \psi(x)| \} >0 	
		$
		and such that in the set 
		$$
			\mathscr{C} = \left\{ x \in \Omega, \  \frac{R}{2} < d(x, \omega) < R \right\},
		$$
		we have, as pictured in Figure~\ref{fig:weight},
		\begin{equation}
			\label{ConditionPsi-Bel}
			0 = \inf_{\overline\omega} \psi > \sup_{\overline{\mathscr{C}}} \psi.
		\end{equation}		
\begin{figure}[h]
\begin{center}
\scalebox{0.8}{\input{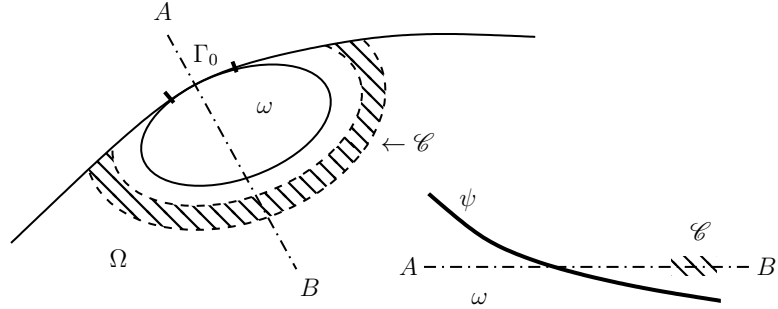}}
\end{center}
\caption{Construction of the weight function $\psi(x)$.
}
\label{fig:weight}
\end{figure}
		
		\noindent We finally define, for $\mu \geq 1$, 
		\begin{equation}
			\label{DefPhi-Bel}
			\varphi := \varphi(s,x) = \exp( \mu (\psi(x) - s^2)), \qquad (s,x) \in [-3, 3] \times \overline{\Omega}.
		\end{equation}
		According to \cite{Hormander-III} (see also \cite{FursikovImanuvilov,Robbiano95}) one has the following Carleman estimate  for \eqref{EllipticEq-Bel}:
		\begin{lemma}[An elliptic Carleman estimate]\label{Lemma-Carleman} 
			There exist $\mu \geq 1$ and a constant $C>0$ such that for all $\tau \geq 1$, for all $g\in L^2((-3, 3)\times \Omega)$ and $w$ solution of \eqref{EllipticEq-Bel} supported in $(-3, 3) \times \omega_R$, 
			\begin{multline}
				\label{Carleman-Ass-Bel}
					\tau^3 \norm{e^{\tau \varphi} w}_{L^2((-3, 3) \times \Omega)}^2 + \tau \norm{e^{\tau \varphi} \nabla_{s,x} w}_{L^2((-3, 3) \times \Omega)}^2 
					\\
					\leq C \norm{e^{ \tau \varphi} g}_{L^2((-3, 3) \times \Omega)}^2  + C \tau \norm{e^{\tau \varphi} \partial_\nu w}_{L^2((-3,3) \times \Gamma_0)}^2,
			\end{multline}
			where the constant $C$ can be taken uniformly with respect to $q \in L^\infty (\Omega)$ with $\norm{q}_{L^\infty} \leq m$.
		\end{lemma}
		Estimate \eqref{Carleman-Ass-Bel} has to be understood as a Carleman estimate with observation on $(-3,3) \times \Gamma_0$ and in $(-3,3) \times (\Omega \setminus \omega_R)$. But, as we assumed that $w$ is supported in $(-3, 3) \times \omega_R$, we simply omit the observation in $(-3,3) \times (\Omega \setminus \omega_R)$.
		
		Now,  introduce smooth cut-off functions $\chi_S = \chi_S(s)$ and $\chi_R = \chi_R(x)$ such that
		$$
			\chi_S (s) = 
			\left\{
				\begin{array}{ll}
					1 \quad & \hbox{ if } |s| \leq 2, 
					\\
					0 \quad & \hbox{ if } |s| \geq 3, 
				\end{array}
			\right.
			\quad \hbox{ and } 
					\ds \norm{\chi_S}_{W^{2, \infty}(\R)} \leq C,
		$$
		and 
		$$
			\chi_R (x) = 
			\left\{
				\begin{array}{ll}
					1 \quad & \hbox{ if } d(x, \omega) \leq R/2, 
					\\
					0 \quad & \hbox{ if } d(x, \omega) \geq R, 
				\end{array}
			\right.
			\quad \hbox{ and } 
					\ds \norm{\chi_R}_{W^{2, \infty}(\Omega)} \leq C.
		$$
		We can then define 
		\begin{equation}
			\label{Truncated-V-a-lambda-Bel}
			w_{a,\lambda} (s ,x ) = \chi_S(s) \chi_R(x) v_{a, \lambda}(s,x), \quad (s,x) \in \R \times \Omega
		\end{equation}
		which satisfies
		\begin{equation}
			\label{EllipticEq-w-a-lambda-Bel}
				\left\{
					\begin{array}{ll}
						(- \partial_{ss} - \Delta_x + q) w_{a,\lambda} = g_{a, \lambda} \quad &\hbox{ in } (-3,3) \times \Omega, 
						\\
						w_{a, \lambda} = 0 \quad & \hbox{ on }  \partial ((-3, 3) \times \Omega),
					\end{array}
				\right.
		\end{equation}
		where (using the fact that $f_{a, \lambda,2}$ vanishes in $\omega_R$  by assumption \eqref{Cond-Support-F})
		$$
			 g_{a, \lambda} = \chi_S \chi_R f_{a, \lambda,1} 
			- 2 \chi_R \partial_s \chi_S \partial_s v_{a,\lambda} 
			-\chi_R \partial_{ss} \chi_S v_{a, \lambda} 
			- 2 \chi_S \nabla \chi_R \nabla v_{a, \lambda} 
			- \chi_S \Delta \chi_R v_{a, \lambda}.
		$$	
		 Thus, Carleman estimate \eqref{Carleman-Ass-Bel} can be applied, and gives: for all $\tau \geq 1$,
		 \begin{multline*}
			\tau^3 \norm{e^{\tau \varphi} w_{a, \lambda}}_{L^2((-3, 3) \times \Omega)}^2 + \tau \norm{e^{\tau \varphi} \nabla_{s,x} w_{a, \lambda}}_{L^2((-3, 3) \times \Omega)}^2 
			\\
			\leq C \norm{e^{ \tau \varphi} g_{a,\lambda}}_{L^2((-3, 3) \times \Omega)}^2  + C \tau \norm{e^{\tau \varphi} \partial_\nu w_{a,\lambda}}_{L^2((-3,3) \times \Gamma_0)}^2.
		\end{multline*}
		Since $w_{a, \lambda} = v_{a,\lambda}$ on $(-1,1) \times \omega$ and $\norm{\chi_S \chi_R}_{W^{2,\infty}(\R \times \Omega)} \leq C$, we obtain
		 \begin{multline}
			\label{Carleman-Use-1-Bel}
			\tau^3 \norm{e^{\tau \varphi} v_{a, \lambda}}_{L^2((-1,1) \times \omega)}^2 + \tau \norm{e^{\tau \varphi} \nabla_{s,x} v_{a, \lambda}}_{L^2((-1,1) \times \omega))}^2 
			\\
			\leq C \norm{e^{ \tau \varphi} g_{a, \lambda}}_{L^2((-3, 3) \times \Omega)}^2  + C \tau \norm{e^{\tau \varphi} \partial_\nu v_{a,\lambda}}_{L^2((-3,3) \times \Gamma)}^2.
		\end{multline}

		Now, we estimate from below the left hand side and from above the right hand side of \eqref{Carleman-Use-1-Bel}. Notice first that according to \eqref{ConditionPsi-Bel}, we can choose $\epsilon_0 \in (0,1)$ such that 
		\begin{equation}
			\label{Cond-Ocal-C}
			\inf_{|s| \leq \epsilon_0,\, x \in \omega }  \varphi >  \sup_{|s| \leq 3,\ x \in \mathscr{C}} \varphi.
		\end{equation}
		In order to simplify notations, we set
		\begin{equation}
			\label{Notations-Sup-Inf-Bel}
			 \mathscr{I}_{\omega} = \inf_{|s| \leq \epsilon_0,\, x \in \omega }  \varphi, 
			 \quad 
			 \mathscr{S} = \sup_{|s| \leq 3,\, x \in \Omega} \varphi,
			 \quad 
			 \mathscr{S}_{(2,3)} = \sup_{|s| \in (2,3),\, x \in \Omega} \varphi,
			 \quad
			 \mathscr{S}_{\mathscr{C}} = \sup_{|s| \leq 3,\ x \in \mathscr{C}} \varphi.
		\end{equation}
		Remark that, similarly to \eqref{Cond-Ocal-C}, that writes now $ \mathscr{I}_{\omega} >  \mathscr{S}_{\mathscr{C}}$, using the explicit form of $\varphi$ and the fact that $\norm{\psi}_{L^\infty(\Omega)} \leq 1$,  
		we have
		\begin{equation}
			\label{Cond-Ocal-s}
			 \mathscr{I}_{\omega}> \mathscr{S}_{(2,3)}.
		\end{equation}
		
		Going back to \eqref{Carleman-Use-1-Bel}, on the one hand, for all $\tau \geq 1$, the left hand side satisfies,
		 \begin{equation}
			\label{Est-from-below-Bel}
			e^{2 \tau \mathscr{I}_{\omega} } \norm{v_{a, \lambda}}_{H^1((-\epsilon_0,\epsilon_0) \times \omega))}^2
			\leq 
			\tau^3 \norm{e^{\tau \varphi} v_{a, \lambda}}_{L^2((-1,1) \times \omega)}^2 + \tau \norm{e^{\tau \varphi} \nabla_{s,x} v_{a, \lambda}}_{L^2((-1,1) \times \omega)}^2.
		\end{equation}
		On the other hand, the first term of the right hand side in \eqref{Carleman-Use-1-Bel} can be estimated from above:
		\begin{equation}
			\label{Est-from-above-0-Bel}
			\norm{e^{ \tau \varphi} g_{a, \lambda}}_{L^2((-3, 3) \times \Omega)}^2
			\leq 
			e^{2 \tau \mathscr{S} } \norm{f_{a,\lambda,1}}_{L^2((-3, 3 )\times \Omega)}^2
			+ C \left(e^{2 \tau \mathscr{S}_{(2,3)} }  + e^{2 \tau \mathscr{S}_{\mathscr{C}}} \right)\norm{v_{a,\lambda}}_{H^1((-3, 3 )\times \Omega)}^2
		\end{equation}
		since $\partial_s \chi_S, \, \partial_{ss} \chi_S$ are supported in $\{s\in\mathbb R, \hbox { s.t. } |s| \in (2,3)\}$ and $\nabla \chi_R, \,  \Delta \chi_R$ are supported in $\mathscr{C}$.
		Plugging \eqref{Est-SourceTerm-Bel-1} and \eqref{Est-v-a-lambda-Bel} into \eqref{Est-from-above-0-Bel}, we obtain
		\begin{equation}
			\label{Est-from-above-Bel}
			\norm{e^{ \tau \varphi} g_{a, \lambda}}_{L^2((-3, 3) \times \Omega)}^2
			\leq 
			C  e^{2 \tau \mathscr{S} } \lambda^{2\gamma} e^{ -2 c_1 \lambda  (T/2)^{1/\gamma}} \mathscr{D}^2
			+ 
			C \left(e^{2 \tau \mathscr{S}_{(2,3)} }  + e^{2 \tau \mathscr{S}_{\mathscr{C}}} \right)	\lambda^{4\gamma} e^{c_3 \lambda} \mathscr{D}^2.
		\end{equation}
		Combining now estimates \eqref{Carleman-Use-1-Bel} with \eqref{Est-from-below-Bel}, \eqref{ObsTerm-Bel} and \eqref{Est-from-above-Bel}, we get
		\begin{multline}
			\label{Est-v-a-lambda-obs-Bel}
			e^{2 \tau \mathscr{I}_{\omega} } \norm{v_{a, \lambda}}_{H^1((-\epsilon_0,\epsilon_0) \times \omega)}^2
			\leq
			C  e^{2 \tau \mathscr{S} } \lambda^{2\gamma} e^{ -2 c_1 \lambda  (T/2)^{1/\gamma}} \mathscr{D}^2
			\\
			+ 
			C \left(e^{2 \tau \mathscr{S}_{(2,3)} }  + e^{2 \tau \mathscr{S}_{\mathscr{C}}} \right)	\lambda^{4\gamma} e^{c_3 \lambda} \mathscr{D}^2
			+ 
			C \tau e^{ 2 \tau \mathscr{S} } \lambda^{4\gamma} e^{c_3 \lambda } \norm{\partial_\nu  \zeta}_{L^2((-T,T )\times \Gamma_0)}^2.
		\end{multline}

$\bullet$ {\it Step 4. Estimating $\zeta$ from its FBI transform $v_{a, \lambda}$.}
		Writing $\zeta$ as follows,
		$$
			\zeta(t,x ) = \zeta(t,x) - v_{t, \lambda}(0,x) + v_{t,\lambda}(0,x),
		$$
		we obtain that, for $t \in (-T/8, T/8)$, 
		\begin{multline}
			\label{Est-y-v-a-lambda-Bel}
			\norm{\zeta}_{L^2((-T/8, T/8)\times \omega)} \leq
			 \norm{(t,x) \mapsto \zeta(t,x) - v_{t, \lambda}(0,x)}_{L^2((-T/8, T/8)\times \omega)}
			 \\
			 +  \norm{(a,x) \mapsto v_{a, \lambda}(0,x)}_{L^2((-T/8, T/8)\times \omega)}.
		\end{multline}
		As already detailed in  \cite{Phung09}, since $v_{t,\lambda} (0,x) = F_\lambda \star (\eta \zeta) (t)$, where the convolution is only in  the time variable, 
		we obtain, from \eqref{Fourier-F-lambda} the following estimate (notice $\eta=1$ in $(-T/8,T/8)$):
		\begin{align*}
			 \lefteqn{\norm{(t,x) \mapsto\zeta(t,x) - v_{t, \lambda}(0,x)}_{L^2((-T/8, T/8)\times \omega)}
			 =  
			 \norm{\eta \zeta-F_{\lambda}\star (\eta \zeta)}_{L^2((-T/8, T/8)\times \omega)}}
			 \\
			&  \leq  
			 \norm{ \left(1 - \mathcal{F}(F_\lambda) \right) \mathcal{F}(\eta \zeta)}_{L^2(\mathbb R\times \omega)}
			  \leq 
			  \norm{(\xi,x) \mapsto \frac{|\xi|}{\lambda^{\gamma}} |\mathcal{F}(\eta \zeta)(\xi,x)|}_{L^2(\mathbb R\times \omega)}
			 \\
			  & \leq 
			 \frac{C}{\lambda^{\gamma}} \norm{\eta \zeta}_{H^1(\mathbb R \times \omega)}
			 ~\leq ~\frac{C}{\lambda^{\gamma}} \norm{\zeta}_{H^1((-T,T) \times \omega)}.
		\end{align*}
		
		Besides, since $F_\lambda$ is holomorphic,  the map $a+ {\bf i} s \mapsto v_{a, \lambda}(s,x)$ is holomorphic in the variable $a+ {\bf i} s$ for all $\lambda$ and $x$, and the Cauchy formula implies that (see appendix of \cite{BellassouedIP04}, for some details)
		$$
			\norm{(a,x) \mapsto v_{a, \lambda}(0,x)}_{L^2((-T/8, T/8)\times \omega)} \leq C \sup_{a \in (-T/4, T/4)} \norm{v_{a, \lambda}}_{L^2((-\epsilon_0,\epsilon_0) \times \omega)}.
		$$
		Hence, from \eqref{Est-y-v-a-lambda-Bel}, combining the above estimates we get
		$$
			\norm{\zeta}_{L^2((-T/8, T/8)\times \omega)} \leq \frac{C}{\lambda^{\gamma}} \norm{\zeta}_{H^1((-T,T) \times \omega)}
			 + C \sup_{a \in (-T/4, T/4)} \norm{v_{a, \lambda}}_{L^2((-\epsilon_0,\epsilon_0) \times \omega)}.
		$$
		Having an estimate on $v_{a,\lambda}$ in $H^1((-\epsilon_0, \epsilon_0) \times \omega)$ at our disposal, we can apply the latter to $\partial_t \zeta$ and  $\nabla \zeta$ and obtain
		\begin{eqnarray}
			\norm{\zeta}_{H^1((-T/8, T/8)\times \omega)} 
			& \leq & \frac{C}{\lambda^{\gamma}} \norm{\zeta}_{H^2((-T,T) \times \omega)}
			 + C \sup_{a \in (-T/4, T/4)} \norm{v_{a, \lambda}}_{H^1((-\epsilon_0,\epsilon_0) \times \omega)}
			 \notag
			 \\
			 & \leq &
			 \frac{C}{\lambda^{\gamma}} \mathscr{D}
			 + C \sup_{a \in (-T/4, T/4)} \norm{v_{a, \lambda}}_{H^1((-\epsilon_0,\epsilon_0) \times \omega)}.
			\label{Link-y-v-a-lambda-Bel}			 
		\end{eqnarray}
			
$\bullet$ {\it Step 5. Concluding step.} Combining estimates \eqref{Est-v-a-lambda-obs-Bel} and \eqref{Link-y-v-a-lambda-Bel}, we have shown that for all $\lambda \geq 1$ and $\tau \geq 1$,
		\begin{multline}
			\label{Est-v-a-lambda-obs-f}
			\norm{\zeta}_{H^1((-T/8, T/8)\times \omega)}^2 
			\leq 
			\frac{C}{\lambda^{2\gamma}} \mathscr{D}^2
			+
			C  e^{2 \tau (\mathscr{S}- \mathscr{I}_{\omega} )}  \lambda^{2\gamma} e^{ -2 c_1 \lambda  (T /2)^{1/\gamma}} \mathscr{D}^2
			\\
			\hspace{-1ex}
			+ 
			C e^{ -2 \tau \mathscr{I}_{\omega}} \left(e^{2 \tau \mathscr{S}_{(2,3)} }  + e^{2 \tau \mathscr{S}_{\mathscr{C}}} \right)	\lambda^{4\gamma} e^{c_3 \lambda} \mathscr{D}^2
			+ 
			C \tau e^{ 2 \tau (\mathscr{S}- \mathscr{I}_{\omega}) } \lambda^{4\gamma} e^{c_3 \lambda } \norm{\partial_\nu  \zeta}_{L^2((-T,T )\times \Gamma_0)}^2.
		\end{multline}
		Recalling \eqref{Cond-Ocal-C} and \eqref{Cond-Ocal-s},
		we can choose the Carleman parameter $\tau$ as a linear function of the FBI parameter $\lambda$ by setting
		\begin{equation}
			\label{Tau(Lambda)}
			\tau  = \frac{ c_3 \lambda }{\mathscr{I}_{\omega}-\max\{ \mathscr{S}_{\mathscr{C}}, \mathscr{S}_{(2,3)} \} }. 
		\end{equation}
		With this choice, one should assume $\lambda \geq \lambda_*$, where
		$
			\lambda_* = \frac{1}{c_3} \left( \mathscr{I}_{\omega}-\max\{ \mathscr{S}_{\mathscr{C}}, \mathscr{S}_{(2,3)} \}\right), 
		$
		in order to guarantee \eqref{Est-v-a-lambda-obs-f} (since $\tau \ge 1$).
		Thereby, there exist positive constants $c_4, c_5, c_6$ such that for all $\lambda \geq \lambda_*$,
		\begin{eqnarray*}
			 e^{ -2 \tau \mathscr{I}_{\omega}} \left(e^{2 \tau \mathscr{S}_{(2,3)} }  + e^{2 \tau \mathscr{S}_{\mathscr{C}}} \right)	\lambda^{4\gamma} e^{c_3 \lambda}
			 & \leq &
			C e^{- c_4 \lambda},
			\\
			e^{2 \tau (\mathscr{S} - \mathscr{I}_{\omega} ) } \lambda^{2\gamma} e^{ -2 c_1 \lambda  (T /2)^{1/\gamma}}
			& \leq &
			C e^{\lambda (c_5 - 2 c_1 (T /2)^{1/\gamma})},
			\\
			 \tau e^{ 2 \tau (\mathscr{S} - \mathscr{I}_{\omega} ) } \lambda^{4\gamma} e^{c_3 \lambda }
			 & \leq & 
			 C e^{c_6 \lambda}.
		\end{eqnarray*}
		Obviously, there exists $T_0>0$ such that for all $T\geq T_0$,	$c_5 \leq c_1 (T /2)^{1/\gamma}$. Thus, estimate \eqref{Est-v-a-lambda-obs-f}  yields, for all $T\geq T_0$ and $\lambda \geq \lambda_*$,
		\begin{equation*}\label{AlmostOver-0}
			\norm{\zeta}_{H^1((-T/8, T/8)\times \omega)}^2 
			\leq 
			C  \mathscr{D}^2
			\left(\frac{1}{\lambda^{2\gamma}} +  e^{-c_4 \lambda} + e^{- \lambda c_1 (T /2)^{1/\gamma} } \right)
			+ 
			C  e^{c_6 \lambda } \norm{\partial_\nu  \zeta}_{L^2((-T,T )\times \Gamma_0)}^2
		\end{equation*}
		or, in a more concise form, for all $\lambda \geq \lambda_*$,
		\begin{equation}
			\label{AlmostOver}
			\norm{\zeta}_{H^1((-T/8, T/8)\times \omega)}
			\leq 
			\frac{C}{\lambda^{\gamma}} \mathscr{D} + C  e^{c_6 \lambda/2 } \norm{\partial_\nu  \zeta}_{L^2((-T,T )\times \Gamma_0)}.
		\end{equation}
		
		Finally, if we define the ratio ``data over measurement''
		$$
			~\rho=\dfrac{\mathscr{D}}{\norm{\partial_\nu  \zeta}_{L^2((-T,T )\times \Gamma_0)}} ~
		$$
		and the critical value
		\begin{equation}
			\label{Opt-Lambda-0}
			~\lambda_0 = \dfrac{1}{c_6} \log\left(2+ \rho \right),~
		\end{equation}
		taking $\lambda = \lambda_0$ if $\lambda_0 \geq \lambda_*$ we have
		$$
		\norm{\zeta}_{H^1((-T/8, T/8)\times \omega)}
			\leq C \mathscr{D} \left(\frac{1}{[\log(2+\rho)]^{\gamma}}+\frac{(2+\rho)^{1/2}}{\rho}\right).
		$$
		We can drop the second term of the right hand side since the first term dominates as $\rho\to\infty$ ($\rho$ is bounded from below by the continuity of the operator $z\mapsto \partial_\nu z$ from $H^2((-T,T)\times \Omega)$ to $L^2((-T,T) \times \partial \Omega)$).
		Otherwise, if $\lambda_0 <\lambda_*$, we take $\lambda = \lambda_*$ : In this case, $\rho \leq \exp( c_6 \lambda_* ) = C$, i.e. $\mathscr{D} \leq C \norm{\partial_\nu \zeta}_{L^2((-T,T) \times \Gamma_0)} $, so that \eqref{AlmostOver} with $\lambda = \lambda_*$ yields
		$$
			\norm{\zeta}_{H^1((-T/8, T/8)\times \omega)}
			\leq C\norm{\partial_\nu  \zeta}_{L^2((-T,T )\times \Gamma_0)} \leq C \frac{\mathscr{D}}{\rho}.
		$$
		This concludes the proof of \eqref{Est-on-Ocal-zeta} since $-\gamma<-1/(1+\alpha)$.	
\end{proof}

\begin{remark}\label{Rem-Phung09}
	When $f$ vanishes everywhere in $(0,T) \times \Omega$, no cut-off function $\chi_R$ is needed and one obtains the following quantification of unique continuation result due to \cite[Theorem F]{Phung09} (see also \cite{Robbiano95} for $\alpha = 1$): For all $T>0$ large enough, for all $\zeta \in H^2((-T,T) \times \Omega)$ solution of the wave equation \eqref{Eq-zeta-Thm} with $f = 0$, 
	$$
		\norm{\zeta }_{H^1((-T/8,T/8)\times \Omega)} 
		\leq
		C \norm{\zeta }_{H^2((-T,T)\times \Omega)} \left[\log\left( 2 + \frac{\norm{\zeta }_{H^2((-T,T)\times \Omega)}}{\norm{\partial_\nu \zeta}_{L^2((-T,T)\times \Gamma)}} \right)\right]^{-\frac{1}{1+\alpha}}, 
	$$
	or, equivalently, 
	\begin{equation*}
		\norm{\zeta }_{H^2((-T,T)\times \Omega)} \leq C \exp\left(C \Lambda^{1+\alpha}  \right) \norm{\partial_\nu \zeta}_{L^2((-T,T)\times \Gamma)},
		\hbox{ where }
		\Lambda = \frac{\norm{\zeta }_{H^2((-T,T)\times \Omega)} }{\norm{\zeta }_{H^1((-T/8,T/8)\times \Omega)} }.
	\end{equation*}
	Since $\zeta$ in that case is a solution of the wave equation with no source term, this last formulation can be written in terms of the initial data $(\zeta(0), \partial_t \zeta(0) ) = (\zeta^0, \zeta^1) \in H^2\cap H^1_0(\Omega) \times H^1_0(\Omega)$:
	\begin{equation*}
		\norm{ (\zeta^0, \zeta^1)}_{H^2\cap H^1_0(\Omega) \times H^1_0(\Omega)} \leq C \exp(C  \Lambda_0^{1+\alpha} ) \norm{\partial_\nu \zeta}_{L^2((-T,T)\times \Gamma)},
		 \\
		 \hbox{ where }  \Lambda_0 = \frac{ \norm{ (\zeta_0, \zeta_1)}_{H^2 \cap H^1_0 \times H^1_0} }{ \norm{ (\zeta_0, \zeta_1)}_{ H^1_0 \times L^2} }.
	\end{equation*}
\end{remark}
\subsection{Uniform stability in the semi-discrete case}
The goal of this section is to derive the semi-discrete counterpart of Theorem \ref{Thm-Est-By-FBI}. Similarly as in the continuous case, that will be the main ingredient for the proof of Theorem \ref{TCWElog}.
\\
As specified in the introduction, we limit ourselves to the case $\Omega = (0,1)^2$. We may thus assume that $\Gamma_0$ is a subset of one edge. Due to the invariance by rotation, with no loss of generality, we may further assume that this edge is $\{1\} \times (0,1)$.  
\\
We claim the following result: 
\begin{theorem}\label{Thm-Est-By-FBI-Discrete}
	Let $\Omega = (0,1)^2$ and $\Gamma_0$ be a non-empty open subset of the edge $\{1\} \times (0,1)$. Let $\omega$ be a connected open subset of $\Omega$ with Lipschitz boundary and assume that $\partial \omega \cap \partial \Omega$ is an open neighborhood of $\Gamma_0$. Also set $\omega_h = \omega \cap \Omega_h$.
	Let $m>0$ and $q_h \in L_h^\infty(\Omega_h)$ satisfying $\norm{q_h}_{L^\infty_h(\Omega_h)}\leq m$.
	Let $\mathscr{D} >0$ and $R_0 >0$, and assume that $\zeta_h$ is a solution of the wave equation 
		\begin{equation}
			\label{Eq-zeta-Thm-Discrete}
			\left\{
				\begin{array}{ll}
					\partial_{tt} \zeta_h - \Delta \zeta_h + q_h \zeta_h = f_h,  &\qquad \hbox{ in }  (-T,T) \times \Omega_h, 
					\\
					 \zeta_h = 0  &\qquad  \hbox{ on }  (-T,T) \times \partial \Omega_h,
				\end{array}
			\right.
		\end{equation}
	for some $f_h \in L^{1}(-T,T; L^2_h(\Omega_h))$ satisfying
	$f_h = 0 \text{ in }	(-T,T) \times \{x_h \in \Omega_h, \, d(x_h, \omega) < R_0\},~$	
	and satisfies $\zeta_h \in H_h^2((-T,T) \times \Omega_h)$ with
	$$
		\norm{\zeta_h}_{H^2_h((-T,T) \times \Omega_h)} \leq \mathscr{D}.
	$$
	for some $R_0>0$ and $\mathscr{D}$ independent of $h>0$.
	
	Let $\alpha >0$. There exist $T_0>0$ and $h_0>0$ such that for any $T\geq T_0$, there exists a constant $C$ independent of $h$ such that for all $h \in (0,h_0)$, 
	\begin{equation}
		\label{Est-on-Ocal-zeta-h}
		\norm{\zeta_h }_{H_h^1((-T/8,T/8)\times \omega_h)} 
		\leq
		C \mathscr{D}\left[\log\left( 2 + \frac{\mathscr{D}}{\norm{\partial_{h,2}^- \zeta_h}_{L^2((-T,T);L^2_h (\Gamma_{0,h}))}} \right)\right]^{-\frac{1}{1+\alpha}} \hspace{-3ex}+ C \mathscr{D} h^{1/(1+\alpha)}.
	\end{equation}
\end{theorem}
Before proving Theorem \ref{Thm-Est-By-FBI-Discrete}, let us point out that it differs from Theorem \ref{Thm-Est-By-FBI} by the last term $h^{1/(1+\alpha)} \mathscr{D}$ in \eqref{Est-on-Ocal-zeta-h}. Nonetheless, this term vanishes in the limit $h \to 0$ and thus estimate \eqref{Est-on-Ocal-zeta} can be recovered from \eqref{Est-on-Ocal-zeta-h} when $h \to 0$.
But in particular, estimate \eqref{Est-on-Ocal-zeta-h} does not state a uniqueness result anymore, but rather an ``almost-uniqueness'' result: if $\partial_{h,2}^- \zeta_h$ vanishes on $(-T,T) \times \Gamma_{0,h}$ for some $\zeta_h$ satisfying the assumptions of Theorem \ref{Thm-Est-By-FBI-Discrete}, we only have that the norm of $\zeta_h$ in $H_h^1((-T/8,T/8)\times \omega_h)$ is smaller than $C h^{1/(1+\alpha)} \mathscr{D}$. Due to the definition of $\mathscr{D}$, this corresponds to the case where 
$$
	\norm{\zeta_h }_{H_h^1((-T/8,T/8)\times \omega_h)} \leq C h^{1/(1+\alpha)}\norm{\zeta_h}_{H^2_h((-T,T) \times \Omega_h)}, 
$$ 
i.e. functions that are localized outside $(-T/8,T/8)\times \omega_h$. This is completely consistent with the presence of spurious high-frequency modes that are localized, see \cite{Tref,Zua05Survey,ErvZuaCime}. We refer for instance to a counterexample due to O. Kavian: if $w_h$ denotes the discrete function given by $w_{i,j} = (-1)^i$ when $i = j$ and vanishing for $i \neq j$, the function $\zeta_h(t,x_h) = \exp(2{\bf i} t/h) w_h(x_h)$ is a solution of \eqref{Eq-zeta-Thm-Discrete} with $q_h = 0$ and $f_h = 0$ whose discrete normal derivative on $\{1\}\times (1/4,3/4)$ vanishes identically. 
\begin{proof}[Proof of Theorem \ref{Thm-Est-By-FBI-Discrete}] It follows the same steps as the one of Theorem \ref{Thm-Est-By-FBI}. More precisely, Steps~1, 2 and 4 involving the FBI transform in time are left unchanged, but Steps 3 and 5 need to be modified.
	Indeed, Step 3 in the proof of Theorem \ref{Thm-Est-By-FBI} is based on the Carleman estimate in Lemma \ref{Lemma-Carleman} and we should thus use a semi-discrete counterpart. Namely, we use the discrete Carleman inequality proved in \cite[Theorem 1.4]{BoyerHubertLeRousseau2} that we rewrite below within our setting and using our notations. 

	Before stating this result, let us make precise how we choose the weight function. In particular, let us emphasize that the weight function in \cite{BoyerHubertLeRousseau2} is assumed to be $C^p([-3, 3]\times \overline{\Omega})$ for $p$ large enough, and this cannot be true with the construction we did for the proof of Theorem~\ref{Thm-Est-By-FBI}, since $\Omega = (0,1)^2$ contains corners.. We thus build the weight function $\psi_{0,\mathrm r}$ as follows (here the subscript `$\mathrm r$' stands for `regularized'): first we conceive an open subset $\omega_{\mathrm r}$ such that $\omega_{\mathrm r} \subset \{x \in \Omega,\,  d(x, \omega) <R_0/2 \}$, $\omega \subset \omega_{\mathrm r}$, and $\partial \omega_{\mathrm r} \setminus \Gamma_+$ is smooth (see Fig.~\ref{fig:weight2}). 	
\begin{figure}[h]
\begin{center}
\scalebox{0.6}{\input{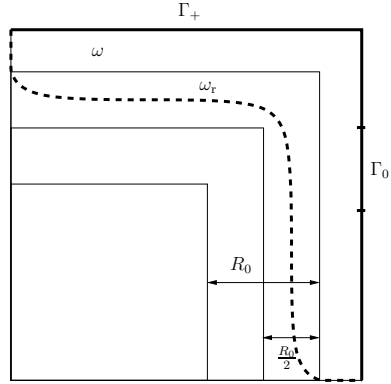}}
\end{center}
\caption{Construction of the weight function $\psi_{0,\mathrm r}(x)$ 
when $\omega$ is a neighborhood of two consecutive edges.
}
\label{fig:weight2}
\end{figure}
	
	\noindent We can then design a smooth weight function $\psi_{0,\mathrm r}$ such that
	\begin{equation}
			\label{Psi-0-discrete}
			\left\{
				\begin{array}{l}
					\forall x \in \overline{ \omega_{\mathrm r} },\, \psi_{0,\mathrm r}(x) \geq 0,
					\\
					\inf_{\overline{\omega_{\mathrm r}}} \{|\nabla \psi_{0,\mathrm r}(x)| \} >0,
					\\
					\forall x \in \partial  \omega_{\mathrm r} \setminus \Gamma_0, \, \partial_\nu \psi_{0,\mathrm r}(x) < 0,
					\\
					\forall x \in \partial  \omega_{\mathrm r} \setminus \Gamma_+, \,  \psi_{0,\mathrm r}(x) = 0,
					\\
					\norm{\psi_{0,\mathrm{r}}}_{L^\infty(\omega_{\mathrm r})} \leq 1/2. 
				\end{array}
			\right.
	\end{equation}

	Again, such a function $\psi_{0,\mathrm r}$ exists according to the construction in \cite{FursikovImanuvilov,TWbook} and it can be extended as a smooth function $\psi_{\mathrm r}$ on $\overline{\Omega}$ satisfying 
		$\norm{\psi_{\mathrm r}}_{L^\infty(\Omega)} \leq 1$. By continuity, there exists $R \in (0, R_0/2)$ such that for the sets 
		$$
			\omega_{\mathrm r,R} = \{ x \in \Omega,\, d(x, \omega_{\mathrm r}) < R\} \quad \hbox{ and } \quad \mathscr{C}_{\mathrm r}  = \{ x \in \Omega, \, R/2 \leq d(x, \omega_{ \mathrm r}) \leq R\}, 
		$$
		 we have 
		\begin{equation}
			\label{ConditionPsi-OcalR-Discrete}
			\inf_{\overline{\omega_{\mathrm r,R}}} \{|\nabla \psi_{ \mathrm r}(x)| \} >0, \quad \hbox{ and } \quad \inf_{\overline{\omega_{\mathrm r}}} \psi_{ \mathrm r} > \sup_{\overline{\mathscr{C}_{\mathrm r}}} \psi_{\mathrm r}.
		\end{equation}	
		We then define $\varphi_{ \mathrm r}$ as in \eqref{DefPhi-Bel} but with this function $\psi_{\mathrm r}$: for $\mu \geq 1$, 
		$$
			\varphi_{\mathrm r} := \varphi_{\mathrm r}(s,x) = \exp( \mu (\psi_{ \mathrm r}(x) - s^2))
			\quad (s,x) \in [-3, 3] \times \overline{\Omega}.
		$$
	\begin{theorem}[\cite{BoyerHubertLeRousseau2}]\label{Thm-Carl-BHLR}
		Let $\varphi_{\mathrm r}$ be as above and its restriction on the mesh $\varphi_{\mathrm r,h} = \rh \varphi_{\mathrm r}$.\\
		There exist $\mu \geq 1$, $C>0$, $h_0>0$ and $\varepsilon_0 >0$ such that for all $h \in (0,h_0)$, $\tau \geq 1$ with $\tau h \leq \varepsilon_0$, for all $g_h\in L^2((-3, 3); L^2_h( \Omega_h))$ and $w_h$ solution of
		$$
				\left\{
					\begin{array}{ll}
						(- \partial_{ss} - \Delta_h + q_h ) w_h = g_h  &\qquad\hbox{ in } (-3,3) \times \Omega_h, 
						\\
						w_h = 0  &\qquad \hbox{ on } ( (-3, 3) \times \partial \Omega_h) \cup (\{-3,3\}\times \Omega_h),
					\end{array}
				\right.
		$$
		 supported in $(-3, 3) \times \omega_{\mathrm{r},R}$, 
			\begin{multline}
				\label{Carleman-Ass-Bel-Disc}
					\tau^3 \norm{e^{\tau \varphi_{\mathrm{r},h}} w_h}_{L^2(-3, 3;L^2_h(\Omega_h))}^2
					+ \tau \norm{e^{\tau \varphi_{\mathrm{r},h}} \nabla_{s} w_h}_{L^2(-3, 3; L_h^2(\Omega_h))}^2 
					+ \tau \sum_{k=1,2} \norm{e^{\tau \varphi_{\mathrm{r},h}} \partial_{h,k}^+ w_h}_{L^2(-3, 3; L_h^2(\Omega_{k,h}^-))}^2 
					\\
					\leq C \norm{e^{ \tau \varphi_{\mathrm{r},h}} g_h}_{L^2(-3, 3; L^2_h(\Omega_h))}^2  + C \tau \norm{e^{\tau \varphi_{\mathrm{r},h}} \partial_{h,2}^- w_h}_{L^2(-3,3; L^2_h( \Gamma_{0,h}))}^2.
			\end{multline}
		Besides, the constant $C$ can be taken uniformly with respect to $q_h \in L^\infty_h (\Omega_h)$ with $\norm{q_h}_{L^\infty_h} \leq m$.
	\end{theorem}

	\begin{remark}
		Before going further, let us comment more precisely Theorem \ref{Thm-Carl-BHLR}, which cannot be found under that precise form in \cite{BoyerHubertLeRousseau2} and differs from \cite[Theorem 1.4]{BoyerHubertLeRousseau2} at three levels.
		
		The first issue is that Theorem 1.4 in \cite{BoyerHubertLeRousseau2} concerns the case of an observation on the boundary of the continuous variable, corresponding here to $s = \pm 3$. Therefore, Assumption 1.3 on the weight function in \cite{BoyerHubertLeRousseau2} is designed to yield observations on the boundary of the continuous variable, and in our case, they are replaced by the condition $\forall x \in \partial  \omega_{\mathrm r} \setminus \Gamma_0, \, \partial_\nu \psi_{0,\mathrm r}(x) < 0$ in \eqref{Psi-0-discrete}. 
		We claim that this condition is enough to guarantee a Carleman estimate with an observation on the boundary of the discrete variables. This can be proved following the lines of \cite{BoyerHubertLeRousseau2} in that case and looking at the boundary terms denoted $Y$ and estimated in \cite[Lemma 3.7]{BoyerHubertLeRousseau2}, which are strong enough to absorb the boundary terms in $J_{11}$ in \cite[Lemma 3.3]{BoyerHubertLeRousseau2} on $\partial \Omega \setminus \Gamma_0$.

		The second issue is that Assumption 1.3 in \cite{BoyerHubertLeRousseau2} requires some convexity condition in the neighborhood of the boundary. But, as mentioned in \cite[Remark 1.3]{BLR-ANIHP13}, this can be avoided by suitably modifying the proof of Lemma C.4 in \cite{BoyerHubertLeRousseau2}.

		The third and last issue is that our weight function may degenerate outside $(-3, 3) \times \omega_{\mathrm{r},R}$. But, as in the continuous case, this actually does not come into play as we apply Carleman estimate \eqref{Carleman-Ass-Bel-Disc} to discrete functions $w_h$ supported in $(-3, 3) \times \omega_{\mathrm{r},R}$.
	\end{remark}
	Note that the main difference in the discrete Carleman estimate of Theorem \ref{Thm-Carl-BHLR} with respect to the one in Lemma \ref{Lemma-Carleman} is the fact that the parameter $\tau$ is assumed to satisfy $\tau h \leq \varepsilon_0$. The proof of Theorem \ref{Thm-Est-By-FBI} shall then be modified to keep track on this restriction. 
	Thus, Step 3 can be done as in the proof of Theorem \ref{Thm-Est-By-FBI}, except that the construction of the cut-off function $\chi_R$ is now based on $\omega_{\mathrm r}$, and the existence of $\varepsilon_0>0$ such that 
	$$
		\inf_{|s| \leq \varepsilon_0,\, x \in \omega_{\mathrm r}} \psi_{\mathrm r}(s,x) > \sup_{|s| \leq 3, \, x \in \mathscr{C}_{\mathrm r}} \psi_{\mathrm r}(s,x)
	$$
	is granted by \eqref{ConditionPsi-OcalR-Discrete}. Then, all the constants $\mathscr{I}_\omega$, $\mathscr{S}$,  $\mathscr{S}_{(2,3)}$, $ \mathscr{S}_{\mathscr{C}}$ in \eqref{Notations-Sup-Inf-Bel}, now denoted $\mathscr{I}_{\omega_{\mathrm r}}$,  $\mathscr{S}$, $ \mathscr{S}_{(2,3)}$, $ \mathscr{S}_{\mathscr{C}_{\mathrm r}}$, are defined by replacing $\omega$ by $\omega_{\mathrm r}$, $\varphi$ by $\varphi_{\mathrm r}$ and $\mathscr{C}$ by $\mathscr{C}_{\mathrm r}$. 
	Hence, instead of \eqref{Est-v-a-lambda-obs-f}, we obtain the following: for all $h \in (0,h_0)$, $ \tau \geq 1$ with $\tau h \leq \varepsilon_0$, for all $\lambda \geq 1$,
	\begin{multline*}
			 \norm{\zeta_h}_{H^1((-T/8, T/8)\times \omega_{{\mathrm r},h})}^2 
			 \leq 
			\frac{C}{\lambda^{2\gamma}} \mathscr{D}^2
			+
			C  e^{2 \tau (\mathscr{S}- \mathscr{I}_{\omega_{\mathrm r}} )}  \lambda^{2\gamma} e^{ -2 c_1 \lambda  (T /2)^{1/\gamma}} \mathscr{D}^2
			\\
			+
			C e^{ -2 \tau \mathscr{I}_{\omega_{\mathrm r}}} \left(e^{2 \tau \mathscr{S}_{(2,3)} }  + e^{2 \tau \mathscr{S}_{\mathscr{C}_{\mathrm r}}} \right)	\lambda^{4\gamma} e^{c_3 \lambda} \mathscr{D}^2
			+ ~
			C \tau e^{ 2 \tau (\mathscr{S}- \mathscr{I}_{\omega_{\mathrm r}}) } \lambda^{4\gamma} e^{c_3 \lambda } \norm{\partial_{h,2}^-  \zeta_h}_{L^2(-T,T ; L^2_h(\Gamma_{0,h}))}^2.
	\end{multline*}
	The discussion then follows the same path as in the Step 5 of the proof of Theorem \ref{Thm-Est-By-FBI}: the natural choice is to take $\tau$ as a linear function of $\lambda$ as in \eqref{Tau(Lambda)}. Thereby, we get the following discrete counterpart of \eqref{AlmostOver}: there are constants $C>0$ and $\varepsilon_*>0$ independent of $h>0$ such that for all $h \in (0,h_0)$ and for all $\lambda \in (\lambda_*, \varepsilon_*/h)$, 
	\begin{equation}
		\label{Step-5-Modified}
			\norm{\zeta_h}_{H^1_h((-T/8, T/8)\times \omega_h)} 
			\leq 
			\frac{C}{\lambda^{\gamma}} \mathscr{D} + C  e^{c_6 \lambda/2 } \norm{\partial_{h,2}^-  \zeta_h}_{L^2(-T,T ; L^2_h(\Gamma_{0,h}))}.
	\end{equation}
	Introducing the ratio
	$$
			\rho_h=\frac{\mathscr{D}}{\norm{\partial_{h,2}^-  \zeta_h}_{L^2(-T,T ; L^2_h(\Gamma_{0,h}))}}, 
	$$
	the optimal value of the parameter $\lambda$ is 
	$$
		\lambda_{0,h} =  \frac{1}{c_6}\log\left(2+ \rho_h \right),
	$$ corresponding to the choice \eqref{Opt-Lambda-0} in the proof of Theorem \ref{Thm-Est-By-FBI}. We then have to discuss the cases $\lambda_{0,h} \leq \lambda_*$, $\lambda_{0,h} \in (\lambda_*, \varepsilon_*/h)$ and $\lambda_{0,h} \geq \varepsilon_*/h$. Of course, the first two cases can be handled as in the continuous setting. There only remains the last case $\lambda_{0,h} \geq \varepsilon_*/h$. But this corresponds to
	$
		\rho_h \geq \exp( c_6 \varepsilon_*/h) - 2 \geq \exp( c_6 \varepsilon_*/h)/2,
	$
	for $h$ small enough, which in particular implies
	$$
		2 \norm{\partial_{h,2}^-  \zeta_h}_{L^2(-T,T ; L^2_h(\Gamma_{0,h}))} \leq  \mathscr{D} \exp( - c_6 \varepsilon_*/h).
	$$
	Thus, taking $\lambda = \varepsilon_*/h$ in \eqref{Step-5-Modified}, we obtain 
	$$
		\norm{\zeta_h}_{H^1_h((-T/8, T/8)\times \omega_h)}\leq C h^\gamma \mathscr{D}.
	$$
	This explains the presence of the last term in \eqref{Est-on-Ocal-zeta-h}.
\end{proof}
We finally conclude this section with the proof of Theorem \ref{TCWElog}.
\begin{proof}[Proof of Theorem \ref{TCWElog}]
	As for the proof of Theorem \ref{ThmBellassoued} from \eqref{Thm-Est-By-FBI}, it follows immediately by applying Theorem \ref{Thm-Est-By-FBI-Discrete} to $\zeta_h = \partial_t y_h[q_h^a] - \partial_t y_h[q_h^b]$. The use of estimate \eqref{UniformStability-3Loc} of Theorem~\ref{TCWE1} then completes the proof. Details are left to the reader.
\end{proof}
\begin{remark}
	\label{Rem-Phung09-discrete}
	Following Remark \ref{Rem-Phung09}, we can derive a quantification of a kind of unique continuation result for solutions $\zeta_h$ of discrete wave equations \eqref{Eq-zeta-Thm-Discrete} with no source term: For all $\alpha >0$ and $T>0$ large enough, there exists a constant $C$ independent of $h>0$ such that for all $\zeta_h $ solution of the wave equation \eqref{Eq-zeta-Thm-Discrete} with $f_h = 0$ and initial data $(\zeta^0_h, \zeta^1_h) \in H^2_h\cap H^1_{0,h}(\Omega_h)\times H^1_{0,h}(\Omega_h)$,
	\begin{multline}
		\label{Discrete-Weak-Obs}
		\norm{ (\zeta^0_h, \zeta^1_h)}_{H^1_{0,h}(\Omega_h)\times L^2_h(\Omega_h)}\leq C e^{C \Lambda_h^{1+\alpha}} \norm{\partial_{h,2}^- \zeta}_{L^2(-T,T; L^2_h (\Gamma_{0,h}))} 
		\\
		+ C h^{1/(1+\alpha)} \norm{ (\zeta^0_h, \zeta^1_h)}_{H^2_h\cap H^1_{0,h}(\Omega_h)\times H^1_{0,h}(\Omega_h)},
	\end{multline}
	where
	$
		\Lambda_h = \dfrac{ \norm{ (\zeta^0_h, \zeta^1_h)}_{H^2_h\cap H^1_{0,h}(\Omega_h)\times H^1_{0,h}(\Omega_h)} }{ \norm{ (\zeta^0_h, \zeta^1_h)}_{ H^1_{0,h}(\Omega_h) \times L^2_h(\Omega_h)} }
	$
	or, equivalently,
	$$
		(1 - C h^{1/(1+\alpha)} \Lambda_h)\norm{ (\zeta^0_h, \zeta^1_h)}_{H^1_{0,h}(\Omega_h)\times L^2_h(\Omega_h)}\leq C e^{C \Lambda_h^{1+\alpha}} \norm{\partial_{h,2}^- \zeta}_{L^2((-T,T); L^2_h (\Gamma_{0,h}))}.
	$$
	Note that \eqref{Discrete-Weak-Obs} only yields an ``almost uniqueness'' result in the sense that it does not imply $\zeta_h \equiv 0$ when the discrete normal derivative $\partial_{h,2}^- \zeta_h$ vanishes on $(-T,T) \times \Gamma_{0,h}$. Recall here that this term is needed as unique continuation for the discrete wave equations does not hold as shown by the counterexample of O. Kavian of an eigenfunction of the discrete Laplace operator which is localized on the diagonal of the square.
\end{remark}
%
\section{Convergence and consistency issues}\label{SecCV}
This last section is devoted to the proof of the convergence results stated in Theorem~\ref{Thm-Cvg}.  
\subsection{Convergence results for the inverse problem}
We will first state and prove two theorems of convergence under more detailed consistency assumptions.
The feasibility of these assumptions will be studied next.
Under the Gamma-conditions, and more specifically in the geometric setting \eqref{Gamma-+-T-sqrt2}, we obtain:
\begin{theorem}[Convergence under Gamma-conditions]
	\label{Thm-Cvg-Gamma}
	Assume that $(\Omega,  \Gamma_0,T)$ satisfies the configuration \eqref{Gamma-+-T-sqrt2} and that $(y^0,y^1, f, f_\partial)$ follows the conditions \eqref{SmallerPossibleClass}.
	Let $q \in L^\infty(\Omega)$ and assume that there exist sequences $q_h^a \in L^\infty_h(\Omega_h)$, and $(y^0_h,y^1_h, f_h, f_{\partial,h})$ of discrete functions in $L^2_h(\Omega_h)^2 \times L^1(0,T; L^2_h(\Omega_h)) \times L^2(0,T; L^2_h( \partial \Omega_h))$ such that
	\begin{eqnarray}
		&& \lim_{h \to 0} \norm{\eh^0 (q_h^a)-q }_{L^2(\Omega)} = 0, \label{Cond1-ConsGamma}
		\\
		&& \lim_{h \to 0} \norm{\widetilde{\mathscr{M}_h}[q_h^a] - \widetilde{\mathscr{M}_0}[q]}_{ H^1(0,T;L^2(\Gamma_0))\times L^2((0,T)\times\Omega)} = 0, \label{Cond3-ConsGamma}
		\\
		&& \limsup_{h \to 0} \norm{q_h^a}_{L^\infty_h(\Omega_h)} < \infty, \label{Cond2-ConsGamma}
		\\
		&& \limsup_{h \to 0} \| y_h[q_h^a] \|_{H^1(0,T;L^\infty_h(\Omega_h))} < \infty, \label{Cond4-ConsGamma}
		\\
		&& \exists \alpha_0 >0, \, \forall h >0, \quad  \inf_{\Omega_h} |y^{0}_h|\geq \alpha_0. \label{Cond5-ConsGamma}
	\end{eqnarray}
	Then for all sequence $(q_h^b)_{h>0}$ of potentials satisfying 
	$$
		\limsup_{ h \to 0}\norm{q_h^b}_{L^\infty_h(\Omega_h)} <\infty, 
		\quad\hbox{ and }\quad \lim_{h \to 0} \norm{\widetilde{\mathscr{M}_h}
		[q_h^b] - \widetilde{\mathscr{M}_0}
		[q]}_{ H^1(0,T;L^2(\Gamma_0))\times L^2((0,T)\times\Omega)} = 0,
	$$
	we have
	$$
		 \lim_{h \to 0} \norm{\eh^0 (q_h^b)-q }_{L^2(\Omega)} = 0.
	$$	
\end{theorem}
When no geometric condition on the observation domain is satisfied, we get:
\begin{theorem}[Convergence under weak geometric conditions]
	\label{Thm-Cvg-log}
	Assume the geometric configuration \eqref{Config-Bellassoued} for $(\Omega,  \Gamma_0,  \Gamma_+)$, the conditions \eqref{SmallerPossibleClass} for $(y^0,y^1, f, f_\partial)$, and let $\mathcal{O}$ be a neighborhood of $\Gamma_+$. 
	Let $q \in L^\infty(\Omega)$ and assume that there exist sequences $q_h^a \in L^\infty_h(\Omega_h)$, and $(y^0_h,y^1_h, f_h, f_{\partial,h})$ of discrete functions in $L^2_h(\Omega_h)^2 \times L^1(0,T; L^2_h(\Omega_h)) \times L^2(0,T; L^2_h(\partial \Omega_h))$ such that \eqref{Cond1-ConsGamma}, \eqref{Cond3-ConsGamma} and \eqref{Cond2-ConsGamma} are fulfilled, along with 
	\begin{eqnarray}
		&& \limsup_{h \to 0} \| y_h[q_h^a] \|_{H^1(0,T;L^\infty_h(\Omega_h))\cap W^{2,1}(0,T; L^2_h(\Omega_h))} < \infty, \label{Cond4-Conslog}
		\\
		&& \exists \alpha_0 >0, \, \forall h >0, \quad  \inf_{\Omega_h} |y^{0}_h|\geq \alpha_0 \quad \hbox{ and } \quad \limsup_{h >0} \norm{y^0_h}_{H^1_{h}(\Omega_h)} < \infty. \label{Cond5-Conslog}
	\end{eqnarray}
	Then for $T>0$ large enough, for all sequence $(q_h^b)_{h>0}$ of potentials satisfying 
	\begin{eqnarray*}
		&& q_h^b = q_h^a \hbox{ in } \mathcal{O}_h \quad  \hbox{ and } \quad
		q_h^a - q_h^b \in H^1_{0,h}(\Omega_h) \quad \hbox{ with } \quad \limsup_{h\to 0} \norm{q_h^b - q_h^a}_{H^1_{0,h}(\Omega_h)} < \infty,
		\\
		&&\limsup_{h \to \infty} \norm{q_h^b}_{L^\infty_h(\Omega_h)} <\infty, 
		\, \hbox{ and }\, 
		 \lim_{h \to 0} \norm{\widetilde{\mathscr{M}_h}[q_h^b] - \widetilde{\mathscr{M}_0}[q]}_{ H^1(0,T;L^2(\Gamma_0))\times L^2((0,T)\times\Omega)} = 0,
	\end{eqnarray*}
	we have
	$$
		 \lim_{h \to 0} \norm{\eh^0 (q_h^b)-q }_{L^2(\Omega)} = 0.
	$$	
\end{theorem}

Theorems \ref{Thm-Cvg-Gamma} and \ref{Thm-Cvg-log} follow from the same arguments and can be proved simultaneously.
\begin{proof}[Proof of Theorems \ref{Thm-Cvg-Gamma} and \ref{Thm-Cvg-log}]
	Let $q_h^a$ and $q_h^b$ be as assumed in Theorem \ref{Thm-Cvg-Gamma} (resp. Theorem~\ref{Thm-Cvg-log}). One easily gets
	$$
		\lim_{h \to 0} \norm{\widetilde{\mathscr{M}_h}[q_h^a] -\widetilde{\mathscr{M}_h}[q_h^b]}_{ H^1(0,T;L^2(\Gamma_0))\times L^2((0,T)\times\Omega)} = 0.
	$$	
	Since one can find $m >0$ larger than $\norm{q}_{L^\infty(\Omega)}$ and  $\limsup_{h \to 0} (\norm{q_h^a}_{L^\infty_h(\Omega_h)}+\norm{q_h^b}_{L^\infty_h(\Omega_h)})$, according to Theorem \ref{TCWE1}, (resp. Theorem \ref{TCWElog}), we get
	$$
		\lim_{h \to 0} \norm{q_h^a - q^b_h}_{L^2_h(\Omega_h)} = 0,
	\quad \hbox{ or equivalently, } \quad 
		\lim_{h\to 0}\norm{\eh^0(q_h^b) - \eh^0(q_h^a)}_{L^2(\Omega)}= 0.
	$$
	 We then conclude by the triangular inequality
	$$
		\norm{\eh^0(q_h^b) - q}_{L^2(\Omega)}\leq \norm{\eh^0(q_h^b) - \eh^0(q_h^a)}_{L^2(\Omega)} + \norm{\eh^0(q_h^a) - q}_{L^2(\Omega)},
	$$
	since each term in the right hand-side converges to zero as $h \to 0$.
\end{proof}
Of course, Theorems \ref{Thm-Cvg-Gamma} and \ref{Thm-Cvg-log} are based on the strong assumption that there exists a sequence of potentials $q_h^a$ satisfying suitable convergence assumptions for some $(y^0_h,y^1_h, f_h, f_{\partial,h})$ that are not even supposed to be convergent to their continuous counterpart. This rises the natural question: given $(y^0, y^1, f, f_\partial)$ satisfying \eqref{SmallerPossibleClass}, can we guarantee that the natural approximations $(y^0_h,y^1_h, f_h, f_{\partial,h})$ of $(y^0, y^1, f, f_\partial)$ yields the existence of a sequence of potentials $q_h^a$ satisfying the convergence conditions of Theorem \ref{Thm-Cvg-Gamma} or Theorem \ref{Thm-Cvg-log} ?

This is the consistency of the inverse problem, and the cornerstone of the proof of Theorem~\ref{Thm-Cvg} once stability results are proved. These consistency issues are  discussed in the following subsection. 
\subsection{Consistency issues}

The difficulty to derive the consistency of the inverse problem is the condition \eqref{Cond4-ConsGamma} (or \eqref{Cond4-Conslog} in the case of Theorem \ref{Thm-Cvg-log}). Indeed, passing to the limit, it indicates that $y[q]$ should belong to $H^1((0,T); L^\infty(\Omega))$. But there is no simple way to guarantee this condition, since the ``natural'' spaces for the wave equation are the $H^s(\Omega)$-spaces.

Let us remind the reader that we consider $\Omega = (0,1)^2 \subset \mathbb R^2$. We recall this setting here  because of its influence on the Sobolev's embeddings we will repeatedly use in this last section. 

Besides that, as our theorems of stability are given with conditions on $y[q]$ instead of conditions on the coefficients $(y^0, y^1, f, f_\partial)$, we will stick to that approach. We claim the following result:
\begin{lemma}\label{Lem-Consistency-Main}
	Assume $q \in H^1\cap L^\infty(\Omega) $ and that we know $q_\partial = q|_{\partial \Omega}$. 
	Furthermore, assume that the trajectory $y[q]$ solution of \eqref{WaveEq-Q} satisfies the regularity given in  \eqref{ConditionOnYq-1}.
	Finally, assume there exists $\alpha_0>0$ such that $\inf_{\overline{\Omega} } |y^0| \geq \alpha_0$.
	\\
	Then we can construct discrete sequences $(y^0_h, y^1_h, f_h, f_{\partial,h})$ depending only on $(y^0, y^1, f, f_\partial, q_\partial)$ such that the corresponding sequence $y_h[q_h]$ solution of \eqref{DisWaveEq} for  $q_h = \trh(q)$ satisfies conditions \eqref{Cond1-ConsGamma}--\eqref{Cond5-Conslog}. In particular, if $q$ is known on some open set $\Ocal$ and takes value $q|_{\Ocal} = Q$, we can further impose $q_h = \trh(Q)$ in $\Ocal_h$.
\end{lemma}

\begin{proof}[Proof of Theorem \ref{Thm-Cvg}]
Taking the discrete sequence $(y^0_h, y^1_h, f_h, f_{\partial,h})$ given by Lemma \ref{Lem-Consistency-Main}, the sequence $q_h^a = \trh(q)$ satisfies the assumption of Theorem \ref{Thm-Cvg-Gamma}, or Theorem \ref{Thm-Cvg-log} if $q$ is known in some open set $\Ocal$, which corresponds to the first item of Theorem~\ref{Thm-Cvg}.  The second item of Theorem~\ref{Thm-Cvg} thus follows immediately from Theorems~\ref{Thm-Cvg-Gamma} and \ref{Thm-Cvg-log}.
\end{proof}

\begin{proof}[Proof of Lemma \ref{Lem-Consistency-Main}]
	We split it in two steps. First, we will construct $(y^0_h, y^1_h, f_h, f_{\partial,h})$ and $q_h$; Second, we will explain why our construction is suitable for conditions \eqref{Cond1-ConsGamma}--\eqref{Cond5-Conslog}.\\

	Let us choose $\tilde q \in H^1\cap L^\infty(\Omega)$ with $\tilde q|_{\partial \Omega} = q_\partial$ (note that such $\tilde q$ exists since $q_\partial$ is the trace of $q \in H^1\cap L^\infty(\Omega)$ by assumption).
	We define $\tilde y = y[\tilde q]$ the solution of \eqref{WaveEq-Q} with potential $\tilde q$. Then, setting $z = y[q] - \tilde y$, it satisfies
	\begin{equation}
	\label{WaveEq-z}
		\left\{\begin{array}{ll}
			\partial_{tt} z - \Delta z+ \tilde q z=  ( \tilde q-q ) y[q],  &\qquad \hbox{ in }  (0,T)\times \Omega,
				\\
			z = 0, &\qquad\hbox{ on } (0,T) \times \partial \Omega,
				\\
			z(0,\cdot)= 0, \quad \partial_t z(0,\cdot)= 0. &\qquad \hbox{ in } \Omega,
		\end{array}
		\right.
	\end{equation}
	Hence $z_2 = \partial_{tt} z$ solves
	\begin{equation}
	\label{WaveEq-dtt-z}
		\left\{\begin{array}{ll}
			\partial_{tt} z_2 - \Delta z_2+ \tilde q z_2=  (  \tilde q- q) \partial_{tt} y[q],  &\qquad \hbox{ in }  (0,T)\times \Omega,
				\\
			z_2 = 0, & \qquad \hbox{ on }  (0,T)\times \partial\Omega,
				\\
			z_2(0,\cdot)= (\tilde q- q) y^0, \quad \partial_t z_2(0,\cdot)= ( \tilde q - q) y^1. & \qquad \hbox{ in } \Omega.
		\end{array}
		\right.
	\end{equation}
	Since \eqref{ConditionOnYq-1} implies $y^0 \in H^1 \cap L^\infty(\Omega)$, $y^1 \in L^2(\Omega)$ and $\partial_{tt}y[q] \in L^1(0,T; L^2(\Omega))$, and since $q - \tilde q \in H^1_0 \cap L^\infty( \Omega)$, we have that $z_2 = \partial_{tt} z$ belongs to $C([0,T]; H^1_0(\Omega)) \cap C^1([0,T]; L^2(\Omega))$.  In particular, since $z(0, \cdot ) = \partial_t z(0, \cdot) = 0$, we have $z \in H^2(0,T; H^1_0(\Omega))$. 
	
	Besides, by differentiating \eqref{WaveEq-z} once with respect to time, we get that $\partial_t z$ solves
	$$
		(- \Delta + \tilde q)\partial_t z = (\tilde q -  q) \partial_t y[q] - \partial_{ttt} z \in C([0,T]; L^2(\Omega)), \quad \hbox{ with } \partial_t z = 0 \hbox{ for } (t,x) \in (0,T) \times \partial \Omega.
	$$
	Therefore, by elliptic regularity estimates, see \cite[Theorem 3.2.1.2]{Grisvard}, $\partial_t z \in C([0,T]; H^2(\Omega))$, thus $z \in H^1(0,T; H^2(\Omega))$. 
	\\
	Recalling that $\tilde y = y[q] - z$ and $y[q]$ satisfies \eqref{ConditionOnYq-1}, $\tilde y$ belongs to $H^{2} (0,T; H^1(\Omega) )\cap H^1(0,T; H^2(\Omega))$.
	\\
	We then define $\tilde y_h = \trh(\tilde y)$ and, for $\tilde q_h = \trh(\tilde q)$, we set 
	\begin{eqnarray}
		\label{Init-Data}
		 y^0_h = \tilde y_h(0) =  \trh(y^0), & \quad  & y^1_h = \partial_t \tilde y_h(0) = \trh (y^1),
		 \\
		\label{SourceTerms}
		 f_h = \partial_{tt} \tilde y_h - \Delta \tilde y_h + \tilde q_h \tilde y_h,&\quad & f_{\partial, h} (t)= \tilde y_h(t)|_{\partial \Omega_h}.
	\end{eqnarray}
	Note that this choice immediately implies that conditions \eqref{Cond1-ConsGamma}, \eqref{Cond2-ConsGamma} and \eqref{Cond5-Conslog} (thus also \eqref{Cond5-ConsGamma}) are satisfied.\\
	
	We now prove that this construction yields condition \eqref{Cond4-Conslog}. This is based on the remark that by construction, for $q_h = \trh(q)$ we have $y_h[q_h] = \tilde y_h + z_h$, where $z_h$ solves
	 \begin{equation}
	\label{DisWaveEq-zh}
	\left\{
		\begin{array}{ll}
			\partial_{tt} z_h - \Delta_h z_h + q_h z_h= (\tilde q_h - q_h) \tilde y_h, &\qquad \hbox{ in } (0,T) \times \Omega_h, 
			\\
			z_h =0, &\qquad \hbox{ on } (0, T) \times \partial \Omega_h, 
			\\
			(z_h(0), \partial_t z_h(0) ) = (0, 0), &\qquad \hbox{ in } \Omega_h.
		\end{array}
	\right.
	\end{equation} 
	Then $z_{2,h} = \partial_{tt} z_h$ solves 
	 \begin{equation}
	\label{DisWaveEq-dtt-zh}
	\left\{
		\begin{array}{ll}
			\partial_{tt} z_{2,h} - \Delta_h z_{2,h} + q_h z_{2,h}= (\tilde q_h - q_h) \partial_{tt} \tilde y_h, & \hbox{ in } (0,T) \times \Omega_h, 
			\\
			z_{2,h} =0, & \hbox{ on } (0, T) \times \partial \Omega_h, 
			\\
			(z_{2,h}(0), \partial_t z_{2,h}(0) ) = ((\tilde q_h - q_h) y_h^0, (\tilde q_h- q_h)y_h^1), & \hbox{ in } \Omega_h.
		\end{array}
	\right.
	\end{equation} 
	One easily checks that with our construction 
	\begin{align*}
		& \tilde q_h - q_h \in H^1_{0,h}(\Omega_h) \cap L^\infty_h(\Omega_h), 
		\\
		& \tilde y_h \in H^{2}(0,T; H^1_h(\Omega_h)) \cap H^1(0,T; H^2_h(\Omega_h)), 
		\\
		& y^0_h \in H^1_h(\Omega_h) \cap L^\infty_h(\Omega_h), \quad y^1_h \in L^2_h(\Omega_h),
	\end{align*}
	where all these estimates stand with bounds uniform with respect to $h>0$. Hence $z_{2,h}$ is uniformly bounded in $ C([0,T]; H^{1}_{0,h} (\Omega_h)) \cap C^1([0,T]; L^2_h(\Omega_h))$ by energy estimates, so that $\partial_{ttt} z_h \in C([0,T]; L^2_h(\Omega_h))$ and thus $\partial_t z_h$ solves
	$$
		-\Delta_h \partial_t z_h + q_h \partial_t z_h =  (\tilde q_h - q_h) \partial_t \tilde y_h - \partial_{ttt} z_h \in C([0,T]; L^2_h(\Omega_h)) \hbox{ with } \partial_t z_h = 0 \hbox{ on } \partial \Omega_h.
	$$
	We use the following lemma, whose proof is postponed to Appendix \ref{Appendix-Elliptic}.
	\begin{lemma}
		\label{Lem-Elliptic-Reg}
		Let $w_h \in L^2_h(\Omega_h)$ be a solution of 
		\begin{equation}
			\label{Elliptic-Eq}
			-\Delta_h w_h +q_h w_h = g_h \hbox{ in } \Omega_h \quad \hbox{ and } \quad w_h = 0 \hbox{ on } \partial \Omega_h
		\end{equation}
		with $g_h \in L^2_h(\Omega_h)$ and $q_h \in L^\infty_h(\Omega_h)$. Let $m>0$ and assume $\norm{q_h}_{L^\infty_h(\Omega_h)}\leq m$. Then, $w_h \in H^2_h\cap H^1_{0,h}(\Omega_h)$ and there exists a constant $C = C(m)>0$ independent of $h>0$ such that
		\begin{equation}
			\label{Elliptic-Reg}
			\norm{w_h}_{H^2_h \cap H^1_{0,h} (\Omega_h)} \leq C \norm{g_h}_{L^2_h(\Omega_h)}.
		\end{equation}
	\end{lemma}
	Accordingly, $\partial_t z_h$ is uniformly bounded in $C([0,T]; H^2_h\cap H^1_{0,h}(\Omega_h))$. Thus, $y_h[q_h] = \tilde y_h + z_h$ is uniformly bounded in $H^2(0,T; H^1_h(\Omega_h)) \cap H^1(0,T; L^\infty_h(\Omega_h))$, yielding \eqref{Cond4-Conslog} (and \eqref{Cond4-ConsGamma}). \\
	
	We finally focus on the proof of the convergence condition \eqref{Cond3-ConsGamma}. As $\tilde y \in H^1(0,T; H^2(\Omega))$, $\tilde y_h$ is uniformly bounded in $ H^1(0,T; H^2_h(\Omega_h))$. In particular, for $k \in \{1, 2\}$, $\partial_{h,k}^\mp \tilde y_h$ is uniformly bounded in $H^1(0,T; H^1_h(\Omega_{h,k}^\pm))$, so $\eh(\partial_{h,k}^\mp \tilde y_h)$ is uniformly bounded in $H^1(0,T; H^1(\Omega))$. Besides, it is easy to check that, since $\tilde y \in H^1(0,T; H^2(\Omega))$,  $\eh(\partial_{h,k}^\mp \tilde y_h)$ strongly converges to $\partial_{x_k} \tilde y$ in $H^1(0,T;L^2(\Omega))$. Hence we get the strong convergence of $\eh(\partial_{h,k}^\mp \tilde y_h)$ to $\partial_{x_k} \tilde y$ in all spaces $H^1(0,T;H^s(\Omega))$ with $s <1$.
	We then remark that
	\begin{equation}
		\label{NormalDerivative}
		\partial_\nu \eh(\tilde y_h) = \left( \begin{array}{c} \eh( \partial_{h,1}^{\mp} \tilde y_h) \\ \eh (\partial_{h,2}^{\mp} \tilde y_h) \end{array}\right)\cdot \nu \quad \hbox{ on } \Gamma_{\pm}, 
	\end{equation}
	where $\nu$ is the normal vector to $\Omega$ on $\Gamma_{\pm}$. But the sequence $\eh(\partial_{h,k}^\mp \tilde y_h)$ strongly converges to $\partial_{x_k} \tilde y$ in $H^1(0,T; H^{3/4}(\Omega))$ and the trace operator is continuous from $H^{3/4}(\Omega)$ to $L^2(\partial \Omega)$ (see \cite[Thm 1.5.2.1]{Grisvard}). Therefore, $\partial_\nu \eh \tilde y_h$ strongly converges to $\partial_\nu y$ in $H^1(0,T; L^2(\partial \Omega))$. 
	
	One also easily checks that, since $\tilde y \in H^2(0,T; H^1(\Omega))$, the discrete function $\partial_{h,k}^+ \partial_{tt} \tilde y_h$ ($k \in \{1, 2\}$) is uniformly bounded in $L^2(0,T;L^2_h(\Omega_{h,k}^-))$. Hence $h \nabla \eh(\partial_{tt} \tilde y_h) $ strongly converges to~$0$ as $h \to 0$ in $L^2((0,T) \times \Omega)$.
	
	We then study the convergence of the normal derivative of $z_h$ and of $h \nabla \eh(\partial_{tt} z_h) $. We have seen that $z_h $ is uniformly bounded in $H^2(0,T; H^1_{0,h}(\Omega_h))\cap H^1(0,T; H^2_h(\Omega_h))$.  This immediately implies that $\partial_{h,k}^+ \partial_{tt} z_h$ is uniformly bounded in $L^2(0,T;L^2_h(\Omega_{h,k}^-))$ for $k \in \{1, 2\}$ and, following, $h \nabla \eh(\partial_{tt} z_h)$ strongly converges to $0$ in $L^2((0,T) \times \Omega)$ as $h \to 0$. Let us then remark that $\eh(q_h)$ and $ \eh(\tilde q_h - q_h)$ respectively converges to $q, \, \tilde q - q$ as $h \to 0$ strongly in $L^2(\Omega)$, weakly in $H^1(\Omega)$ and weakly-$*$ in $L^\infty(\Omega)$. Besides, as $\tilde y \in H^2(0,T;H^1(\Omega))$, $\eh(\tilde y_h )$ strongly converges to $\tilde y$ in $H^2(0,T; H^s(\Omega))$ for all $s \in [0,1)$. Following,
	\begin{align}
		& \eh (q_h ) \underset{h \to 0}{\longrightarrow} q 
		\quad \hbox{ strongly in all } L^p(\Omega) \hbox{ with } p <\infty,
		\\
		& \eh((\tilde q_h - q_h) \tilde y_h ) \underset{h \to 0}{\longrightarrow} (\tilde q -q)  \tilde y 
		\quad \hbox{ strongly in } H^2(0,T; L^2( \Omega)),
		\\
		& \eh((\tilde q_h - q_h) y_h^0) \underset{h \to 0}{\longrightarrow} (\tilde q -q) y^0 
		\quad \hbox{ strongly in } L^2(\Omega).
	\end{align}
	Easy computations then yields that $\eh(z_h)$ and  $\eh(\partial_t z_h)$ strongly converge in $H^1((0,T) \times \Omega)$ to $z$ and $\partial_t z$, where $z$ is the solution of \eqref{WaveEq-z}. This can indeed be done in three steps: First show that it converges weakly in $\mathcal{D}'((0,T) \times \Omega)$ toward $z$ and $\partial_t z$; Second, use that the energy estimates imply that the convergence is actually weak in $H^1((0,T) \times \Omega)$ and in particular strong in $L^2(0,T; L^p( \Omega))$ for any $p < \infty$; Third, use the energy identity to show the convergence of the $H^1((0,T) \times \Omega)$ norm. 
	
	Hence $\eh(\partial_{h,k}^\mp z_h)$ strongly converges to $\nabla z$ in $H^1(0,T; L^2(\Omega))$. Recall that $z_h$ is also uniformly bounded in $H^1(0,T; H^2_h(\Omega_h))$, so that $\eh(\partial_{h,k}^\mp z_h)$ is uniformly bounded in $H^1(0,T; H^1(\Omega))$. Thus $\eh(\partial_{h,k}^\mp z_h)$ strongly converges to $\nabla z$ in $H^1(0,T; H^{3/4}(\Omega))$, so that formula \eqref{NormalDerivative} and the continuity of the trace operator from $H^{3/4}(\Omega)$ to $L^2(\partial \Omega)$ show the strong convergence of $\partial_\nu \eh(z_h)$ to $\partial_\nu z$ in $H^1(0,T; L^2(\partial\Omega))$. 
	
	Since $y[q] = \tilde y + z$, we have proved the convergence \eqref{Cond3-ConsGamma} for the sequence $y_h[q_h] = \tilde y_h + z_h$. 
\end{proof} 	

\begin{remark}\label{Rem-Reg-Y[Q]}
In this proof, let us emphasize that the construction of the sequence of source terms $\tilde f_h$ and $\tilde f_{\partial,h}$ in \eqref{SourceTerms} is not straightforward. But we point out that this is done explicitly from the knowledge of the trace $q_\partial$ of $q$ on $\partial \Omega$.

Note however that this happens because we have chosen to keep a presentation where the assumptions are set on the trajectory $y[q]$, and not directly on the data $(y^0, y^1), f, f_\partial$. But this other choice would not yield any improvement as the natural space to get $y[q] \in H^1(0,T; L^\infty(\Omega))$ in $2$-d is $y[q] \in H^1(0,T; H^2(\Omega))$, or $H^3((0,T) \times \Omega)$. According to \cite{LasieckaLionsTriggiani}, this would correspond to 
$$
	 y^0 \in H^3(\Omega),  \quad y^1 \in H^2(\Omega),
	\quad f \in \cap_{k = 0,1,2} W^{k,1}(0,T; H^{2-k}(\Omega)), \quad f_\partial \in H^3((0,T)\times\partial \Omega),
$$
with the compatibility conditions 
$$	\left.y^0\right|_{\partial \Omega} = f_\partial (t = 0), \quad \left.y^1\right|_{\partial \Omega}= \partial_t f_\partial(t = 0),
	\hbox{ and } \left.(f(t = 0) + \Delta y^0 - q y^0)\right|_{\partial \Omega} = \partial_{tt} f_\partial (t = 0).
$$
Of course, this latest compatibility condition is very strong and requires in particular the knowledge of $q$ on the boundary, as we also assumed in the approach of Lemma \ref{Lem-Consistency-Main}. But very likely, taking projections of all these data on the discrete mesh $\overline{\Omega_h}$ also yields a suitable sequence $(y^0_h, y^1_h, f_h, f_{\partial,h})$ satisfying conditions \eqref{Cond3-ConsGamma}--\eqref{Cond5-Conslog}, even if one would have to study in that case the convergence of the discrete wave equations with non-homogeneous boundary conditions, which to our knowledge has only been done in 1-d so far in \cite{ErvZuaContinuousApproach}.
\end{remark}

\bibliographystyle{elsarticle-num}

\begin{thebibliography}{}
\expandafter\ifx\csname url\endcsname\relax
  \def\url#1{\texttt{#1}}\fi
\expandafter\ifx\csname urlprefix\endcsname\relax\def\urlprefix{URL }\fi
\expandafter\ifx\csname href\endcsname\relax
  \def\href#1#2{#2} \def\path#1{#1}\fi

\end{thebibliography}


\begin{thebibliography}{10}

\bibitem{Bardos}
C.~Bardos, G.~Lebeau, and J.~Rauch.
\newblock Sharp sufficient conditions for the observation, control and
  stabilization of waves from the boundary.
\newblock {\em SIAM J. Control and Optim.}, 30(5):1024--1065, 1992.

\bibitem{Baudouin01}
L.~Baudouin.
\newblock Lipschitz stability in an inverse problem for the wave equation,
  2010, http://hal.archives-ouvertes.fr/hal-00598876/fr/.

\bibitem{BaudouinErvedoza11}
L.~Baudouin and S.~Ervedoza.
\newblock Convergence of an inverse problem for a 1-{D} discrete wave equation.
\newblock {\em SIAM J. Control Optim.}, 51(1):556--598, 2013.

\bibitem{BaudouinDeBuhanErvedoza}
L.~Baudouin, M.~De~Buhan, and S.~Ervedoza.
\newblock Global {C}arleman estimates for waves and applications.
\newblock {\em Comm. Partial Differential Equations}, 38(5):823--859, 2013.

\bibitem{BellassouedIP04}
M.~Bellassoued.
\newblock Global logarithmic stability in inverse hyperbolic problem by
  arbitrary boundary observation.
\newblock {\em Inverse Problems}, 20(4):1033--1052, 2004.

\bibitem{BellaChoulli-JMPA09}
M.~Bellassoued and M.~Choulli.
\newblock Logarithmic stability in the dynamical  inverse problem for the {S}chr\"odinger equation by arbitrary boundary
  observation, 
  \newblock {\em J. Math. Pures Appl.}, (9) \textbf{91} (2009), no.~3, 233--255.
  
\bibitem{BellaYam06}
M.~Bellassoued, M.~Yamamoto. 
\newblock Logarithmic stability in determination of a coefficient in an acoustic equation by arbitrary boundary observation, 
  \newblock {\em J. Math. Pures Appl.},  (9) \textbf{85} (2006), no.~2, 193--224. 
  
\bibitem{BoyerHubertLeRousseau}
F.~Boyer, F.~Hubert, and J.~Le~Rousseau.
\newblock Discrete {C}arleman estimates for elliptic operators and uniform
  controllability of semi-discretized parabolic equations.
\newblock {\em J. Math. Pures Appl. (9)}, 93(3):240--276, 2010.

\bibitem{BoyerHubertLeRousseau2}
F.~Boyer, F.~Hubert, and J.~Le~Rousseau.
\newblock Discrete carleman estimates for elliptic operators in arbitrary
  dimension and applications,.
\newblock {\em SIAM J. Control Optim.}, 48:5357--5397, 2010.

\bibitem{BoyerHubertLeRousseau3}
F. Boyer, F. Hubert, and J. Le~Rousseau.
\newblock Uniform controllability properties for space/time-discretized
  parabolic equations.
\newblock {\em Numer. Math.}, 118(4):601--661, 2011.

\bibitem{BLR-ANIHP13}
F.~Boyer and J.~Le~Rousseau.
\newblock Carleman estimates for semi-discrete
  parabolic operators and application to the controllability of semi-linear
  semi-discrete parabolic equations, 
  \newblock{\em Annales de l'Institut Henri Poincare (C)  Non Linear Analysis}, 
  (2013), 46p.

\bibitem{BuKli81}
A.~L. Bukhge{\u\i}m and M.~V. Klibanov.
\newblock Uniqueness in the large of a class of multidimensional inverse
  problems.
\newblock {\em Dokl. Akad. Nauk SSSR}, 260(2):269--272, 1981.

\bibitem{dBO-IP10}
M. de~Buhan and A. Osses, 
\newblock Logarithmic stability in determination of a  3{D} viscoelastic coefficient and a numerical example, 
 \newblock{\em Inverse Problems},
  \textbf{26} (2010), no.~9, 095006, 38. 
  
\bibitem{ErvGournay}
S.~Ervedoza and F.~de~Gournay.
\newblock Uniform stability estimates for the discrete Calder\'on problems.
\newblock {\em Inverse Problems}, 27(12):125012, 2011.

\bibitem{ErvZuaCime}
S.~Ervedoza and E.~Zuazua.
\newblock The wave equation: Control and numerics.
\newblock In P.~M. Cannarsa and J.~M. Coron, editors, {\em Control of Partial
  Differential Equations}, Lecture Notes in Mathematics, CIME Subseries.
  Springer Verlag, 2011.
  
\bibitem{ErvZuaContinuousApproach}
S.~Ervedoza and E.~Zuazua.
\newblock {\em Numerical approximation of exact controls for waves}.
\newblock Springer Briefs in Mathematics. Springer, New York, 2013.

\bibitem{FursikovImanuvilov}
A.~V. Fursikov and O.~Y. Imanuvilov.
\newblock {\em Controllability of evolution equations}, volume~34 of {\em
  Lecture Notes Series}.
\newblock Seoul National University Research Institute of Mathematics Global
  Analysis Research Center, Seoul, 1996.

\bibitem{Grisvard}
P.~Grisvard.
\newblock {\em Elliptic problems in nonsmooth domains}, volume~24 of {\em
  Monographs and Studies in Mathematics}.
\newblock Pitman (Advanced Publishing Program), Boston, MA, 1985.


\bibitem{Ho}
L.~F. Ho, 
\newblock {\em Observabilit\'e fronti\`ere de l'\'equation des ondes}, C. R.
  Acad. Sci. Paris S\'er. I Math. {302} (1986), no.~12, 443--446.
  
\bibitem{Hormander-III}
L.~H{\"o}rmander.
\newblock {\em The analysis of linear partial differential operators. {III}},
  volume 274 of {\em Grundlehren der Mathematischen Wissenschaften [Fundamental
  Principles of Mathematical Sciences]}.
\newblock Springer-Verlag, Berlin, 1985.
\newblock Pseudodifferential operators.

\bibitem{Im02}
O.~Y. Imanuvilov.
\newblock On {C}arleman estimates for hyperbolic equations.
\newblock {\em Asymptot. Anal.}, 32(3-4):185--220, 2002.

\bibitem{ImYamIP01}
O.~Y. Imanuvilov and M.~Yamamoto.
\newblock Global {L}ipschitz stability in an inverse hyperbolic problem by
  interior observations.
\newblock {\em Inverse Problems}, 17(4):717--728, 2001.
\newblock Special issue to celebrate Pierre Sabatier's 65th birthday
  (Montpellier, 2000).

\bibitem{ImYamCom01}
O.~Y. Imanuvilov and M.~Yamamoto.
\newblock Global uniqueness and stability in determining coefficients of wave
  equations.
\newblock {\em Comm. Partial Differential Equations}, 26(7-8):1409--1425, 2001.

\bibitem{ImYamIP03}
O.~Y. Imanuvilov and M.~Yamamoto.
\newblock Determination of a coefficient in an acoustic equation with a single
  measurement.
\newblock {\em Inverse Problems}, 19(1):157--171, 2003.

\bibitem{InfZua}
J.A. Infante and E.~Zuazua.
\newblock Boundary observability for the space semi discretizations of the 1-d
  wave equation.
\newblock {\em Math. Model. Num. Ann.}, 33:407--438, 1999.

\bibitem{KazemiKlibanov}
M.~A. Kazemi and M.~V. Klibanov.
\newblock Stability estimates for ill-posed {C}auchy problems involving
  hyperbolic equations and inequalities.
\newblock {\em Appl. Anal.}, 50(1-2):93--102, 1993.

\bibitem{KlibanovSantosa}
M.~V. Klibanov and F.~Santosa.
\newblock A computational quasi-reversibility method for {C}auchy problems for
  {L}aplace's equation.
\newblock {\em SIAM J. Appl. Math.}, 51(6):1653--1675, 1991.

\bibitem{LasieckaLionsTriggiani}
I.~Lasiecka, J.-L. Lions, and R.~Triggiani.
\newblock Nonhomogeneous boundary value problems for second order hyperbolic
  operators.
\newblock {\em J. Math. Pures Appl. (9)}, 65(2):149--192, 1986.

\bibitem{LebRob97}
G. Lebeau and L. Robbiano.
\newblock Stabilisation de l'\'equation des ondes par le bord.
\newblock {\em Duke Math. J.}, 86(3):465--491, 1997.

\bibitem{Lions}
J.-L. Lions.
\newblock {\em Contr\^olabilit\'e exacte, Stabilisation et Perturbations de
  Syst\`emes Distribu\'es. Tome 1. Contr\^olabilit\'e exacte}, volume RMA 8.
\newblock Masson, 1988.

\bibitem{Phung09}
K.~D. Phung.
\newblock Waves, damped wave and observation.
\newblock In Ta-Tsien Li, Yue-Jun Peng, and Bo-Peng Rao, editors, {\em Some
  Problems on Nonlinear Hyperbolic Equations and Applications}, Series in
  Contemporary Applied Mathematics CAM 15, 2010.
  
\bibitem{PuelYam96}
J.-P. Puel and M.~Yamamoto.
\newblock On a global estimate in a linear inverse hyperbolic problem.
\newblock {\em Inverse Problems}, 12(6):995--1002, 1996.

\bibitem{PuelYam97}
J.-P. Puel and M.~Yamamoto.
\newblock Generic well-posedness in a multidimensional hyperbolic inverse
  problem.
\newblock {\em J. Inverse Ill-Posed Probl.}, 5(1):55--83, 1997.

\bibitem{Robbiano91}
L.~Robbiano.
\newblock Th\'eor\`eme d'unicit\'e adapt\'e au contr\^ole des solutions des
  probl\`emes hyperboliques.
\newblock {\em Comm. Partial Differential Equations}, 16(4-5):789--800, 1991.

\bibitem{Robbiano95}
L.~Robbiano.
\newblock Fonction de co\^ut et contr\^ole des solutions des \'equations
  hyperboliques.
\newblock {\em Asymptotic Anal.}, 10(2):95--115, 1995.

\bibitem{StefUhlmann09}
P.~Stefanov and G.~Uhlmann.
\newblock Recovery of a source term or a speed with one measurements and applications.
\newblock {\em Trans. of A.M.S.}, 0002-9947(2013)05703-0.

\bibitem{Tref}
L.~N. Trefethen.
\newblock Group velocity in finite difference schemes.
\newblock {\em SIAM Rev.}, 24(2):113--136, 1982.

\bibitem{TWbook}
M.~Tucsnak and G.~Weiss.
\newblock {\em Observation and Control for Operator Semigroups}, volume~XI of
  {\em Birk{\"a}user Advanced Texts}.
\newblock Springer, 2009.

\bibitem{Yam99}
M.~Yamamoto.
\newblock Uniqueness and stability in multidimensional hyperbolic inverse
  problems.
\newblock {\em J. Math. Pures Appl. (9)}, 78(1):65--98, 1999.

\bibitem{Zua05Survey}
E.~Zuazua.
\newblock Propagation, observation, and control of waves approximated by finite
  difference methods.
\newblock {\em SIAM Rev.}, 47(2):197--243 (electronic), 2005.\\

\end{thebibliography}

\appendix

\section{Discrete integration by parts formula in 1-d}\label{Sec-IPP}

For the sake of completeness, we mention the basic discrete integration by parts formula obtained in \cite[Lemma 2.6]{BaudouinErvedoza11} in the 1-d setting as they are the main ingredients used to perform integration by parts on 2-d (and higher dimensional) domains. To do so, we shall make precise some 1-d notations.

We assume that we consider integration by parts on discretized versions of $(0,1)$. For $N \in \mathbb{N}$, we introduce $h = 1/(N+1)$ and the discrete sets
$$
	(0,1)_h = \{jh, \, j \in \llbracket 1,  N\rrbracket \}, \quad [0,1)_h = \{jh, \, j \in \llbracket 0,  N\rrbracket \}, \quad (0,1]_h = \{jh, \, j \in \llbracket 1,  N+1\rrbracket \}.
$$
Here, discrete functions $f_h$ are functions $f_h = (f_{j})_{j \in \{0,\cdots, N+1\}}$ for which we define
$$
	\int_{(0,1)_h} f_h =h \sum_{j\in \{1,\cdots, N\} } f_j, \quad \int_{[0,1)_h} f_h = h\sum_{j\in \{0,\cdots, N\} } f_j, \quad \int_{(0,1]_h} f_h = h\sum_{j\in \{1,\cdots, N+1\} } f_j.
$$
We also introduce the discrete operators for $j\in\{1,\ldots,N\}$: 
\begin{align*}
	& (m^+_h f_h)_j = (m_h^- f_h)_{j+1} = \dfrac{f_{j+1} + f_{j}}{2}~;
\\
	& (\partial_h f_h)_{j} = \dfrac{f_{j+1} -f_{j-1}}{2h}~; \quad
	(\partial^+_h f_h)_{j} = (\partial^-_h f_h)_{j+1} = \dfrac{f_{j+1} -f_{j}}{h}  ~; \quad
	(\Delta_h f_h)_j  = \dfrac{f_{j+1} - 2 f_j+f_{j-1}}{h^2}.
\end{align*}

\begin{lemma}[\cite{BaudouinErvedoza11}, 1-d discrete integration by parts formulas]
	\label{LemIPP2-fromBE}
	Let $v_h,f_h,g_h$ be discrete functions such that $v_0=v_{N+1}=0$. Then we have the following identities:
	\begin{align}
		\label{IPP1}
		&\bullet~ \int_{[0,1)_h} g_h (\partial_h^+ f_h)=  - \int_{(0,1]_h} (\partial_h^- g_h) f_h  + g_{N+1} f_{N+1}  - g_0 f_0 ~ ;
		\\
		\label{IPP2}
		&\bullet~  \int_{(0,1)_h} g_h (\partial_h f_h)=  \int_{[0,1)_h} (m_h^+ g_h)(\partial_h^+ f_h)
		 - \frac{h}{2} g_0 (\partial_h^+ f)_0- \frac{h}{2} g_{N+1} (\partial_h^- f)_{N+1} ~ ;
		\\
		\label{IPP3}
		&\bullet~   2 \int_{(0,1)_h } g_h v_h (\partial_h v_h)  = - \int_{(0,1)_h}|v_h|^2 ~\partial_h g_h
		+ \dfrac{h^2}2 \int_{[0,1)_h} |\partial_h^+ v_h|^2 \partial_h^+g_h ~ ;
		\\
		\label{IPP4}
		 &\bullet~ \int_{(0,1)_h} g_h(\Delta_h v_h)  = - \int_{[0,1)_h}(\partial^+_h v_h) ~(\partial^+_h g_h)
		 - (\partial_h^+ v)_0 g_0+  (\partial_h^- v)_{N+1} g_{N+1} ~ ;
		\\
		\label{IPP5}
		 &\bullet~ \int_{(0,1)_h} g_h v_h(\Delta_h v_h)  = - \int_{[0,1)_h}(\partial^+_h v_h)^2 ~(m^+_h g_h) + \frac{1}{2} \int_{(0,1)_h} |v_h|^2 \Delta_h g_h ~ ;
		\\
		\label{IPP6}
		 &\bullet~ \int_{(0,1)_h} g_h \Delta_h v_h \partial_h v_h =- \dfrac 12\int_{[0,1)_h} |\partial_h^+ v_h|^2 \partial_h^+g_h 
		 + \dfrac 12 \left|(\partial_h^- v)_{N+1} \right|^2 g_{N+1} - \dfrac 12\left|(\partial_h^+ v)_0 \right|^2 g_{0}.
	\end{align}
\end{lemma}
In a square in dimension 2, we will apply Lemma \ref{LemIPP2-fromBE} when doing integrations by part in each direction. For instance, identity \eqref{IPP3} 
easily yields, for $k \in \{1, 2\}$:
$$
	2\int_{\Omega_h} g_h v_h (\partial_{h,k} v_h)=  -\int_{\Omega_{h}} (\partial_{h,k} g_h) |v_h|^2
		+ \dfrac{h^2}2 \int_{\Omega_{h,k}^-} |\partial_{h,k}^+ v_h|^2 \partial_{h,k}^+g_h.
$$
For convenience, we will also use the formula $\int_{[0,1)_h} m_h^+ v_h f_h = \int_{(0,1]_h} v_h m_h^- f_h$, valid for $v_h$ vanishing on the boundary, and its consequence
\begin{equation}
	\label{IPPnew}
		  \int_{[0,1)_h} m_h^+ v_h (\partial_h^+ f_h) (\partial_h^+ g_h) = \int_{(0,1)_h } v_h (\partial_h f_h)  (\partial_h g_h)  
		+ \dfrac{h^2}4 \int_{(0,1)_h} v_h(\Delta_h f_h)(\Delta_h g_h), 
\end{equation}
whose proof is left to the reader.

\section{Proof of a conjugate Carleman estimate
}\label{Sec-Proof-Prop-Decompo}
\begin{proof}[Proof of Proposition \ref{PropDecompo}]

{\it Notations.} In this proof, we will use the Landau notation $\mathcal{O}_{\mu}(\tau h)$ to denote discrete functions of $(t,x_h)$ depending on $\mu$ satisfying for some constant $C_\mu>0$ that 
$$\norm{\mathcal{O}_\mu(\tau h)}_{L^\infty( L^\infty_h)} \leq C_\mu \tau h.$$ We will also use the shortcut $\mathcal{O}_{\mu}(1)$ to denote bounded functions. Moreover, we will write $v$ instead of $v_h$ as no confusion can occur: here, $v$ is always a discrete function defined on $(-T,T) \times \overline{\Omega_h}$ satisfying $v(\pm T, x_h) = \partial_t v(\pm T, x_h) = 0$ for all $x_h \in \Omega_h$ and $v(t,x_h) = 0$ for all $t \in (-T,T)$ and $x_h \in \partial \Omega_h$.
In order to simplify the integrals, we will also set $Q_h=(-T,T)\times\Omega_h$, $Q_{h,k}^\pm=(-T,T)\times\Omega_{h,k}^\pm$,
$\Sigma_h=(-T,T)\times\Gamma_h$, $\Sigma_{h,k}^\pm=(-T,T)\times\Gamma_{h,k}^\pm$ and use the notations
$$
	\int_{Q_h} = \int_{-T}^T \int_{\Omega_h}, \qquad \int_{Q_{h,k}^\pm} = \int_{-T}^T \int_{\Omega_{h,k}^\pm}, 
	\qquad
	\int_{\Sigma_h} = \int_{-T}^T \int_{\Gamma_h}, \qquad \int_{\Sigma_{h,k}^\pm} = \int_{-T}^T \int_{\Gamma_{h,k}^\pm}.
\medskip
$$

In the following we will use the estimates of Proposition \ref{Prop-B}, in particular \eqref{ProxC2}, and the discrete integration by parts formula in Lemma \ref{LemIPP2-fromBE} and Lemma \ref{LemIPP2}. Finally, let us emphasize that all the constants below are independent of $h \in (0,1)$ and $\tau\geq 1$.\\

\noindent {\it $\bullet$ Step 1. Explicit computations of the cross product.}
The proof of estimate \eqref{decompo} relies first of all on the computation of the multiplication of each term of $\mathscr{L}_{h,1}v$ by each term of $\mathscr{L}_{h,2}v$:
$$
	\ds\int_{Q_h} \mathscr{L}_{h,1}v \, \mathscr{L}_{h,2}v\, dt  = \sum_{n,m=1}^3 I_{nm},
$$ 
where $I_{nm}$ denotes the product between  the $n$-th term of $\mathscr{L}_{h,1}$ in \eqref{P1} and the $m$-th term of $\mathscr{L}_{h,2}$ in \eqref{P2}.
We now perform the computation of each $I_{nm}$ term. 
\\
\noindent{\it Computation of $I_{11}$.} As in \cite{BaudouinErvedoza11}, we integrate by parts in time:
\begin{align*}
	I_{11}
	=~&
	(\alpha_1  - 1)\tau \mu \int_{Q_h}\!\partial_{tt} v \left( \varphi \partial_{tt} \psi -A_{4} \right)v 
	\\	
	=&~
	(1-\alpha_1 )\tau\mu \int_{Q_h}\!| \partial_t v|^2 \varphi(\partial_{tt} \psi - \Delta \psi) 
	+  \tau \int_{Q_h}\! \mathcal O_\mu(1) |v|^2  + \tau \int_{Q_h}\! \mathcal O_\mu(\tau h)|\partial_t v|^2.
\end{align*}
Here, we used $ A_{4}  = \varphi \Delta \psi +  \mathcal O_{\mu}(\tau h)$ and 
$\partial_{tt} A_{4} =\partial_{tt}(\varphi \Delta \psi)   +  \mathcal O_{\mu}(\tau h)$.
\\
\noindent{\it Computation of $I_{12}$.} Similarly,
\begin{align*}
	I_{12}
	=~&
	-\tau\mu^2\int_{Q_h}\!\partial_{tt}v \left(\varphi |\partial_t \psi|^2- A_3 \right)v 
	\\
	=~&
	\tau \mu ^2 \int_{Q_h}\!|\partial_t v|^2\varphi(|\partial_t \psi|^2-|\nabla \psi|^2)+ \tau \int_{Q_h}\!  \mathcal O_\mu(1)|v|^2
	+ \tau \int_{Q_h}\!\mathcal O_\mu(\tau h)|\partial_t v|^2 ,
\end{align*}
where we used $ A_{3} = \varphi |\nabla \psi|^2+  \mathcal O_{\mu}(\tau h)$ and $\partial_{tt} A_{3} =\partial_{tt}\left(\varphi |\nabla \psi|^2\right) +  \mathcal O_{\mu}(\tau h).$
\\
\noindent{\it Computation of $I_{13}$.} 
Using $ {\textstyle\sum_{k}} \partial_{h,k} A_{1,k} = \mu \varphi |\nabla \psi|^2 + \varphi \Delta \psi + \mathcal O_\mu(\tau h)$,  $  \partial_{t} A_{1,k} = \mu \varphi \partial_{x_k}\psi \partial_t \psi+ \mathcal O_\mu(\tau h)$, and \eqref{IPP3}, we obtain:
\begin{align*}
	I_{13}
		=~&-2\tau\mu \int_{Q_h}\!\partial_{tt}v \left( \varphi \partial_t \psi  \partial_t v - {\textstyle\sum_{k}} A_{1,k} \partial_{h,k} v\right) 
	\\
		=~& \tau\mu \int_{Q_h}\! | \partial_t v|^2\varphi (\partial_{tt} \psi +\Delta\psi)
	+\tau\mu^2 \int_{Q_h}\! | \partial_t v|^2\varphi (|\partial_t \psi|^2 + |\nabla \psi|^2 )
	\\
	&-2\tau\mu^2 \int_{Q_h}\!  \partial_t v\, \partial_t\psi\, \varphi~\nabla_h v\cdot \nabla \psi 
	-  \dfrac{\tau\mu}2 {\textstyle\sum_{k}} \int_{Q_{h,k}^-}  |h\partial_{h,k}^+ \partial_t v|^2 \partial_{h,k}^+ A_{1,k}
	\\
	&
	+\tau \int_{Q_h}\!\mathcal O_\mu(\tau h)| \partial_t v|^2 
	+\tau\int_{Q_h}\! \partial_t v \left({\textstyle\sum_{k}} \mathcal O_\mu(\tau h) \partial_{h,k} v\right).
\end{align*}
\noindent{\it Computation of $I_{21}$.} Since $A_{4} = \varphi \Delta \psi + \mathcal O_\mu(\tau h)$ and $A_{0,k}= \mathcal O_\mu(\tau h)$, we get:
\begin{align*}
	I_{21}
	=~&
	(1- \alpha_1 )\tau \mu \int_{Q_h}\! {\textstyle\sum_{k}} (1+A_{0,k})\Delta_{h,k} v \left( \varphi \partial_{tt} \psi -A_{4} \right)v 
	\\
	= ~&(\alpha_1 -1) \tau \mu  {\textstyle\sum_{k}} \int_{Q_{h,k}^-} |\partial_{h,k}^+ v|^2 \varphi (\partial_{tt} \psi - \Delta \psi ) 
	+ \tau \int_{Q_h}\!  \mathcal{O}_\mu(1)|v|^2
	+\tau  {\textstyle\sum_{k}}\int_{Q_{h,k}^-}\mathcal O_\mu(\tau h)|\partial_{h,k}^+ v|^2.
\end{align*}
\\
\noindent{\it Computation of $I_{22}$.}  Using $A_{3} = \varphi |\nabla\psi|^2 +  \mathcal{O}_\mu(\tau h)$ and \eqref{IPP5}, we obtain
\begin{align*}
	I_{22} =~&
		\tau\mu^2\int_{Q_h}\!{\textstyle\sum_{k}} (1+A_{0,k})\Delta_{h,k} v  \left( \varphi |\partial_t \psi|^2 -  A_{3} \right)v 
		\\
		=~&-\tau\mu^2 {\textstyle\sum_{k}} \int_{Q_{h,k}^-} |\partial_{h,k}^+ v|^2 \varphi (|\partial_t \psi|^2 - |\nabla \psi|^2) 
		+ \tau  \int_{Q_h}\! \mathcal{O}_\mu(\tau h)|v|^2
		+\tau {\textstyle\sum_{k}} \int_{Q_{h,k}^-} \mathcal O_\mu(\tau h)|\partial_{h,k}^+ v|^2  .
\end{align*}
\noindent{\it Computation of $I_{23}$.}
We can split this term in two parts as follows
\begin{align*}
	I_{23} &= \underbrace{2\tau\mu\int_{Q_h}\!{\textstyle\sum_{k}} (1+A_{0,k})\Delta_{h,k} v\,  \varphi \, \partial_t \psi \, \partial_t v}_{I_{23a}}
		 \ \underbrace{- 2\tau\mu\int_{Q_h}\!{\textstyle\sum_{k}} (1+A_{0,k})\Delta_{h,k} v  \left({\textstyle\sum_{\ell}} A_{1,\ell} \partial_{h,\ell}  v\right).}_{I_{23b}}
\end{align*}

For $I_{23a}$ we use $\Delta_{h,k} = \partial_{h,k}^- \partial_{h,k}^+$ and the zero boundary conditions on $v$. Setting $g_{0,k} =   (1+A_{0,k})\,  \varphi \, \partial_t \psi$ and using \eqref{IPP1}, we get:
\begin{align*}
	I_{23a} =~& - 2 \tau \mu {\textstyle\sum_{k}} \int_{Q_{h,k}^-}  \partial_{h,k}^+ v\,  \partial_{h,k}^+( g_{0,k}  \partial_t v)
	\\
	=~& - 2 \tau \mu {\textstyle\sum_{k}} \int_{Q_{h,k}^-}  \partial_{h,k}^+ v\,  \partial_{h,k}^+( \partial_t v) m_{h,k}^+g_{0,k}
	 -  2 \tau \mu {\textstyle\sum_{k}} \int_{Q_{h,k}^-}  \partial_{h,k}^+ v\,  m_{h,k}^+( \partial_t v) \partial_{h,k}^+ g_{0,k}
\end{align*}	
Noticing that, on the one hand,
\begin{align*}
	 - 2 \tau \mu {\textstyle\sum_{k}} \int_{Q_{h,k}^-}  \partial_{h,k}^+ v\,  \partial_{h,k}^+( \partial_t v) m_{h,k}^+g_{0,k}
	& =
	   \tau \mu {\textstyle\sum_{k}} \int_{Q_{h,k}^-}  |\partial_{h,k}^+ v|^2 \partial_t (m_{h,k}^+ g_{0,k})
	  \\
	 & =  \tau \mu {\textstyle\sum_{k}} \int_{Q_{h,k}^-}  |\partial_{h,k}^+ v|^2 (\mu \varphi |\partial_t \psi|^2 + \varphi \partial_{tt} \psi+ \mathcal{O}_\mu(\tau h)),
\end{align*}
and on the other hand (using \eqref{IPPnew}),
\begin{align*}
	- 2 \tau \mu {\textstyle\sum_{k}} \int_{Q_{h,k}^-}  &
		\partial_{h,k}^+ v\,  m_{h,k}^+(\partial_t v)  \partial_{h,k}^+g_{0,k}
	\\
	=~&
	 -2 \tau \mu \sum_k \int_{Q_{h}}  
		\partial_{h,k} v\, \partial_t v \, \partial_{h,k}g_{0,k}
	- \frac{\tau \mu h^2}{2} 
		 \sum_k \int_{Q_{h}}  
		\Delta_{h,k} v\, \partial_t v \, \Delta_{h,k} g_{0,k}
	\\
	=~&
	 -2 \tau \mu^2 \int_{Q_{h}}  \partial_t v\, \partial_t \psi\, \varphi  \, \nabla_h v \cdot (\nabla \psi+ \mathcal{O}_\mu(\tau h)) 
	 - \tau h^2 \sum_k \int_{Q_{h}}  \mathcal{O}_\mu(1)\Delta_{h,k} v\, \partial_t v, 
\end{align*}
the term $I_{23a}$ takes the form
\begin{align*}
	I_{23a} =~& 
	\tau \mu^2 {\textstyle\sum_{k}} \int_{Q_{h,k}^-}  |\partial_{h,k}^+ v|^2  \varphi |\partial_t \psi|^2 
	+ 	\tau \mu {\textstyle\sum_{k}} \int_{Q_{h,k}^-}  |\partial_{h,k}^+ v|^2 ( \varphi \partial_{tt} \psi+ \mathcal{O}_\mu(\tau h))
	\\
	 & -2 \tau \mu^2 \int_{Q_{h}}  \partial_t v\, \partial_t \psi\, \varphi  \, \nabla_h v \cdot (\nabla \psi+ \mathcal{O}_\mu(\tau h)) 
	 - \tau h^2 \sum_k \int_{Q_{h}}  \mathcal{O}_\mu(1)\Delta_{h,k} v\, \partial_t v.
\end{align*}

To compute $I_{23b}$, we consider the integrals $I_{23b, k , \ell}$ indexed by $(k,\ell) \in \{1,2\}^2$ and defined by
$$
	I_{23b, k, \ell} = - 2\tau\mu\int_{Q_h}\! (1+A_{0,k})\Delta_{h,k} v\, A_{1,\ell} \, \partial_{h,\ell}  v.
$$
When $k = \ell$, using formula \eqref{IPP6} with $g_{k} = (1+A_{0,k})A_{1,k} = \varphi \partial_{x_k}\psi (1+\mathcal O_\mu(\tau h))$, we obtain
\begin{align*}
	I_{23b, k, k} =~& \tau \mu \int_{Q_{h,k}^-} |\partial_{h,k}^+ v|^2 \partial_{h,k}^+g_k - \tau \mu \int_{\Sigma_{h,k}^+}g_k |\partial_{h,k}^- v|^2 + \tau \mu \int_{\Sigma_{h,k}^-} g_k |\partial_{h,k}^+ v|^2
	\\
	=~& \tau \mu \int_{Q_{h,k}^-} |\partial_{h,k}^+ v|^2( \partial_{x_k}(\varphi \partial_{x_k} \psi) + \mathcal{O}_\mu(\tau h))- \tau \mu \int_{\Sigma_{h,k}^+}g_k |\partial_{h,k}^- v|^2 + \tau \mu \int_{\Sigma_{h,k}^-} g_k |\partial_{h,k}^+ v|^2.
\end{align*}
When $k \neq \ell$, we use Lemma \ref{LemIPP2} with $g_{k, \ell} = (1+A_{0,k})A_{1, \ell}$:
\begin{align*}
	I_{23b, k, \ell} =~& - \tau \mu \int_{Q_{h,k}^-} |\partial_{h,k}^+ v|^2 \partial_{h,\ell} (m_{h,k}^+ g_{k, \ell})  + 2 \tau \mu  \int_{Q_{h,k}^-} \partial_{h,k}^+ v\, m_{h,k}^+ (\partial_{h, \ell} v) \partial_{h,k}^+ g_{k, \ell}   
	\\
	& + \frac{\tau \mu h^2}{2} \int_{Q_h^-} |\partial_{h,k}^+ \partial_{h, \ell}^+ v|^2 \partial_{h, \ell}^+ (m_{h,k}^+ g_{k, \ell}).
\end{align*}
Using \eqref{IPPnew} for  $v_h$ replaced by $\partial_{h, \ell} v$, which vanishes on the boundary $\Sigma_{h,k}$ as $k\neq \ell$,
we get:
\begin{align*}
	I_{23b,k, \ell} =~&- \tau \mu \int_{Q_{h,k}^-} |\partial_{h,k}^+ v|^2 \left( \partial_{x_\ell} (\varphi \partial_{x_\ell} \psi)  + \mathcal{O}_\mu(\tau h) \right)
	+ 2 \tau \mu \int_{Q_h}  \partial_{h,k} v\,  \partial_{h, \ell} v \, \left(\partial_{x_k}(\varphi \partial_{x_\ell} \psi)+ \mathcal{O}_\mu(\tau h) \right)
	\\
	& 
	+\frac{\tau \mu h^2}{2} \int_{Q_{h}} \Delta_{h,k} v\,  \partial_{h, \ell} v  \left(\Delta_{x_k} (\varphi \partial_{x_\ell} \psi) + \mathcal{O}_\mu(\tau h)\right)
	+ \frac{\tau \mu h^2}{2} \int_{Q_h^-} |\partial_{h,k}^+ \partial_{h, \ell}^+ v|^2 \partial_{h, \ell}^+ (m_{h,k}^+ g_{k, \ell}).
\end{align*}
Hence we obtain
\begin{align*}
	I_{23b} =~& 
	\tau \mu {\textstyle\sum_{k}} \int_{Q_{h,k}^-} |\partial_{h,k}^+ v|^2 \left(\partial_{x_k} (\varphi  \partial_{x_k} \psi) - {\textstyle \sum_{\ell \neq k}} \partial_{x_\ell} (\varphi \partial_{x_\ell} \psi) + \mathcal{O}_\mu(\tau h) \right)
	\\
	& + 
	2 \tau \mu \int_{Q_h} \partial_{h,1} v\, \partial_{h,2} v \left(\partial_{x_1}(\varphi \partial_{x_2} \psi)+ \partial_{x_2}(\varphi \partial_{x_1} \psi) + \mathcal{O}_\mu(\tau h)\right)
	\\
	& 
	+ \tau h^2 \int_{Q_{h}} \mathcal{O}_\mu(1) (\Delta_{h,1} v \, \partial_{h,2} v + \Delta_{h,2} v \, \partial_{h,1} v)
	+ \frac{\tau \mu h^2}{2} \int_{Q_h^-}  |\partial_{h,1}^+ \partial_{h, 2}^+ v|^2 \left( \hbox{div} (\varphi \nabla \psi) + \mathcal{O}_\mu(\tau h) \right)
	\\
	& - \tau \mu \sum_k \int_{\Sigma_{h,k}^+}  |\partial_{h,k}^- v|^2 \varphi \partial_{x_k} \psi (1+ \mathcal{O}_\mu(\tau h))+ \tau \mu \sum_k \int_{\Sigma_{h,k}^-}  |\partial_{h,k}^+ v|^2\varphi \partial_{x_k} \psi (1+ \mathcal{O}_\mu(\tau h)).
\end{align*}
We now remark that $\partial_{x_1}(\varphi \partial_{x_2} \psi)+ \partial_{x_2}(\varphi \partial_{x_1} \psi) = 2 \mu \varphi \partial_{x_1} \psi \partial_{x_2} \psi$, and that we can write
$$
	4 \tau \mu^2 \int_{Q_h} \partial_{h,1} v\, \partial_{h,2} v\, \varphi \partial_{x_1} \psi \partial_{x_2} \psi 
	= 
	2 \tau \mu^2 \int_{Q_h} \varphi |\nabla_h v\cdot \nabla \psi|^2 - 2 \tau \mu^2 \sum_k \int_{Q_h} |\partial_{h,k} v|^2 |\partial_{x_k} \psi|^2.
$$
Therefore,
\begin{align*}
	I_{23b} =~& 
	\tau \mu {\textstyle\sum_{k}} \int_{Q_{h,k}^-} |\partial_{h,k}^+ v|^2 \left( {2} \partial_{x_k} (\varphi  \partial_{x_k} \psi) - \hbox{div} (\varphi \nabla \psi) + \mathcal{O}_\mu(\tau h) \right)
	\\
	& 
	+ 2 \tau \mu^2 \int_{Q_h} \varphi |\nabla_h v\cdot \nabla \psi|^2 - 2 \tau \mu^2 \sum_k \int_{Q_h} |\partial_{h,k} v|^2 \varphi |\partial_{x_k} \psi|^2 + \tau \int_{Q_h} \mathcal{O}_\mu(\tau h ) \partial_{h,1} v \, \partial_{h,2} v 
	\\
	&
	+ \tau h^2 \int_{Q_{h}} \left(\mathcal{O}_\mu(1) \Delta_{h,1} v \,  \partial_{h,2} v   + \mathcal{O}_\mu(1) \Delta_{h,2} v \, \partial_{h,1} v\right)
	\\
	& 
	+ \frac{\tau \mu h^2}{2} \int_{Q_h^-}  |\partial_{h,1}^+ \partial_{h, 2}^+ v|^2 \left( \hbox{div} (\varphi \nabla \psi) + \mathcal{O}_\mu(\tau h) \right)
	\\
	& - \tau \mu \sum_k \int_{\Sigma_{h,k}^+}  |\partial_{h,k}^- v|^2 (\varphi \partial_{x_k} \psi + \mathcal{O}_\mu(\tau h))+ \tau \mu \sum_k \int_{\Sigma_{h,k}^-}  |\partial_{h,k}^+ v|^2( \varphi \partial_{x_k} \psi+ \mathcal{O}_\mu(\tau h)).
\end{align*}
Of course, this yields $I_{23}$ as $I_{23} = I_{23a}+ I_{23b}$.\\
\noindent{\it Computation of $I_{31}$.}
Using $\ds  A_{2} = \varphi^2 |\nabla \psi|^2 + O_\mu(\tau h)$ and $\ds A_{4} = \varphi \Delta\psi + O_\mu(\tau h)$, 
one easily obtains:
\begin{align*}
	I_{31}
	=~&(\alpha_1  - 1)\tau^3\mu^3\int_{Q_h}\!  |v|^2 \left(\varphi^2 \left( \partial_t \psi \right)^2 - A_{2} \right)
		\left( \varphi \, \partial_{tt} \psi - A_{4}\right)
	\\
	=~&(\alpha_1 -1) \tau^3\mu^3 \int_{Q_h}\!|v|^2\varphi^3(|\partial_t \psi|^2-|\nabla\psi|^2) (\partial_{tt} \psi - \Delta \psi) 
	+\tau^3\int_{Q_h}\! \mathcal O_\mu(\tau h) |v|^2. 
\end{align*}
\noindent{\it Computation of $I_{32}$.} Using here $\ds A_{3} = \varphi |\nabla \psi|^2 + O_\mu(\tau h)$, 
\begin{align*}
	I_{32}
		=~&
		-\tau^3\mu^4\int_{Q_h}\!  |v|^2 \left(\varphi^2 \left( \partial_t \psi \right)^2 - A_{2} \right)
		\left( \varphi |\partial_t \psi|^2 - A_{3} \right)
		\\
		=~&
		-\tau^3\mu^4 \int_{Q_h}\! |v|^2 \varphi^3 (|\partial_t \psi|^2-|\nabla\psi|^2)^2
		+\tau^3 \int_{Q_h}\! \mathcal{O}_{\mu} (\tau h)|v|^2 .
\end{align*}
\noindent{\it Computation of $I_{33}$.}
Finally, using \eqref{IPP3} we get
\begin{align*}
	I_{33}
		=~&
		-2\tau^3\mu ^3\int_{Q_h}\!  \left(\varphi^2 \left( \partial_t \psi \right)^2 -  A_{2} \right)v
		\left( \varphi \, \partial_t \psi  \, \partial_t v - {\textstyle\sum_{k}} A_{1,k} \partial_{h,k} v\right)
		\\
		=~&
		\tau^3\mu ^3 \int_{Q_h}\!|v|^2 \partial_t\left((\varphi^2|\partial_t \psi|^2- A_{2})\, \varphi \, \partial_t \psi\right) 
		- \tau^3\mu ^3  \int_{Q_h}\! |v|^2  {\textstyle\sum_{k}}\partial_{h,k}\left( A_{1,k}(\varphi^2|\partial_t \psi|^2- A_{2})\right) 
		\\
		&
		+\dfrac{\tau^3\mu ^3h^2}2 {\textstyle\sum_{k}} \int_{Q_{h,k}^-} |\partial_{h,k}^+v|^2 
			\partial_{h,k}^+ \left(  A_{1,k} (\varphi^2|\partial_t \psi|^2- A_{2})\right).
\end{align*}
But we have
\begin{align*}
	&\partial_t((\varphi^2|\partial_t \psi|^2  -A_2) \varphi \partial_t \psi) 
	\\
	&=
	~3\mu \varphi^3 |\partial_t \psi|^2 \left(  |\partial_t \psi|^2 - |\nabla \psi|^2 \right)
	+ \varphi^3 \partial_{tt}\psi  \left(  |\partial_t \psi|^2 - |\nabla \psi|^2 \right)
	+ 2 \varphi^3 |\partial_t \psi|^2 \partial_{tt}\psi  + \mathcal O_\mu(\tau h),
	\smallskip
\\
 	&{\textstyle\sum_{k}}\partial_{h,k}\left(A_{1,k}(\varphi^2|\partial_t \psi|^2-A_{2})\right) 
	\\
	&=
	 ~3\mu\, \varphi^3 |\nabla \psi|^2\left(|\partial_t \psi|^2 - |\nabla \psi|^2\right) 
		+ \varphi^3\Delta \psi \left(|\partial_t \psi|^2 - |\nabla \psi|^2\right) 
		- \varphi^3 \nabla \psi \cdot \nabla (|\nabla \psi|^2) 
		+ \mathcal O_\mu(\tau h),
		\smallskip
\\
	& \partial_{h,k}^+(A_{1,k}(\varphi^2|\partial_t \psi|^2-A_{2})) 
	= \partial_{x_k} (\varphi^3\,  \partial_{x_k}\psi\, (|\partial_t \psi|^2 - |\nabla\psi|^2))
		+ \mathcal O_\mu(\tau h)= \mathcal O_\mu(1), 
\end{align*}
so that we obtain
\begin{align*}
	I_{33}
		=~&
	3\tau^3\mu^4 \int_{Q_h}\! |v|^2 \varphi^3 (|\partial_t \psi|^2-|\nabla\psi|^2)^2
	+\tau^3\mu ^3 \int_{Q_h}\!|v|^2 \varphi^3 \left(\partial_{tt} \psi - \Delta \psi  \right)  \left(|\partial_t \psi|^2 - |\nabla \psi|^2 \right)
	\\
		&
	\hspace{-1cm}+ \tau^3\mu ^3 \int_{Q_h}\!|v|^2 \varphi^3\left( 2  \partial_{tt} \psi |\partial_t \psi|^2 +\nabla \psi \cdot \nabla (|\nabla \psi|^2) \right) 
	+\tau^3 \int_{Q_h}\!  \mathcal{O}_{\mu}(\tau h)|v|^2
	+\tau {\textstyle\sum_{k}} \int_{Q_{h,k}^+} \mathcal{O}_{\mu}(\tau h)|\partial_{h,k}^+v|^2 .
\end{align*}
\noindent{\it Final computation.}
Gathering all the terms, one can write
\begin{equation}
	\int_{Q_h}\! \mathscr{L}_{h,1}v \, \mathscr{L}_{h,2}v  = I_v + I_{\partial v}  + I_{\Gamma}+I_\tn{Tych}, 
\end{equation}
where $I_v = 	\ds\int_{Q_h}\! |v|^2 \mathcal F(\psi) $ contains all the terms in $|v|^2$
with
\begin{align*}
	\mathcal F(\psi) =~& 
	 \alpha_1  \tau^3\mu^3 \varphi^3(|\partial_t \psi|^2-|\nabla\psi|^2) (\partial_{tt} \psi - \Delta \psi) 
		+ \tau^3\mu ^3 \varphi^3\left( 2  \partial_{tt} \psi |\nabla \psi|^2 +\nabla \psi \cdot \nabla (|\nabla \psi|^2) \right) 
	\\
	& +2 \tau^3\mu^4 \varphi^3 (|\partial_t \psi|^2-|\nabla\psi|^2)^2+ \tau^3 \mathcal{O}_\mu(\tau h)+ \tau \mathcal{O}_\mu(1)~;
\end{align*}
$I_{\partial v}$ contains all the terms involving first-order derivatives of $v$:
\begin{align*}
	\lefteqn{
	I_{\partial v} = 
	2 \tau \mu ^2 \int_{Q_h}\!|\partial_t v|^2\varphi \, |\partial_t \psi|^2
	+2\tau\mu^2 \int_{Q_h}\! |\nabla_{h}  v\cdot \nabla \psi|^2 \varphi
	-4 \tau\mu^2 \int_{Q_h}\!  \partial_t v \, \partial_t\psi \, \varphi~\nabla_h v\cdot \nabla \psi 
	}
	\\
	&
	+ \tau \mu \int_{Q_h} |\partial_t v|^2 \varphi \left(2\partial_{tt} \psi - \alpha_1( \partial_{tt} \psi -\Delta \psi) \right)
	+ \tau \mu \sum_k \int_{Q_{h,k}^-} |\partial_{h,k}^+ v|^2 \varphi \left(\alpha_1 (\partial_{tt} \psi - \Delta \psi) + 2 \partial_{x_kx_k} \psi \right)
	\\
	& 
	+ 2  \tau \mu^2 \sum_k\left(  \int_{Q_{h,k}^-} |\partial_{h,k}^+ v|^2 \varphi |\partial_{x_k}\psi|^2 - \int_{Q_h} |\partial_{h,k} v|^2 \varphi |\partial_{x_k} \psi|^2 \right)
 + I_{\mathcal{O}_\mu}, 
\end{align*}
where $I_{\mathcal{O}_\mu}$ contains all the terms involving $\mathcal{O}_\mu$ terms (and a first-order derivative of $v$); 
%
\\
$I_{\Gamma}$ contains all the boundary terms:
$$
	I_{\Gamma}	=
		- ~ \tau\mu {\textstyle\sum_{k}}  \int_{\Sigma_{h,k}^+} |\partial_{h,k}^-  v|^2 (\varphi \, \partial_{x_k} \psi  +\mathcal{O}_{\mu} (\tau h))
	+~ \tau\mu {\textstyle\sum_{k}}  \int_{\Sigma_{h,k}^-} |\partial_{h,k}^+  v|^2 (\varphi \partial_{x_k} \psi  +\mathcal{O}_{\mu} (\tau h)) ;
$$
$I_\tn{Tych}$ contains the terms corresponding to the Tychonoff regularization:
\begin{align*}
	I_\tn{Tych} =& -  \dfrac{\tau\mu}2 {\textstyle\sum_{k}} \int_{Q_{h,k}^-}  |h\partial_{h,k}^+ \partial_t v|^2 \partial_{h,k}^+ A_{1,k}
	\\
	&+~ \frac{ \tau\mu}2 \int_{Q_h^-} |h\partial_{h,1}^+ \partial_{h,2}^+ v|^2 
	\left(\partial_{h,2}^+m_{h,1}^+((1+A_{0,1})   A_{1,2})+ \partial_{h,1}^+m_{h,2}^+((1+A_{0,2})   A_{1,1})\right).
\medskip
\end{align*}

\noindent {\it $\bullet$ Step 2. Bounding each term from below.}\\
\noindent {\it Step 2.1. Dealing with the $0$ order terms in $v$.} Since
 $ \nabla \psi \cdot \nabla (|\nabla \psi|^2) = 4 |\nabla \psi|^2 =  16|x-x_a|^2$, $\Delta \psi = 4$ and $\partial_{tt} \psi = -2 \beta $ and denoting 
$X= |\partial_t \psi|^2 - |\partial_x \psi|^2$, one can obtain
$$
	\mathcal F(\psi)=  \tau^3 \mu^3 \varphi^3 \underbrace{\Big( 2 \mu X^2 -2 \alpha_1 ( \beta + 2) X+ 16(1-\beta) |x-x_a|^2 \Big)}_{\mathcal G(\psi)} +  \tau^3 \mathcal{O}_\mu(\tau h)+ \tau \mathcal{O}_\mu(1),
$$
Since $x_a \notin \overline\Omega$, $ \inf_{(0,1)^2}  |x-x_a|^2$ is strictly positive and  we have
$$
	\mathcal G(\psi) \geq 2 \mu X^2 - 2 \alpha_1 ( \beta + 2) X + c, \quad \hbox{ with } c =  16(1-\beta) \inf_{(0,1)^2}  |x-x_a|^2>0.
$$
Thus, there exists $\mu_0\geq 1$ such that for $\mu = \mu_0$, $\mathcal G(\psi)>0$ uniformly.
Therefore, we get $c_0>0$ independent of $h$ such that 
\begin{equation}
	I_v 
	\geq 2 c_0 \tau^3 \int_{Q_h}\!  |v|^2 \varphi^3 - (\tau^3 \mathcal{O}_{\mu_0}(\tau h)+ \tau \mathcal{O}_{\mu_0}(1) ) \int_{Q_h}\!  |v|^2
	\label{order0}
	\geq  c_0 \tau^3 \int_{Q_h}\! |v|^2 - \tau^3 \mathcal{O}_{\mu_0}(\tau h) \int_{Q_h}\! |v|^2, 
\end{equation}
where the last line is obtained by bounding $\varphi$ from below by $1$ and by taking $\tau \geq \tau_0$ to absorb the $\mathcal{O}_{\mu_0}(1)$-term.
From now, we fix $\mu = \mu_0$ and we simply write $\mathcal{O}_\mu$ instead of $\mathcal{O}_{\mu_0}$.
\\
\noindent {\it Step 2.2. Dealing with the first-order derivatives.} The first line in $I_{\partial v}$ is positive as 
$$	
	\left|\int_{Q_h}\!  \partial_t v \, \partial_t\psi \, \varphi\, \nabla_h v\cdot \nabla \psi \right|	
	\leq	 \dfrac 12 \int_{Q_h}\!|\partial_t v|^2\varphi |\partial_t \psi|^2
	+\dfrac 12 \int_{Q_h}\! |\nabla_{h}  v\cdot \nabla \psi|^2 \varphi.
$$
The second line of $I_{\partial v}$ can be computed explicitly as $ \partial_{tt} \psi = -2\beta$, $\partial_{x_k x_k} \psi = 2$ and $\Delta\psi  = 4$: 
$$
	 2 \partial_{tt} \psi - \alpha_1( \partial_{tt} \psi -\Delta \psi )=  -4 \beta + 2 \alpha_1 (2+ \beta)  ; 
	\quad
	 \alpha_1 (\partial_{tt} \psi - \Delta \psi) + 2 \partial_{kk} \psi =- 2 \alpha_1 ( 2 + \beta) + 4.  
$$
Hence the choice $\alpha_1 = (\beta +1) / (\beta +2)$ makes each term strictly positive and equal to $2 (1- \beta)$ (recall $\beta \in (0,1)$), so that
\begin{multline*}
	\tau \mu \int_{Q_h} |\partial_t v|^2 \varphi \left(2\partial_{tt} \psi - \alpha_1( \partial_{tt} \psi -\Delta \psi) \right)
	+ \tau \mu \sum_k \int_{Q_{h,k}^-} |\partial_{h,k}^+ v|^2 \varphi \left(\alpha_1 (\partial_{tt} \psi - \Delta \psi) + 2 \partial_{kk} \psi \right)
	\\
	= 2 (1- \beta) \tau \mu \left(\int_{Q_h} |\partial_t v|^2+ \sum_k \int_{Q_{h,k}^-} |\partial_{h,k}^+ v|^2  \right).
\end{multline*}
We now remark that the third line of $I_{\partial v}$ is negligible. Indeed, writing $\partial_{h,k} v = m_{h,k}^-(\partial_{h,k}^+ v)$, one easily checks that 
$$
	\int_{Q_{h,k}^-} |\partial_{h,k}^+ v|^2 \varphi |\partial_{x_k}\psi|^2 - \int_{Q_h} |\partial_{h,k} v|^2 \varphi |\partial_{x_k} \psi|^2
	\geq - \int_{Q_{h,k}^-} \mathcal{O}_\mu(\tau h) |\partial_{h,k}^+ v|^2.
$$
Concerning the terms in $I_{\mathcal{O}_\mu}$, the only term that needs to be discussed are the ones coming from $I_{23}$: But using that $h^2 \Delta_{h,k}$ is a discrete operator with norm bounded by $8$, we get
\begin{multline*}
	\left|
	 \tau h^2 \int_{Q_{h}}  (\Delta_{h,1} v \, (\mathcal{O}_\mu(1) \partial_{h,2} v +\mathcal{O}_\mu(1) \partial_t v)  
	 +
	  \tau h^2 \int_{Q_{h}}  (\Delta_{h,2} v \, (\mathcal{O}_\mu(1) \partial_{h,1} v +\mathcal{O}_\mu(1) \partial_t v)
	  \right|
	\\
	  \leq C \left( \int_{Q_h} |\partial_t v|^2 + \sum_k \int_{Q_{h,k}^-} |\partial_{h,k}^+ v|^2 + \tau^{2} \int_{Q_h} |v|^2\right).
\end{multline*}
Combining these estimates, for $\tau$ large enough, we obtain constants $c_1>0$, $C_0>0$ such that 
\begin{multline}
	I_{\partial v} 
		\geq ~c_1 \tau \int_{Q_h} | \partial_t v|^2 
			+c_1 \tau  {\textstyle\sum_{k}} \int_{Q_{h,k}^-} |\partial_{h,k}^+ v|^2 
		\\
		 -\tau  \int_{Q_h}\mathcal O_{\mu}(\tau h)  |\partial_t v|^2 
		-\tau{\textstyle\sum_{k}} \int_{Q_{h,k}^-}  \mathcal{O}_{\mu}(\tau h) |\partial_{h,k}^+v|^2
		- C_0 \tau^2 \int_{Q_h} |v|^2. \label{order1}
\end{multline}

\noindent {\it Step 2.3. The boundary terms.} Since $\min_{(-T, T) \times \Omega} \{\varphi \partial_{x_k} \psi \}  >0$ (recall $x_a \notin \overline\Omega$), then there exists $\varepsilon_1 >0$ such that taking $\tau h \leq \varepsilon_1$, 
$$
	|\mathcal{O}_{\mu}(\tau h)| \leq \min_{(t,x) \in (-T, T) \times \Omega}\left\{ \varphi(t,x) \partial_{x_k} \psi(t,x)\right\}.
$$
so there exists $C>0$ independent of $\tau$ and $h$ such that
\begin{align}
	I_{\Gamma}	
		\geq 
		& - 2 \tau\mu {\textstyle\sum_{k}}  \int_{\Sigma_{h,k}^+} |\partial_{h,k}^-  v|^2 \varphi \partial_{x_k} \psi
		\geq 
		 -  C \tau {\textstyle\sum_{k}}  \int_{\Sigma_{h,k}^+} |\partial_{h,k}^-  v|^2.
		 \label{boundary}
\end{align}
\noindent {\it Step 2.4. The Tychonoff regularization.} We have
$ \partial_{h,k}^+ A_{1,k} = \mu \, \varphi |\partial_{x_k} \psi|^2 + \varphi \, \partial_{x_kx_k} \psi + \mathcal{O}_{\mu}(\tau h)
= \mathcal{O}_{\mu}(1) $ and 
$\partial_{h,k}^+m_{h,\ell}^+((1+A_{0,\ell})   A_{1,k}) =  \mu\,  \varphi |\partial_{x_k} \psi|^2 + \varphi\, \partial_{x_kx_k} \psi + \mathcal{O}_{\mu}(\tau h)$.
Thus, for $\tau h$ small enough, i.e. $\tau h \leq \varepsilon_2$ for some $\varepsilon_2 \in (0,\varepsilon_1)$, 
$$
	\left(\partial_{h,2}^+m_{h,1}^+((1+A_{0,1})   A_{1,2})+ \partial_{h,1}^+m_{h,2}^+((1+A_{0,2})   A_{1,1})\right) > 0, 
$$
and the term involving $\partial_{h,1}^+\partial_{h,2}^+ v$ is positive, whereas the other term in $I_{Tych}$ is negative.
We bound it directly and get a constant $C>0$ independent of $\tau$ and $h$ such that
\begin{equation} \label{tych}
	I_\tn{Tych} 	 \geq -  ~C \tau  {\textstyle\sum_{k}} \int_{Q_{h,k}^-}  |h\partial_{h,k}^+ \partial_t v|^2.
\medskip
\end{equation}

\noindent {\it $\bullet$ Step 3. End of the proof of Proposition \ref{PropDecompo}.} 
Collecting the results \eqref{order0}--\eqref{tych} of Step 2, we have proved that for $\tau\geq \tau_0$ and $\tau h \leq \varepsilon_2$,
\begin{eqnarray*}
	\int_{Q_h}\! \mathscr{L}_{h,1}v \, \mathscr{L}_{h,2}v
	 &\geq&~ c_0 \tau^3 \int_{Q_h}\! |v|^2 
	 + c_1 \tau \int_{Q_h}\!| \partial_t v|^2 
		+c_1 \tau  {\textstyle\sum_{k}} \int_{Q_{h,k}^-} |\partial_{h,k}^+ v|^2 - C_0 \tau^2 \int_{Q_h} |v|^2\\
	&& -  C \tau {\textstyle\sum_{k}}  \int_{\Sigma_{h,k}^+} |\partial_{h,k}^-  v|^2
	  	-  C \tau  {\textstyle\sum_{k}} \int_{Q_{h,k}^-}  |h\partial_{h,k}^+ \partial_t v|^2\\
		&& - \tau^3  \int_{Q_h} \mathcal{O}_{\mu}(\tau h) |v|^2
		-	\tau  \int_{Q_h} \mathcal{O}_{\mu}(\tau h) |\partial_t v|^2
		-\tau  {\textstyle\sum_{k}} \int_{Q_{h,k}^-} \mathcal O_{\mu}(\tau h)|\partial_{h,k}^+v|^2 . 
\end{eqnarray*}
Therefore, taking $\tau$ large enough so that  $c_0 \tau^3 > 2 C_0 \tau^2$ and $\tau h$ small enough such that 
$
	|\mathcal{O}_{\mu} (\tau h)| \leq\min\left\{ c_0   , c_1 , \varepsilon_2 \right\}, 
$
which defines $\varepsilon_0>0$, we obtain, for some constant $C_1>0$,
\begin{multline*}
	\tau \int_{Q_h}\!| \partial_t v|^2
	+  \tau  {\textstyle\sum_{k}} \int_{Q_{h,k}^-} |\partial_{h,k}^+ v|^2 
	+	\tau^3 \int_{Q_h}\! |v|^2
\\
	\leq 
	C_1 \int_{Q_h}\! \mathscr{L}_{h,1}v \, \mathscr{L}_{h,2}v
	+
	C_1 \tau {\textstyle\sum_{k}}  \int_{\Sigma_{h,k}^+} |\partial_{h,k}^-  v|^2 
	+ 
	C_1 \tau {\textstyle\sum_{k}} \int_{Q_{h,k}^-}  |h\partial_{h,k}^+ \partial_t v|^2.
\end{multline*}
From \eqref{double}, there exists $C_2 >0$ such that
\begin{multline}
	 \tau \int_{Q_h}\!| \partial_t v|^2
	+  \tau  {\textstyle\sum_{k}} \int_{Q_{h,k}^-} |\partial_{h,k}^+ v|^2 
	+	\tau^3 \int_{Q_h}\! |v|^2
	+ \int_{Q_h}\! |\mathscr{L}_{1,h} v|^2 
	\\
	\leq 
	C_2 \int_{Q_h}\! |\mathscr{L}_{h} v|^2
	+
	C_2 \int_{Q_h}\! |\mathscr{R}_h v|^2 
	\label{decompo-00}  
	+
	C_2 \tau {\textstyle\sum_{k}}  \int_{\Sigma_{h,k}^+} |\partial_{h,k}^-  v|^2 
	+ 
	C_2 \tau {\textstyle\sum_{k}} \int_{Q_{h,k}^-}  |h\partial_{h,k}^+ \partial_t v|^2.
	\end{multline}
But
$$
	 \int_{Q_h}\! |\mathscr{R}_h v|^2 
	 \leq C \tau^2  \int_{Q_h}\! |v|^2,
$$
which can also be absorbed by the left hand side of \eqref{decompo-00} by taking $\tau$ large enough, thus yielding to \eqref{decompo}.
\end{proof}

\section{Proof of an elliptic regularity result }\label{Appendix-Elliptic}
	\begin{proof}[Proof of Lemma \ref{Lem-Elliptic-Reg}]
		Multiplying the equation \eqref{Elliptic-Eq} by $w_h$, using the discrete Poincar\'e's inequality, one easily obtains that 
		\begin{equation}
			\label{Est-H1-wh}
			w_h \in H^1_{0,h}(\Omega_h) \quad \hbox{ with } \norm{w_h}_{H^1_{0,h} (\Omega_h)} \leq C \norm{g_h}_{L^2_h(\Omega_h)}, 
		\end{equation}
		for some constant $C = C(m)>0$ independent of $h>0$. Accordingly, replacing $g_h$ by $g_h - q_h w_h$, we are reduced to the case $q_h = 0$, that we assume from now. 
		
		Since $\Omega_h =( h \Z)^2 \cap (0,1)^2$, we first propose to extend $w_h$ \emph{a priori} defined on the discrete domain $\Omega_h$ to $\Omega_{ext, h} = (h \Z)^2 \cap (-1, 2)^2$ as follows. First, for $x_h \in \{(0,0), (1,0), (1,1), (0,1)\}$, we set $\tilde w_h(x_h) = 0$. Then, for $x_h  = (x_{h,1}, x_{h,2}) \in [0,1]\times (-1,2) \cap \Omega_{\tn{ext},h}$, we set $\tilde w_h (x_h)=- w_h(x_{h,1}, -x_{h,2})$ for $x_{h,2} \in (-1,0)$ and $\tilde w_h (x_h) = - w_h(x_{h,1}, 1 - (x_{2,h} - 1))$ for $x_{h,2} \in (1,2)$. This defines $\tilde w_h$ on $[0,1]\times (-1,2) \cap \Omega_{\tn{ext},h}$. We then extend it for $x_h = (x_{1,h}, x_{2,h}) \in \Omega_{\tn{ext},h}$ by setting $\tilde w_h(x_h) = -\tilde w_{h}(-x_{h,1}, x_{2,h})$ for $x_{h,1} \in (-1,0)$ and $\tilde w_h(x_h) = -\tilde w_h(1 - (x_{h,1} - 1) , x_{h,2})$ for $x_{h,1} \in (1,2)$. We do a similar extension $\tilde g_h$ of $g_h$ on $\Omega_{\tn{ext},h}$ taking care of choosing $\tilde g_h = 0$ on $\partial \Omega_h \cup \{(0,0), (1,0), (1,1), (0,1) \} $. 
	
		We thus have constructed a solution $\tilde w_h$ of 
		\begin{equation}
			\label{Elliptic-Eq-Extended}
			-\Delta_h \tilde w_h = \tilde g_h \hbox{ in } \Omega_{\tn{ext},h} \quad \hbox{ and } \quad \tilde w_h = 0 \hbox{ on } \partial \Omega_{\tn{ext},h}.
		\end{equation}
		We then choose a function $\chi \in C^{\infty}_c ( (-1,2)^2)$ such that $\chi = 1$ on $[0,1]^2$ and we multiply \eqref{Elliptic-Eq-Extended} by $-\chi_h \Delta_{1,h} \tilde w_h$ with $\chi_h = \rh(\chi)$: After some integrations by parts where all the boundary terms vanish due to the choice of $\chi$, we obtain:
		\begin{gather}
			\int_{\Omega_{\tn{ext},h}} \chi_h |\Delta_{h,1} \tilde w_h|^2 + \int_{\Omega_{\tn{ext},h}} m_{h,1}^+ m_{h,2}^+ \chi_h |\partial_{h,1}^+ \partial_{h,2}^+ \tilde w_h|^2 
						\label{Identity-Delta1}
			\\
			= 
			- \int_{\Omega_{\tn{ext},h}} \hspace{-3ex} \chi_h \tilde g_h \Delta_{h,1} \tilde w_h + 
			\int_{\Omega_{\tn{ext},h}} \hspace{-3ex}  \partial_{h,2}^+ \chi_h \partial_{h,2}^+ \tilde w_h m_{h,2}^+ \Delta_{h,1} w_h 
			- \int_{\Omega_{\tn{ext},h}}  \hspace{-3ex} \partial_{h,1}^+ m_{h,2}^+ \chi_h m_{h,1}^+ \partial_{h,2}^+ w_h \partial_{h,1}^+ \partial_{h,2}^+ \tilde w_h.
			\notag
		\end{gather}
		Of course, since $\chi = 1$ on $[0,1]^2$, the left hand-side of \eqref{Identity-Delta1} is bounded from below by 
		$$ 
			\norm{\Delta_{h,1} w_h}_{L^2_h(\overline{\Omega_h})}^2 + \norm{\partial_{h,1}^+ \partial_{h,2}^+ w_h}_{L^2_h(\overline{\Omega_h})}^2.
		$$
		 On the other hand, using that $\tilde w_h$ and $\tilde g_h$ are symmetric extensions of $w_h$ and $g_h$, the right hand-side of \eqref{Identity-Delta1} is bounded from above by 
		$$
			C \left(\norm{ g_h}_{L^2_h(\Omega_h)} + \norm{ w_h}_{H^1_{0,h}(\Omega_h)}\right) \left(  \norm{\Delta_{h,1} w_h}_{L^2_h(\overline{\Omega_h})} + \norm{\partial_{h,1}^+ \partial_{h,2}^+ w_h}_{L^2_h(\overline{\Omega_h})}\right),
		$$
		for some constant $C$ independent of $h>0$. We thus obtain 
		$$
			\norm{\Delta_{h,1} w_h}_{L^2_h(\overline{\Omega_h})} + \norm{\partial_{h,1}^+ \partial_{h,2}^+ w_h}_{L^2_h(\overline{\Omega_h})} \leq C \left(\norm{ g_h}_{L^2_h(\Omega_h)} +  \norm{ w_h}_{H^1_{0,h}(\Omega_h)} \right),
		$$
		which, together with \eqref{Est-H1-wh} and $-\Delta_{h,2} w_h = (g_h - q_h w_h) + \Delta_{h,1} w_h$ , yields \eqref{Elliptic-Reg}.
	\end{proof}

\end{document}